\documentclass[a4paper, reqno, 11pt]{amsart}
\usepackage[utf8]{inputenc}

\usepackage{amsthm}
\usepackage{amssymb}
\usepackage{exscale}
\usepackage[mathscr]{euscript}
\usepackage{url}
\usepackage[all,ps,tips,tpic]{xy}
\usepackage{graphicx}
\usepackage{color}

\newenvironment{altenumerate}
   {\begin{list}
      {\textup{(\theenumi)} }
      {\usecounter{enumi}
       \setlength{\labelwidth}{0pt}
       \setlength{\labelsep}{2pt}
       \setlength{\leftmargin}{0pt}
       \setlength{\itemsep}{\the\smallskipamount}
       \renewcommand{\theenumi}{\roman{enumi}}
      }}
   {\end{list}}

\newtheorem{lem}{Lemma}[section]

\newtheorem{definition}[lem]{Definition}

\newtheorem{cor}[lem]{Corollary}
\newtheorem{thm}[lem]{Theorem}
\newtheorem{prop}[lem]{Proposition}

\theoremstyle{remark}
\newtheorem{rem}[lem]{Remark}
\newtheorem{example}[lem]{Example}

\DeclareMathOperator{\Hom}{Hom}
\DeclareMathOperator{\Ext}{Ext}

\DeclareMathOperator{\Tor}{Tor}

\DeclareMathOperator{\im}{im}
\DeclareMathOperator{\coker}{coker}

\DeclareMathOperator{\diag}{diag}
\DeclareMathOperator{\Spa}{Spa}
\DeclareMathOperator{\Spec}{Spec}
\DeclareMathOperator{\Fil}{Fil}
\DeclareMathOperator{\gr}{gr}
\DeclareMathOperator{\Gal}{Gal}

\def\an{\mathrm{an}}
\def\et{\mathrm{\acute{e}t}}
\def\fet{\mathrm{f\acute{e}t}}
\def\pro{\mathrm{pro}}
\def\proet{\mathrm{pro\acute{e}t}}
\def\proetqc{\mathrm{pro\acute{e}tqc}}
\def\profet{\mathrm{prof\acute{e}t}}
\def\fsets{\mathrm{fsets}}
\def\pfsets{\mathrm{pfsets}}
\def\cont{\mathrm{cont}}
\def\pt{\mathrm{pt}}
\def\inf{\mathrm{inf}}

\def\dR{\mathrm{dR}}
\def\id{\mathrm{id}}

\newcommand{\N}{\mathbb{N}}
\newcommand{\Z}{\mathbb{Z}}
\newcommand{\F}{\mathbb{F}}
\newcommand{\Q}{\mathbb{Q}}
\newcommand{\R}{\mathbb{R}}
\newcommand{\A}{\mathbb{A}}
\newcommand{\B}{\mathbb{B}}
\newcommand{\OO}{\mathcal{O}}
\newcommand{\mm}{\mathfrak{m}}

\newcommand{\GL}{\mathrm{GL}}
\newcommand{\PGL}{\mathrm{PGL}}

\newdir{ >}{{}*!/-5pt/@{>}}
\SelectTips{cm}{10}

\begin{document}
\title[$p$-adic Hodge theory]{$p$-adic Hodge theory for rigid-analytic varieties}
\author{Peter Scholze}
\begin{abstract}
We give proofs of de Rham comparison isomorphisms for rigid-analytic varieties, with coefficients and in families. This relies on the theory of perfectoid spaces. Another new ingredient is the pro-\'{e}tale site, which makes all constructions completely functorial.
\end{abstract}

\date{\today}
\maketitle
\tableofcontents
\pagebreak

\section{Introduction}

This paper starts to investigate to what extent $p$-adic comparison theorems stay true for rigid-analytic varieties. Up to now, such comparison isomorphisms were mostly studied for schemes over $p$-adic fields, but we intend to show here that the whole theory extends naturally to rigid-analytic varieties over $p$-adic fields. This is of course in analogy with classical Hodge theory, which most naturally is formulated in terms of complex-analytic spaces.

Several difficulties have to be overcome to make this work. The first is that finiteness of $p$-adic \'{e}tale cohomology is not known for rigid-analytic varieties over $p$-adic fields. In fact, it is false if one does not make a restriction to the proper case. However, our first theorem is that for proper smooth rigid-analytic varieties, finiteness of $p$-adic \'{e}tale cohomology holds.

\begin{thm}\label{Thm1} Let $K$ be a complete algebraically closed extension of $\Q_p$, let $X/K$ be a proper smooth rigid-analytic variety, and let $\mathbb{L}$ be an $\mathbb{F}_p$-local system on $X_\et$. Then $H^i(X_\et,\mathbb{L})$ is a finite-dimensional $\mathbb{F}_p$-vector space for all $i\geq 0$, which vanishes for $i>2\dim X$.
\end{thm}

The properness assumption is crucial here; the smoothness assumption is probably unnecessary, and an artefact of the proof. We note that it would be interesting to prove Poincar\'{e} duality in this setup.

Let us first explain our proof of this theorem. We build upon Faltings's theory of almost \'{e}tale extensions, amplified by the theory of perfectoid spaces. One important difficulty in $p$-adic Hodge theory as compared to classical Hodge theory is that the local structure of rigid-analytic varieties is very complicated; small open subsets still have a large \'{e}tale fundamental group. We introduce the pro-\'{e}tale site $X_\proet$ whose open subsets are roughly of the form $V\rightarrow U\rightarrow X$, where $U\rightarrow X$ is some \'{e}tale morphism, and $V\rightarrow U$ is an inverse limit of finite \'{e}tale maps. Then the local structure of $X$ in the pro-\'{e}tale topology is simpler, namely, it is locally perfectoid. This amounts to extracting lots of $p$-power roots of units in the tower $V\rightarrow U$. We note that the idea to extract many $p$-power roots is common to all known proofs of comparison theorems in $p$-adic Hodge theory.

The following result gives further justification to the definition of pro-\'etale site.

\begin{thm}\label{ThmKpi1} Let $X$ be a connected affinoid rigid-analytic variety over $K$. Then $X$ is a $K(\pi,1)$ for $p$-torsion coefficients, i.e. for all $p$-torsion local systems $\mathbb{L}$ on $X$, the natural map
\[
H^i_\cont(\pi_1(X,x),\mathbb{L}_x)\to H^i(X_\et,\mathbb{L})
\]
is an isomorphism. Here, $x\in X(K)$ is a base point, and $\pi_1(X,x)$ denotes the profinite \'etale fundamental group.
\end{thm}

We note that we assume only that $X$ is affinoid;  no smallness or nonsingularity hypothesis is necessary for this result. This theorem implies that $X$ is 'locally contractible' in the pro-\'etale site, at least for $p$-torsion local systems.

Now, on affinoid perfectoid subsets $U$, one knows that $H^i(U_\et,\OO_X^+/p)$ is almost zero for $i>0$, where $\OO_X^+\subset \OO_X$ is the subsheaf of functions of absolute value $\leq 1$ everywhere. This should be seen as the basic finiteness result, and is related to Faltings's almost purity theorem. Starting from this and a suitable cover of $X$ by affinoid perfectoid subsets in $X_\proet$, one can deduce that $H^i(X_\et,\OO_X^+/p)$ is almost finitely generated over $\OO_K$. At this point, one uses that $X$ is proper, and in fact the proof of this finiteness result is inspired by the proof of finiteness of coherent cohomology of proper rigid-analytic varieties, as given by Kiehl, \cite{KiehlFiniteness}. Then one deduces finiteness results for the $\mathbb{F}_p$-cohomology by using a variant of the Artin-Schreier sequence
\[
0\rightarrow \mathbb{F}_p\rightarrow \OO_X^+/p\rightarrow \OO_X^+/p\rightarrow 0\ .
\]
In order to make this argument precise, one needs to analyze more closely the category of almost finitely generated $\OO_K$-modules, which we do in Section \ref{AlmostMath}, formalizing the proof of \S3,  Theorem 8, of Faltings's paper \cite{FaltingsAlmostEtale}. In fact, the proof shows at the same time the following result, which is closely related to \S3, Theorem 8, of \cite{FaltingsAlmostEtale}.

\begin{thm}\label{Thm2} In the situation of Theorem \ref{Thm1}, there is an almost isomorphism of $\OO_K$-modules for all $i\geq 0$,
\[
H^i(X_\et,\mathbb{L})\otimes \OO_K/p \rightarrow H^i(X_\et,\mathbb{L}\otimes \OO_X^+/p)\ .
\]
More generally, assume that $f:X\rightarrow Y$ is a proper smooth morphism of rigid-analytic varieties over $K$, and $\mathbb{L}$ is an $\mathbb{F}_p$-local system on $X_\et$. Then there is an almost isomorphism for all $i\geq 0$,
\[
(R^if_{\et\ast} \mathbb{L})\otimes \OO_Y^+/p\rightarrow R^if_{\et\ast}(\mathbb{L}\otimes \OO_X^+/p)\ .
\]
\end{thm}

\begin{rem} The relative case was already considered in an appendix to \cite{FaltingsAlmostEtale}: Under the assumption that $X$, $Y$ and $f$ are algebraic and have suitable integral models, this is \S6, Theorem 6, of \cite{FaltingsAlmostEtale}. In our approach, it is a direct corollary of the absolute version.
\end{rem}

In a sense, this can be regarded as a primitive version of a comparison theorem. Although it should be possible to deduce (log-)crystalline comparison theorems from it, we do only the de Rham case here. For this, we introduce sheaves on $X_\proet$, which we call period sheaves, as their values on pro-\'{e}tale covers of $X$ give period rings. Among them is the sheaf $\B_\dR^+$, which is the relative version of Fontaine's ring $B_\dR^+$. Let $\mathbb{L}$ be lisse $\Z_p$-sheaf on $X$. In our setup, we can define it as a locally free $\hat{\mathbb{Z}}_p$-module on $X_\proet$, where $\hat{\mathbb{Z}}_p = \varprojlim \Z/p^n\Z$ as sheaves on $X_\proet$. Then $\mathbb{L}$ gives rise to a $\B_\dR^+$-local system $\mathbb{M} = \mathbb{L}\otimes_{\hat{\mathbb{Z}}_p} \B_\dR^+$ on $X_\proet$, and it is a formal consequence of Theorem \ref{Thm2} that
\begin{equation}\label{NotProvedByBeilinson}
H^i(X_\et,\mathbb{L})\otimes_{\Z_p} B_\dR^+\cong H^i(X_\proet,\mathbb{M})\ .
\end{equation}

We want to compare this to de Rham cohomology. For this, we first relate filtered modules with integrable connection to $\B_\dR^+$-local systems.

\begin{thm}\label{Thm3} Let $X$ be a smooth rigid-analytic variety over $k$, where $k$ is a complete discretely valued nonarchimedean extension of $\Q_p$ with perfect residue field. Then there is a fully faithful functor from the category of filtered $\OO_X$-modules with an integrable connection satisfying Griffiths transversality, to the category of $\B_\dR^+$-local systems.
\end{thm}

The proof makes use of the period rings introduced in Brinon's book \cite{BrinonRepresentations}, and relies on some of the computations of Galois cohomology groups done there. We say that a lisse $\Z_p$-sheaf $\mathbb{L}$ is de Rham if the associated $\B_\dR^+$-local system $\mathbb{M}$ lies in the essential image of this functor.

Let us remark at this point that the form of this correspondence indicates that the Rapoport -- Zink conjecture on existence of local systems on period domains, cf. \cite[]{RapoportZinkPeriodDomains}, is wrong if the cocharacter $\mu$ is not minuscule. Indeed, in that case they are asking for a crystalline, and thus de Rham, local system on (an open subspace of) the period domain, whose associated filtered module with integrable connection does not satisfy Griffiths transversality. However, the $p$-adic Hodge theory formalism does not allow for an extension of Theorem \ref{Thm3} beyond the situations where Griffiths transversality is satisfied.

We have the following comparison result.

\begin{thm}\label{Thm4} Let $k$ be a discretely valued complete nonarchimedean extension of $\Q_p$ with perfect residue field $\kappa$, and algebraic closure $\bar{k}$, and let $X$ be a proper smooth rigid-analytic variety over $k$. For any lisse $\Z_p$-sheaf $\mathbb{L}$ on $X$ with associated $\B_\dR^+$-local system $\mathbb{M}$, we have a $\Gal(\bar{k}/k)$-equivariant isomorphism
\[
H^i(X_{\bar{k}},\mathbb{L})\otimes_{\Z_p} B_\dR^+\cong H^i(X_{\bar{k}},\mathbb{M})\ .
\]
If $\mathbb{L}$ is de Rham, with associated filtered module with integrable connection $(\mathcal{E},\nabla,\Fil^\bullet)$, then the Hodge-de Rham spectral sequence
\[
H^{i-j,j}_{\mathrm{Hodge}}(X,\mathcal{E})\Rightarrow H^i_\dR(X,\mathcal{E})
\]
degenerates. Moreover, $H^i(X_{\bar{k}},\mathbb{L})$ is a de Rham representation of $\Gal(\bar{k}/k)$ with associated filtered $k$-vector space $H^i_\dR(X,\mathcal{E})$. In particular, there is also a $\Gal(\bar{k}/k)$-equivariant isomorphism
\[
H^i(X_{\bar{k}},\mathbb{L})\otimes_{\Z_p} \hat{\bar{k}}\cong \bigoplus_j H^{i-j,j}_{\mathrm{Hodge}}(X,\mathcal{E})\otimes_k \hat{\bar{k}}(-j)\ .
\]
\end{thm}

\begin{rem} We define the Hodge cohomology as the hypercohomology of the associated gradeds of the de Rham complex of $\mathcal{E}$, with the filtration induced from $\Fil^\bullet$.
\end{rem}

In particular, we get the following corollary, which answers a question of Tate, \cite{TatePDivGroups}, Remark on p.180.

\begin{cor} For any proper smooth rigid-analytic variety $X$ over $k$, the Hodge-de Rham spectral sequence
\[
H^i(X,\Omega_X^j)\Rightarrow H^{i+j}_\dR(X)
\]
degenerates, there is a Hodge-Tate decomposition
\[
H^i(X_{\bar{k},\et},\Q_p)\otimes_{\Q_p} \hat{\bar{k}}\cong \bigoplus_{j=0}^i H^{i-j}(X,\Omega_X^j)\otimes_k \hat{\bar{k}}(-j)\ ,
\]
and the $p$-adic \'{e}tale cohomology $H^i(X_\et,\Q_p)$ is de Rham, with associated filtered $k$-vector space $H^i_\dR(X)$.
\end{cor}

Interestingly, no 'K\"ahler' assumption is necessary for this result in the $p$-adic case as compared to classical Hodge theory. In particular, one gets degeneration for all proper smooth varieties over fields of characteristic $0$ without using Chow's lemma.

Examples of non-algebraic proper smooth rigid-analytic varieties can be constructed by starting from a proper smooth variety in characteristic $p$, and taking a formal, non-algebraizable, lift to characteristic $0$. This can be done for example for abelian varieties or K3 surfaces. More generally, there is the theory of abeloid varieties, which are 'non-algebraic abelian rigid-analytic varieties', roughly, cf. \cite{LuetkebohmertAbeloid}. Theorem \ref{Thm4} also has the following consequence, which was conjectured by Schneider, cf. \cite{SchneiderLocalSystems}, p.633.

\begin{cor} Let $k$ be a finite extension of $\Q_p$, let $X=\Omega^n_k$ be Drinfeld's upper half-space, which is the complement of all $k$-rational hyperplanes in $\mathbb{P}^{n-1}_k$, and let $\Gamma\subset \PGL_n(k)$ be a discrete cocompact subgroup acting without fixed points on $\Omega^n_k$. One gets the quotient $X_\Gamma = X/\Gamma$, which is a proper smooth rigid-analytic variety over $k$. Let $M$ be a representation of $\Gamma$ on a finite-dimensional $k$-vector space, such that $M$ admits a $\Gamma$-invariant $\OO_k$-lattice. It gives rise to a local system $\mathcal{M}_\Gamma$ of $k$-vector spaces on $X_\Gamma$. Then the twisted Hodge-de Rham spectral sequence
\[
H^i(X_\Gamma,\Omega_{X_\Gamma}^j\otimes \mathcal{M}_\Gamma)\Rightarrow H^{i+j}_\dR(X_\Gamma,\OO_{X_\Gamma}\otimes \mathcal{M}_\Gamma)
\]
degenerates.
\end{cor}

The proof of Theorem \ref{Thm4} follows the ideas of Andreatta and Iovita, \cite{AndreattaIovita}, in the crystalline case. One uses a version of the Poincar\'{e} lemma, which says here that one has an exact sequence of sheaves over $X_\proet$,
\[
0\rightarrow \B_\dR^+\rightarrow \OO\B_\dR^+\buildrel\nabla\over\rightarrow \OO\B_\dR^+\otimes_{\OO_X}\Omega_X^1\buildrel\nabla\over\rightarrow\ldots\ ,
\]
where we use slightly nonstandard notation. In \cite{BrinonRepresentations} and \cite{AndreattaIovita}, $\B_\dR^+$ would be called $\B_\dR^{\nabla +}$, and $\OO\B_\dR^+$ would be called $\B_\dR^+$. This choice of notation is used because many sources do not consider sheaves like $\OO\B_\dR^+$, and agree with our notation in writing $\B_\dR^+$ for the sheaf that is sometimes called $\B_\dR^{\nabla +}$. We hope that the reader will find the notation not too confusing.

Given this Poincar\'{e} lemma, it only remains to calculate the cohomology of $\OO\B_\dR^+$, which turns out to be given by coherent cohomology through some explicit calculation. This finishes the proof of Theorem \ref{Thm4}. We note that this proof is direct: All desired isomorphisms are proved by a direct argument, and not by producing a map between two cohomology theories and then proving that it has to be an isomorphism by abstract arguments. In fact, such arguments would not be available for us, as results like Poincar\'{e} duality are not known for the $p$-adic \'{e}tale cohomology of rigid-analytic varieties over $p$-adic fields. It also turns out that our methods are flexible enough to handle the relative case, and our results imply directly the corresponding results for proper smooth algebraic varieties, by suitable GAGA results. This gives for example the following result.

\begin{thm}\label{Thm5} Let $k$ be a discretely valued complete nonarchimedean extension of $\Q_p$ with perfect residue field $\kappa$, and let $f:X\rightarrow Y$ be a proper smooth morphism of smooth rigid-analytic varieties over $k$. Let $\mathbb{L}$ be a lisse $\Z_p$-sheaf on $X$ which is de Rham, with associated filtered module with integrable connection $(\mathcal{E},\nabla,\Fil^\bullet)$. Assume that $R^if_{\proet\ast} \mathbb{L}$ is a lisse $\Z_p$-sheaf on $Y$; this holds true, for example, if the situation comes as the analytification of algebraic objects.

Then $R^if_{\proet\ast} \mathbb{L}$ is de Rham, with associated filtered module with integrable connection given by $R^if_{\dR\ast} (\mathcal{E},\nabla,\Fil^\bullet)$.
\end{thm}

We note that we make use of the full strength of the theory of perfectoid spaces, \cite{ScholzePerfectoidSpaces1}. Apart from this, our argument is rather elementary and self-contained, making use of little more than basic rigid-analytic geometry, which we reformulate in terms of adic spaces, and basic almost mathematics. In particular, we work entirely on the generic fibre. This eliminates in particular any assumptions on the reduction type of our variety, and we do not need any version of de Jong's alterations, neither do we need log structures. The introduction of the pro-\'{e}tale site makes all constructions functorial, and it also eliminates the need to talk about formal projective or formal inductive systems of sheaves, as was done e.g. in \cite{FaltingsAlmostEtale}, \cite{AndreattaIovita}: All period sheaves are honest sheaves on the pro-\'{e}tale site.

Recently, a different proof of the de Rham comparison theorem for algebraic varieties was given by Beilinson, \cite{Beilinson}. Apart from the idea of extracting many $p$-power roots to kill certain cohomology groups, we see no direct relation between the two approaches. We note that adapting Beilinson's approach to the rigid-analytic case seems to require at least Equation \eqref{NotProvedByBeilinson} as input: Modulo some details, the sheaf $\B_\dR^+$ will appear as the sheaf of constants of the derived de Rham complex, not just the constant sheaf $B_\dR^+$ as in Beilinson's case. This happens because there are no bounded algebraic functions on $p$-adic schemes, but of course there are such functions on affinoid subsets. Also, it looks difficult to get results with coefficients using Beilinson's approach, as the formulation of the condition for a lisse $\Z_p$-sheaf to be de Rham seems to be inherently a rigid-analytic condition.

Let us make some remarks about the content of the different sections. Some useful statements are collected in Section \ref{Miscellany}, in particular concerning comparison with the algebraic theory. In Section \ref{AlmostMath}, we prove a classification result for almost finitely generated $\OO_K$-modules, for nonarchimedean fields $K$ whose valuation is nondiscrete. In Section \ref{ProEtaleSection}, we introduce the pro-\'{e}tale site and establish its basic properties. The most important features are that inverse limits of sheaves are often well-behaved on this site, i.e. higher inverse limits vanish, and that it gives a natural interpretation of continuous group cohomology, which may be of independent interest. Moreover, going from the \'{e}tale to the pro-\'{e}tale site does not change the cohomology. In Section \ref{StructureSheafSection}, we introduce structure sheaves on the pro-\'{e}tale site and prove that they are well-behaved on a basis for the pro-\'{e}tale topology, namely on the affinoid perfectoid subsets. This relies on the full strength of the theory of perfectoid spaces. In Section \ref{FinitenessSection}, we use this description to prove Theorem \ref{Thm1} and Theorem \ref{Thm2} as indicated earlier. In Section \ref{PeriodSheavesSection}, we introduce some period sheaves on the pro-\'{e}tale topology, and describe them explicitly. Given the results of Section \ref{StructureSheafSection}, this is rather elementary and explicit. In Section \ref{DeRhamSection}, we use these period sheaves to prove Theorem \ref{Thm3}, and parts of Theorem \ref{Thm4}. Finally, in Section \ref{ApplicationsSection}, we finish the proofs of Theorem \ref{Thm4} and Theorem \ref{Thm5}.

{\bf Acknowledgments.} Some of these results were announced in March 2011 at a conference at the IAS in Princeton, and the author wants to thank the organizers for the invitation to speak there. He would also like to thank Arthur Ogus and Martin Olsson for the invitation to speak about these results at Berkeley in September 2011. The author would like to thank Lorenzo Ramero for discussions related to the results of Section \ref{AlmostMath}, which were inspired by reading Section 9.3 of \cite{GabberRamero2}, and overlap with it to some extent. Also, he would like to thank Kiran Kedlaya for discussions related to Theorem \ref{Thm3}, in particular for proposing the alternate characterization of $\mathbb{M}$ in Proposition \ref{AlternateCharacterization}. The generalization of these results to general proper smooth rigid-analytic varieties was prompted by a question of Davesh Maulik, whom the author wishes to thank. Further, he wants to thank Ahmed Abbes, Pierre Colmez, Jean-Marc Fontaine, Ofer Gabber, Eugen Hellmann, Adrian Iovita, Wieslawa Niziol, Michael Rapoport and Timo Richarz for helpful discussions. These results were the basis both for an ARGOS seminar in Bonn in the summer of 2011, and a lecture course in Bonn in the summer of 2012, and the author thanks the participants for working through this manuscript. This work was done while the author was a Clay Research Fellow.

\section{Almost finitely generated $\OO$-modules}\label{AlmostMath}

Let $K$ be a nonarchimedean field, i.e. a topological field whose topology is induced by a nonarchimedean norm $|\cdot|: K\rightarrow \R_{\geq 0}$. We assume that the value group $\Gamma = |K^\times|\subset \R_{>0}$ is dense. Let $\OO\subset K$ be the ring of integers, and fix $\pi\in \OO$ some topologically nilpotent element, i.e. $|\pi|<1$. Using the logarithm with base $|\pi|$, we identify $\R_{>0}$ with $\R$; this induces a valuation map $v: K\rightarrow \R\cup \{\infty\}$ sending $\pi$ to $1$. We write $\log \Gamma\subset \R$ for the induced subgroup. For any $r\in \log \Gamma$ we fix some element, formally written as $\pi^r\in K$, such that $|\pi^r|=|\pi|^r$.

In this setting, the maximal ideal $\mm$ of $\OO$ is generated by all $\pi^\epsilon$, $\epsilon>0$, and satisfies $\mm^2=\mm$. We consider the category of almost $\OO$-modules with respect to the ideal $\mm$, i.e. an $\OO$-module $M$ is called almost zero if $\mm M=0$:

\begin{definition} The category of $\OO^a$-modules, or almost $\OO$-modules, is the quotient of the category of $\OO$-modules modulo the category of almost zero modules.
\end{definition}

We denote by $M\mapsto M^a$ the functor from $\OO$-modules to $\OO^a$-modules.

\begin{definition} Let $M$ and $N$ be two $\OO$-modules. For any $\epsilon>0$, $\epsilon\in \log \Gamma$, we say that $M\approx_\epsilon N$ if there are maps $f_\epsilon: M\rightarrow N$, $g_\epsilon: N\rightarrow M$ such that $f_\epsilon g_\epsilon = g_\epsilon f_\epsilon = \pi^\epsilon$. Moreover, if $M\approx_\epsilon N$ for all $\epsilon>0$, we write $M\approx N$.
\end{definition}

Note that the relations $\approx_\epsilon$ and $\approx$ are symmetric, and transitive in the following sense: If $M\approx_\epsilon N$ and $N\approx_\delta L$, then $M\approx_{\epsilon+\delta} L$. In particular, $\approx$ is transitive in the usual sense. Also, note that $M$ is almost zero if and only if $M\approx 0$. In general, for two $\OO$-modules $M$, $N$, if $M^a\cong N^a$ as $\OO^a$-modules, then $M\approx N$, but the converse is not true. In this section, we will concentrate on the equivalence classes of the relation $\approx$ instead of isomorphism classes of $\OO^a$-modules, which is slightly nonstandard in almost mathematics. For this reason, we will mostly work with honest $\OO$-modules instead of $\OO^a$-modules, as the use of the latter will often not clarify the situation.

\begin{definition} Let $M$ be an $\OO$-module. Then $M$ is called almost finitely generated (resp. almost finitely presented) if for all $\epsilon>0$, $\epsilon\in \log\Gamma$, there exists some finitely generated (resp. finitely presented) $\OO$-module $N_\epsilon$ such that $M\approx_\epsilon N_\epsilon$.
\end{definition}

The property of being almost finitely generated (resp. presented) depends only on the $\OO^a$-module $M^a$, so that we may also talk about an $\OO^a$-module being almost finitely generated (resp. presented).

\begin{example}\label{ExAlmFinGen} \begin{altenumerate}
\item[{\rm (i)}] Recall that any finitely generated ideal of $\OO$ is principal, so that $\OO$ is coherent, i.e. any finitely generated submodule of a finitely presented module is again finitely presented. Now let $r\in \R$, $r\geq 0$ and consider the ideal
\[
I_r = \bigcup_{\epsilon\in \log\Gamma, \epsilon>r} \pi^\epsilon \OO\subset \OO\ .
\]
Then the inclusions $\OO\cong \pi^\epsilon\OO\subset I_r$ for $\epsilon>r$ show that $\OO\approx I_r$. However, one can check that $I_r^a$ is not isomorphic to $\OO^a$ as $\OO^a$-modules if $r\not\in \log\Gamma$. Note that all nonprincipal ideals of $\OO$ are of the form $I_r$, in particular all nonzero ideals $I\subset \OO$ satisfy $I\approx \OO$, and hence are almost finitely presented.
\item[{\rm (ii)}] Let $\gamma_1,\gamma_2,\ldots\in \R_{\geq 0}$, such that $\gamma_i\rightarrow 0$ for $i\rightarrow \infty$. Then
\[
\OO/I_{\gamma_1}\oplus \OO/I_{\gamma_2}\oplus\ldots
\]
is almost finitely presented.
\end{altenumerate}
\end{example}

The main theorem of this section is the following.

\begin{thm} Let $M$ be any almost finitely generated $\OO$-module. Then there exists a unique series $\gamma_1\geq\gamma_2\geq\ldots\geq 0$ of real numbers such that $\gamma_i\rightarrow 0$ for $i\rightarrow \infty$, and a unique integer $r\geq 0$, such that
\[
M\approx \OO^r\oplus \OO/I_{\gamma_1}\oplus \OO/I_{\gamma_2}\oplus\ldots \ .
\]
\end{thm}

First, we note that $\OO$ is 'almost noetherian'.

\begin{prop} Every almost finitely generated $\OO$-module is almost finitely presented.
\end{prop}

\begin{proof} First, recall the following abstract result.

\begin{prop}[\cite{GabberRamero}, Lemma 2.3.18] Let $0\rightarrow M^\prime\rightarrow M\rightarrow M^{\prime\prime}\rightarrow 0$ be an exact sequence of $\OO$-modules.
\begin{altenumerate}
\item[{\rm (i)}] If $M$ is almost finitely generated, then $M^{\prime\prime}$ is almost finitely generated.
\item[{\rm (ii)}] If $M^\prime$ and $M^{\prime\prime}$ are almost finitely generated (resp. presented), then $M$ is almost finitely generated (resp. presented).
\item[{\rm (iii)}] If $M$ is almost finitely generated and $M^{\prime\prime}$ is almost finitely presented, then $M^\prime$ is almost finitely generated.
\item[{\rm (iv)}] If $M$ is almost finitely presented and $M^\prime$ is almost finitely generated, then $M^{\prime\prime}$ is almost finitely presented.
\end{altenumerate}
\end{prop}

Now let $M$ be an almost finitely generated $\OO$-module; we want to show that it is almost finitely presented. We start with the case that $M$ is generated by one element, $M=\OO/I$ for some ideal $I\subset \OO$. By Example \ref{ExAlmFinGen} (i), $I$ is almost finitely generated, giving the claim by part (iv).

Now assume that $M$ is finitely generated, and let $0=M_0\subset M_1\subset\ldots\subset M_k=M$ be a filtration such that all $M_i/M_{i-1}$ are generated by one element. Then by the previous result, all $M_i/M_{i-1}$ are almost finitely presented, and hence $M$ is almost finitely presented by part (ii).

Finally, take $M$ any almost finitely generated $\OO$-module. Let $\epsilon>0$, $\epsilon\in \log\Gamma$, and choose $N_\epsilon$ finitely generated, $M\approx_\epsilon N_\epsilon$. Then $N_\epsilon$ is almost finitely presented, so there exists some finitely presented $L_\epsilon$ such that $N_\epsilon\approx_\epsilon L_\epsilon$. Then $M\approx_{2\epsilon} L_\epsilon$, and letting $\epsilon\rightarrow 0$, we get the result.
\end{proof}

We note that it follows that any subquotient of an almost finitely generated $\OO$-module is almost finitely generated, so that in particular the category of almost finitely generated $\OO$-modules is abelian.

The following proposition reduces the classification problem to the case of torsion modules.

\begin{prop}\begin{altenumerate}
\item[{\rm (i)}] Let $M$ be a finitely generated torsion-free $\OO$-module. Then $M$ is free of finite rank.
\item[{\rm (ii)}] Let $M$ be an almost finitely generated torsion-free $\OO$-module. Then $M\approx \OO^r$ for a unique integer $r\geq 0$.
\item[{\rm (iii)}] Let $M$ be an almost finitely generated $\OO$-module. Then there exists a unique integer $r\geq 0$ and an almost finitely generated torsion $\OO$-module $N$ such that $M\approx \OO^r\oplus N$.
\end{altenumerate}
\end{prop}

\begin{proof} \begin{altenumerate}
\item[{\rm (i)}] Let $\kappa$ be the residue field of $\OO$. Lifting a basis of the finite-dimensional $\kappa$-vector space $M\otimes \kappa$, we get a surjection $\OO^r\rightarrow M$. Assume that the kernel is nontrivial, and let $f\in \OO^r$ be in the kernel. Write $f=\pi^\gamma g$ for some $\gamma>0$, $\gamma\in \log\Gamma$, such that $g$ has nontrivial image in $\kappa^r$. This is possible, as greatest common divisors of finitely elements of $\OO$ exist. Now $g$ has nontrivial image in $M$, but $\pi^\gamma g$ becomes $0$ in $M$, which means that $M$ has torsion, which is a contradiction.
\item[{\rm (ii)}] By definition, for any $\epsilon>0$, $\epsilon\in \log\Gamma$, there is some finitely generated submodule $N_\epsilon\subset M$ such that $M\approx_\epsilon N_\epsilon$. But then $N_\epsilon$ is finitely generated and torsion-free, hence free of finite rank. The rank is determined as the dimension $r$ of the $K$-vector space $M\otimes K$. Hence $M\approx_\epsilon N_\epsilon\cong \OO^r$ for all $\epsilon>0$, giving the claim.
\item[{\rm (iii)}] Let $N\subset M$ be the torsion submodule, which is almost finitely generated by our previous results. Let $M^\prime = M/N$, which is almost finitely generated and torsion-free, hence $M^\prime\approx \OO^r$, where $r$ is the dimension of $M\otimes K$. For any $\epsilon>0$, $\epsilon\in \log\Gamma$, there is some $M_\epsilon$ such that $M\approx_\epsilon M_\epsilon$ and $M_\epsilon$ is an extension of $\OO^r$ by $N$. As $\OO^r$ is projective, $M_\epsilon\cong \OO^r\oplus N$, hence $M\approx_\epsilon \OO^r\oplus N$. Letting $\epsilon\rightarrow 0$, we get the result.
\end{altenumerate}
\end{proof}

Next, we discuss elementary divisors for finitely presented torsion $\OO$-modules, and then for almost finitely generated torsion $\OO$-modules.

\begin{definition} Let $\ell^\infty(\N)_0$ be the space of sequences $\gamma = (\gamma_1,\gamma_2,\ldots)$, $\gamma_i\in \R$ such that $\gamma_i\rightarrow 0$ for $i\rightarrow \infty$. We endow it with the $\ell^\infty$-norm $||\gamma|| = \max(|\gamma_i|)$. Let $\ell^\infty_{\geq}(\N)_0\subset \ell^\infty(\N)_0$ be the subspace of sequences for which $\gamma_1\geq \gamma_2\geq \ldots \geq 0$.

On $\ell^\infty_{\geq}(\N)_0$, introduce the majorization order $\geq$ by saying that $\gamma\geq \gamma^\prime$ if and only if for all $i\geq 1$,
\[
\gamma_1+\ldots+\gamma_i\geq \gamma_1^\prime+\ldots+\gamma_i^\prime\ .
\]
\end{definition}

Note that $\ell^\infty(\N)_0$ is complete for the norm $||\cdot||$, and that $\ell^\infty_\geq(\N)_0\subset \ell^\infty(\N)_0$ is a closed subspace.

\begin{prop}\begin{altenumerate}
\item[{\rm (i)}] Let $M$ be a finitely presented torsion $\OO$-module. Then there exist unique $\gamma_{M,1}\geq \gamma_{M,2}\geq \ldots\geq \gamma_{M,k} > 0$, $\gamma_{M,i}\in \log \Gamma$, such that
\[
M\cong \OO/\pi^{\gamma_{M,1}}\oplus \OO/\pi^{\gamma_{M,2}}\oplus\ldots\oplus\OO/\pi^{\gamma_{M,k}}\ .
\]
Write $\gamma_M = (\gamma_{M,1},\ldots,\gamma_{M,k},0,\ldots)\in \ell^\infty_\geq(\N)_0$, and set $\lambda(M) = \gamma_{M,1} + \ldots + \gamma_{M,k}$, called the length of $M$.
\item[{\rm (ii)}] Let $M$, $M^\prime$ be finitely presented torsion $\OO$-modules. If $M^\prime$ is a subquotient of $M$, then $\gamma_{M^\prime,i}\leq \gamma_{M,i}$ for all $i\geq 1$.
\item[{\rm (iii)}] Let $0\rightarrow M^\prime\rightarrow M\rightarrow M^{\prime\prime}\rightarrow 0$ be an exact sequence of finitely presented torsion $\OO$-modules. Then $\lambda(M) = \lambda(M^\prime) + \lambda(M^{\prime\prime})$ and $\gamma_M\leq \gamma_{M^\prime} + \gamma_{M^{\prime\prime}}$, using the majorization order.
\item[{\rm (iv)}] Let $M$, $M^\prime$ be finitely presented torsion $\OO$-modules. Then $M\approx_\epsilon M^\prime$ if and only if $||\gamma_M - \gamma_{M^\prime}||\leq \epsilon$.
\end{altenumerate}
\end{prop}

\begin{proof}\begin{altenumerate}
\item[{\rm (i)}] Choose a short exact sequence $0\rightarrow N\rightarrow \OO^k\rightarrow M\rightarrow 0$. Then $N$ is finitely generated (as $M$ is finitely presented) and torsion-free, hence free of finite rank. As $M$ is torsion, the rank of $N$ is $k$, hence $N\cong \OO^k$. It follows that there is some matrix $A\in M_k(\OO)\cap \GL_k(K)$ such that $M=\coker A$. But the Cartan decomposition
\[
M_k(\OO)\cap \GL_k(K) =\bigcup \GL_k(\OO) \diag(\pi^{\gamma_1},\ldots,\pi^{\gamma_k})\GL_k(\OO)
\]
holds true over $K$ with the usual proof: One defines $\pi^{\gamma_k}$ as the greatest common divisor of all entries of $A$, moves this entry to the lower-right corner, eliminates the lowest row and right-most column, and then proceeds by induction on $k$. We may thus replace $A$ by a diagonal matrix with entries $\pi^{\gamma_1},\ldots,\pi^{\gamma_k}$, and then the result is clear. For uniqueness, note that the $\gamma_{M,i}$ are the jumps of the function mapping $\gamma\geq 0$, $\gamma\in \log\Gamma$, to
\[
\dim_{\kappa} ((\pi^\gamma M)\otimes_\OO \kappa)\ .
\]
\item[{\rm (ii)}] It is enough to deal with the case of submodules and quotients. Using the duality $M\mapsto \Hom_\OO(M,K/\OO)$, one reduces the case of submodules to the case of quotients. Hence assume $M^\prime$ is a quotient of $M$ and $\gamma_{M^\prime,i}>\gamma_{M,i}$ for some $i$. Set $\gamma=\gamma_{M^\prime,i}$ and replace $M$ and $M^\prime$ by $M/\pi^\gamma$, resp. $M^\prime/\pi^\gamma$. Then
\[
\gamma_{M,k}=\gamma_{M^\prime,k} = \gamma_{M^\prime,i} = \gamma
\]
for $k<i$, but $\gamma_{M,i}<\gamma$. It follows that $M^\prime$ admits a direct summand $L=(\OO/\pi^\gamma)^k$. The surjection $M\rightarrow L$ of $\OO/\pi^\gamma$-modules splits, hence $L$ is a direct summand of $M$. But then $\gamma_{M,i}\geq \gamma_{L,i} = \gamma$, contradiction.
\item[{\rm (iii)}] The additivity of $\lambda$ follows easily from the multiplicativity of the $0$-th Fitting ideal, but one can phrase the proof more elementary as follows. It is easy to construct a commutative exact diagram
\[\xymatrix{
& 0\ar[d] & 0\ar[d] & 0\ar[d] &\\
0\ar[r] & \OO^{k^\prime}\ar[d]^{A^\prime}\ar[r] & \OO^k\ar[d]^A\ar[r] & \OO^{k^{\prime\prime}}\ar[d]^{A^{\prime\prime}}\ar[r] & 0\\
0\ar[r] & \OO^{k^\prime}\ar[d]\ar[r] & \OO^k\ar[d]\ar[r] & \OO^{k^{\prime\prime}}\ar[d]\ar[r] & 0\\
0\ar[r] & M^\prime\ar[d]\ar[r] & M\ar[d]\ar[r] & M^{\prime\prime}\ar[d]\ar[r] & 0\\
& 0 & 0 & 0 &\\
}\]
Here, $k=k^\prime+k^{\prime\prime}$, the maps $\OO^{k^\prime}\rightarrow \OO^k$ are the inclusion of the first $k^\prime$ coordinates, and the maps $\OO^k\rightarrow \OO^{k^{\prime\prime}}$ are the projection to the last $k^{\prime\prime}$ coordinates. In that case, $\lambda(M^\prime) = v(\det A^\prime)$, and similarly for $M$ and $M^{\prime\prime}$, and $A$ is a block-upper triangular matrix with $A^\prime$ and $A^{\prime\prime}$ on the diagonal, so that $\det A = \det A^\prime \det A^{\prime\prime}$, giving the result.

To show that $\gamma_M\leq \gamma_{M^\prime} + \gamma_{M^{\prime\prime}}$, choose some integer $i\geq 1$, and let $M_i\subset M$ be the direct sum
\[
M_i = \OO/\pi^{\gamma_{M,1}}\oplus\ldots\oplus \OO/\pi^{\gamma_{M,i}}\ .
\]
Note that $\lambda(M_i) = \gamma_{M,1} + \ldots + \gamma_{M,i}$. Define $M_i^\prime = M_i\cap M^\prime$ and $M_i^{\prime\prime}$ as the image of $M_i$ in $M^{\prime\prime}$. We get an exact sequence
\[
0\rightarrow M_i^\prime\rightarrow M_i\rightarrow M_i^{\prime\prime}\rightarrow 0\ .
\]
From part (ii), it follows that $\gamma_{M_i^\prime,j}=0$ for $j>i$ by comparison to $M_i$, and $\gamma_{M_i^\prime,j}\leq \gamma_{M^\prime,j}$ for all $j$. In particular,
\[
\lambda(M_i^\prime)\leq \gamma_{M^\prime,1}+\ldots+\gamma_{M^\prime,i}\ .
\]
Similarly,
\[
\lambda(M_i^{\prime\prime})\leq \gamma_{M^{\prime\prime},1}+\ldots+\gamma_{M^{\prime\prime},i}\ .
\]
Now the desired inequality follows from additivity of $\lambda$.
\item[{\rm (iv)}] If $M\approx_\epsilon M^\prime$, there is a quotient $L$ of $M$ such that $\pi^\epsilon M^\prime\subset L\subset M^\prime$. By coherence of $\OO$, $L$ is finitely presented. Then
\[
\gamma_{M,i}\geq \gamma_{L,i}\geq \gamma_{\pi^\epsilon M^\prime,i} = \max(\gamma_{M^\prime,i} - \epsilon,0)\geq \gamma_{M^\prime,i} - \epsilon\ .
\]
By symmetry, we get $||\gamma_M - \gamma_{M^\prime}||\leq \epsilon$. The other direction is obvious.
\end{altenumerate}
\end{proof}

Now we can go to the limit.

\begin{prop}\begin{altenumerate}
\item[{\rm (i)}] There exists a unique map sending any almost finitely generated torsion $\OO$-module $M$ to an element $\gamma_M\in \ell^\infty_\geq(\N)_0$ that extends the definition for finitely presented $M$, and such that whenever $M\approx_\epsilon M^\prime$, then $||\gamma_M - \gamma_{M^{\prime}}||\leq \epsilon$. Set
\[
\lambda(M) = \sum_{i=1}^\infty \gamma_{M,i}\in \R_{\geq 0}\cup \{\infty\}\ .
\]
\item[{\rm (ii)}] Any almost finitely generated torsion $\OO$-module $M$ with $\gamma_M=(\gamma_1,\gamma_2,\ldots)$ satisfies
\[
M\approx \OO/I_{\gamma_1}\oplus \OO/I_{\gamma_2}\oplus\ldots\ .
\]
In particular, $\gamma_M=0$ if and only if $M$ is almost zero.
\item[{\rm (iii)}] If $M$, $M^\prime$ are almost finitely generated torsion $\OO$-modules such that $M^\prime$ is a subquotient of $M$, then $\gamma_{M^\prime,i}\leq \gamma_{M,i}$ for all $i\geq 1$.
\item[{\rm (iv)}] If $0\rightarrow M^\prime\rightarrow M\rightarrow M^{\prime\prime}\rightarrow 0$ is an exact sequence of almost finitely generated torsion $\OO$-modules, then $\gamma_M\leq \gamma_{M^\prime}+\gamma_{M^{\prime\prime}}$ in the majorization order, and $\lambda(M) = \lambda(M^\prime) + \lambda(M^{\prime\prime})$.
\item[{\rm (v)}] If in (iv), $\gamma_{M^\prime} = \gamma_M$, then $M^{\prime\prime}$ is almost zero. Similarly, if $\gamma_M = \gamma_{M^{\prime\prime}}$, then $M^\prime$ is almost zero.
\end{altenumerate}
\end{prop}

We call $\gamma_M = (\gamma_{M,1},\gamma_{M,2},\ldots)$ the sequence of elementary divisors of $M$.

\begin{proof}\begin{altenumerate}
\item[{\rm (i)}] In order to define $\gamma_M$, choose some sequence of finitely presented torsion $\OO$-modules $M_\epsilon$ such that $M\approx_\epsilon M_\epsilon$. Then by transitivity of $\approx$ and part (iv) of the previous proposition, $\gamma_{M_\epsilon}$ is a Cauchy sequence in $\ell^\infty(\N)_0$, converging to an element $\gamma_M\in \ell^\infty_\geq(\N)_0$. Clearly, the conditions given force this definition. Moreover, the inequality $||\gamma_M - \gamma_{M^\prime}||\leq \epsilon$ for $M\approx_\epsilon M^\prime$ follows by approximating by finitely presented torsion $\OO$-modules.

\item[{\rm (ii)}] Let $N$ denote the right-hand side. Then $\gamma_N = \gamma_M$. Choose some $\epsilon>0$, $\epsilon\in \log\Gamma$, and finitely presented $N_\epsilon$, $M_\epsilon$ as usual. Then $||\gamma_{N_\epsilon} - \gamma_{M_\epsilon}||\leq 2\epsilon$, hence $N_\epsilon\approx_{2\epsilon} M_\epsilon$. It follows that $N\approx_{4\epsilon} M$, and we get the result as $\epsilon\rightarrow 0$.

\item[{\rm (iii)}] Follows by approximation from the finitely presented case.

\item[{\rm (iv)}] The majorization inequality follows by direct approximation from the finitely presented case. For the additivity of $\lambda$, we argue as follows. We know that $\lambda(M)\leq \lambda(M^\prime)+\lambda(M^{\prime\prime})$, and we have to prove the reverse inequality. Let $r^\prime<\lambda(M^\prime)$, $r^{\prime\prime}<\lambda(M^{\prime\prime})$, $r^{\prime}, r^{\prime\prime}\in \R$ be any real numbers. Then there exists a finitely presented subquotient $M_1$ of $M$ inducing subquotients $M_1^\prime$ and $M_1^{\prime\prime}$ of $M^\prime$ and $M^{\prime\prime}$, sitting in an exact sequence
\[
0\rightarrow M_1^\prime\rightarrow M_1\rightarrow M_1^{\prime\prime}\rightarrow 0\ ,
\]
and such that $\lambda(M_1^\prime)>r^\prime$, $\lambda(M_1^{\prime\prime})>r^{\prime\prime}$. Indeed, first replace $M$ by a large finitely generated submodule, and then take a large finitely presented quotient. Then let $M_2^\prime\subset M_1^\prime$ be a finitely generated submodule with small quotient; we still have $\lambda(M_2^\prime)>r^\prime$. One gets an induced quotient $M_2^{\prime\prime}=M_1/M_2^\prime$, which has $M_1^{\prime\prime}$ as further quotient, so that $\lambda(M_2^{\prime\prime})>r^{\prime\prime}$. Because $\OO$ is coherent, $M_2^\prime$ and $M_2^{\prime\prime}$ are finitely presented, and one finishes the proof by using additivity of $\lambda$ in the finitely presented case.

\item[{\rm (v)}] Using duality $M\mapsto \Hom_\OO(M,K/\OO)$, we reduce to the second case. Note that if $\lambda(M)<\infty$, this is a direct consequence of additivity of $\lambda$, and part (ii). In the general case, apply this reasoning to $\pi^\epsilon M$ and $\pi^\epsilon M^{\prime\prime}$ for any $\epsilon>0$: In general, we deduce from the classification result that
\[
\gamma_{\pi^\epsilon M,i} = \max(\gamma_{M,i}-\epsilon,0)\ ,
\]
hence $\gamma_{\pi^\epsilon M} = \gamma_{\pi^\epsilon M^{\prime\prime}}$. Moreover, only finitely many $\gamma_{\pi^\epsilon M,i}$ are nonzero, hence $\pi^\epsilon M$ has finite length. It follows that $M^\prime\cap \pi^\epsilon M$ is almost zero. But
\[
M^\prime\approx_\epsilon M^\prime\cap \pi^\epsilon M\ ,
\]
so that $||\gamma_{M^\prime}||\leq \epsilon$. As $\epsilon\rightarrow 0$, this gives the result.
\end{altenumerate}
\end{proof}

This finishes the proof of the classification result. We will need the following application of these results.

\begin{lem}\label{KeyAlmostLemma} Assume additionally that $K$ is an algebraically closed field of characteristic $p$. Let $M_k$ be an $\OO/\pi^k$-module for any $k\geq 1$, such that
\begin{altenumerate}
\item[{\rm (i)}] $M_1$ is almost finitely generated.
\item[{\rm (ii)}] There are maps $p_k: M_{k+1}\rightarrow M_k$ and $q_k: M_k\rightarrow M_{k+1}$ such that $p_k q_k : M_k\rightarrow M_k$ is multiplication by $\pi$, and such that
\[
M_1\buildrel {q_k\cdots q_1}\over\rightarrow M_{k+1}\buildrel {p_k}\over\rightarrow M_k
\]
is exact in the middle.
\item[{\rm (iii)}] There are isomorphisms
\[
\varphi_k: M_k\otimes_{\OO/\pi^k,\varphi} \OO/\pi^{pk}\cong M_{pk}
\]
compatible with $p_k$, $q_k$.
\end{altenumerate}

Then there exists some integer $r\geq 0$ and isomorphisms of $\OO^a$-modules
\[
M_k^a\cong (\OO^a/\pi^k)^r
\]
for all $k$, such that $p_k$ is carried to the obvious projection, $q_k$ is carried to the multiplication by $\pi$-morphism, and $\varphi_k$ is carried to the coordinate-wise Frobenius map.
\end{lem}

\begin{proof} We may assume that $\OO$ is complete. From part (ii), we see by induction that $M_k$ is almost finitely generated for all $k$, and that $\gamma_{M_{k+1}}\leq \gamma_1 + \gamma_{M_k}$ for all $k\geq 1$. In particular, $\gamma_{M_k}\leq k\gamma_{M_1}$ for all $k\geq 1$. On the other hand, part (iii) implies that $\gamma_{M_{pk}} = p\gamma_{M_k}$. Taken together, this implies that $\gamma_{M_k} = k\gamma_{M_1}$ for all $k\geq 1$. Let $M_1^\prime\subset M_{k+1}$ be the image of $M_1$, and $M_k^\prime\subset M_k$ the image of $M_{k+1}$. Then we have an exact sequence
\[
0\rightarrow M_1^\prime\rightarrow M_{k+1}\rightarrow M_k^\prime\rightarrow 0\ .
\]
It follows that
\[
(k+1)\gamma_{M_1} = \gamma_{M_{k+1}}\leq \gamma_{M_1^\prime} + \gamma_{M_k^\prime}\leq \gamma_{M_1} + \gamma_{M_k} = (k+1)\gamma_{M_1}\ .
\]
Hence, all inequalities are equalities, and in particular $\gamma_{M_1^\prime} = \gamma_{M_1}$ and $\gamma_{M_k^\prime} = \gamma_{M_k}$. By part (v) of the previous proposition, we get $M_1^{\prime a} = M_1^a$ and $M_k^{\prime a} = M_k^a$, so that
\[
0\rightarrow M_1^a\rightarrow M_{k+1}^a\rightarrow M_k^a\rightarrow 0
\]
is exact. By induction over $j\geq 1$, the sequence
\[
0\rightarrow M_j^a\rightarrow M_{k+j}^a\rightarrow M_k^a\rightarrow 0
\]
is exact. Let $M=\varprojlim_k M_k$. Taking the inverse limit over $j$ in the previous exact sequence, we see that
\[
0\rightarrow M^a\buildrel {\pi^k}\over\rightarrow M^a\rightarrow M_k^a\rightarrow 0
\]
is exact. This implies that $M^a$ is flat, and because $M_1^a=M^a/\pi$ is almost finitely generated and $\OO$ is complete, also $M^a$ is almost finitely generated, by the following lemma.

\begin{lem} Let $A$ be an $\OO^a$-module such that $A\cong \varprojlim_k A/\pi^k$, and such that $A/\pi$ is almost finitely generated. Then $A$ is almost finitely generated.
\end{lem}

\begin{proof} Choose some $0<\epsilon<1$, $\epsilon\in\log\Gamma$, and some map $\OO^r\rightarrow A/\pi$ whose kernel is annihilated by $\pi^\epsilon$. Take any lift $f: \OO^r\rightarrow A$; we claim that the cokernel of $\OO^r\rightarrow A$ is the same as the cokernel of $\OO^r\rightarrow A/\pi$, in particular annihilated by $\pi^\epsilon$. Indeed, take any $a_0\in A$ with trivial image in the cokernel of $\OO^r\rightarrow A/\pi$. Then $a_0=f(x_0)+\pi b_1$ for some $x_0\in \OO^r$, $b_1\in A$. Let $a_1 = \pi^\epsilon b_1$. Then $a_0=f(x_0) + \pi^{1-\epsilon} a_1$ and $a_1$ has trivial image in the cokernel of $\OO^r\rightarrow A/\pi$. In particular, we can repeat the argument with $a_0$ replaced by $a_1$, which gives a $\pi$-adically convergent series
\[
a_0=f(x_0+\pi^{1-\epsilon} x_1 + \pi^{2(1-\epsilon)} x_2 + \ldots)\ ,
\]
which shows that $a_0$ is in the image of $f$.
\end{proof}

Now the $\varphi_k$ induce an isomorphism $\varphi: M\otimes_{\OO,\varphi} \OO\cong M$. Then $(M\otimes K,\varphi)\cong (K^r,\varphi)$ for some integer $r\geq 0$: Over any ring $R$ of characteristic $p$, locally free $R$-modules $N$ with an isomorphism $N\otimes_{R,\varphi} R\cong N$ are equivalent to \'{e}tale $\mathbb{F}_p$-local systems over $R$; as $K$ is algebraically closed, these are all trivial. Let $M^\prime\subset M\otimes K$ be the image of $M$; as $M^a$ is flat, $M^a\cong M^{\prime a}$. Now $M^\prime\subset M\otimes K\cong K^r$ is $\varphi$-invariant, and because $M^a$ is almost finitely generated, there is some integer $m\geq 1$ such that $\pi^m\OO^r\subset M^\prime\subset \pi^{-m}\OO^r$. By applying $\varphi^{-1}$, we see that $\pi^{m/p^k} \OO^r\subset M^\prime\subset \pi^{-m/p^k} \OO^r$ for all $k\geq 0$, hence $M^{\prime a}\cong (\OO^a)^r$, compatibly with $\varphi$. This gives the desired statement.
\end{proof}

\section{The pro-\'{e}tale site}\label{ProEtaleSection}

First, let us recall some abstract nonsense about pro-objects of a category. For details, we refer to SGA 4 I, 8.

\begin{definition} Let $\mathcal{C}$ be a category, and let $\hat{\mathcal{C}}=\mathrm{Funct}(\mathcal{C},\mathrm{Set})^{\mathrm{op}}$ with the fully faithful embedding $\mathcal{C}\rightarrow \hat{\mathcal{C}}$ be its Yoneda completion. The category $\pro-\mathcal{C}$ of pro-objects of $\mathcal{C}$ is the full subcategory of those objects of $\hat{\mathcal{C}}$ which are small cofiltered inverse limits of representable objects.
\end{definition}

The category $\pro-\mathcal{C}$ can be described equivalently as follows.

\begin{prop} The category $\pro-\mathcal{C}$ is equivalent to the category whose objects are functors $F: I\rightarrow \mathcal{C}$ from small cofiltered index categories $I$ and whose morphisms are given by
\[
\Hom(F,G) = \varprojlim_J \varinjlim_I \Hom(F(i),G(j))\ ,
\]
for any $F: I\rightarrow \mathcal{C}$ and $G: J\rightarrow \mathcal{C}$.
\end{prop}

In the following, we will use this second description and call $F: I\rightarrow \mathcal{C}$ simply a formal cofiltered inverse system $F_i$, $i\in I$. Note that cofiltered inverse limits exist in $\pro-\mathcal{C}$, cf. (dual of) SGA 4 I, Proposition 8.5.1: This amounts to combining a double inverse system into a single inverse system.

Now let $X$ be a locally noetherian scheme, or a locally noetherian adic space. We recall that an adic space is called locally noetherian if it is locally of the form $\Spa(A,A^+)$, where $A$ is strongly noetherian, or $A$ admits a noetherian ring of definition. As a consequence, if $Y\rightarrow X$ is \'{e}tale, then locally, $Y$ is connected. This will be used in verifying that the pro-\'{e}tale site is a site.

As a first step, we consider the pro-finite \'{e}tale site. Let $X_\fet$ denote the category of spaces $Y$ finite \'{e}tale over $X$. For any $U=\varprojlim U_i\in \pro-X_\fet$, we have the topological space $|U| = \varprojlim |U_i|$.

\begin{definition} The pro-finite \'{e}tale site $X_\profet$ has as underlying category the category $\pro-X_\fet$. A covering is given by a family of open morphisms $\{f_i: U_i\rightarrow U\}$ such that $|U| = \bigcup_i f_i(|U_i|)$.
\end{definition}

We will mostly be using this category in the case that $X$ is connected. In this case, fix a geometric base point $\bar{x}$ of $X$, so that we have the pro-finite fundamental group $\pi_1(X,\bar{x})$: Finite \'{e}tale covers of $X$ are equivalent to finite sets with a continuous $\pi_1(X,\bar{x})$-action.\footnote{In particular in the case of adic spaces, one may consider more refined versions of the fundamental group, cf. e.g. \cite{deJongFundamentalGroup}, which will not be used here.}

\begin{definition} For a profinite group $G$, let $G-\fsets$ denote the site whose underlying category is the category of finite sets $S$ with continuous $G$-action, and a covering is given by a family of $G$-equivariant maps $\{f_i: S_i\rightarrow S\}$ such that $S=\bigcup_i f_i(S_i)$. Let $G-\pfsets$ denote the site whose underlying category is the category of profinite sets $S$ with continuous $G$-action, and a covering is given by a family of open continuous $G$-equivariant maps $\{f_i: S_i\rightarrow S\}$ such that $S=\bigcup_i f_i(S_i)$.
\end{definition}

\begin{prop}\label{ProfinitePi1} Let $X$ be a connected locally noetherian scheme or connected locally noetherian adic space. Then there is a canonical equivalence of sites
\[
X_\profet\cong \pi_1(X,\bar{x})-\pfsets\ .
\]
\end{prop}

\begin{proof} The functor is given by sending $Y=\varprojlim Y_i\rightarrow X$ to $S(Y)=S=\varprojlim S_i$, where $S_i$ is the fibre of $\bar{x}$ in $Y_i$. Each $S_i$ carries a continuous $\pi_1(X,\bar{x})$-action, giving such an action on $S$. Recalling that every profinite set with continuous action by a profinite group $G$ is in fact an inverse limit of finite sets with continuous $G$-action, identifying $G-\pfsets\cong \pro-(G-\fsets)$, the equivalence of categories follows immediately from $X_\fet\cong \pi_1(X,\bar{x})-\fsets$.

We need to check that coverings are identified. For this, we have to show that a map $Y\rightarrow Z$ in $X_\profet$ is open if and only if the corresponding map $S(Y)\rightarrow S(Z)$ is open. It is easy to see that if $Y\rightarrow Z$ is open, then so is $S(Y)\rightarrow S(Z)$. Conversely, one reduces to the case of an open surjection $S(Y)\rightarrow S(Z)$.

\begin{lem}\label{LemSurjProfinite} Let $S\rightarrow S^\prime$ be an open surjective map in $G-\pfsets$ for a profinite group $G$. Then $S\rightarrow S^\prime$ can be written as an inverse limit $S=\varprojlim T_i\rightarrow S^\prime$, where each $T_i$ is of the form $T_i = A_i\times_{B_i} S^\prime$, where $A_i\rightarrow B_i$ is a surjection in $G-\fsets$, and $S^\prime\rightarrow B_i$ is some surjective map. Moreover, one can assume that $S=\varprojlim A_i$ and $S^\prime = \varprojlim B_i$.
\end{lem}

\begin{proof} Write $S=\varprojlim S_i$ as an inverse limit of finite $G$-sets. The projection $S\rightarrow S_i$ gives rise to finitely many open $U_{ij}\subset S$, the preimages of the points of $S_i$. Their open images $U_{ij}^\prime$ form a cover of $S^\prime$. We may take the refinement $V_{ij^\prime}^\prime\subset S^\prime$ given by all possible intersections of $U_{ij}^\prime$'s. Taking the preimages $V_{ij^\prime}$ of $V_{ij^\prime}^\prime$ and again taking all possible intersections of $V_{ij^\prime}$'s and $U_{ij}$'s, one gets an open cover $W_{ik}$ of $S$, mapping to the open cover $V_{ij^\prime}$ of $S^\prime$, giving rise to finite $G$-equivariant quotients $S\rightarrow A_i$ and $S^\prime\rightarrow B_i$, and a surjective map $A_i\rightarrow B_i$, such that $S\rightarrow S^\prime$ factors surjectively over $A_i\times_{B_i} S^\prime\rightarrow S^\prime$. Now $S$ is the inverse limit of these maps, giving the first claim. The last statements follow similarly from the construction.
\end{proof}

Using this structure result, one checks that if $S(Y)\rightarrow S(Z)$ is open and surjective, then $Y\rightarrow Z$ is open.
\end{proof}

Using the site $G-\pfsets$, we get a site-theoretic interpretation of continuous group cohomology, as follows. Let $M$ be any topological $G$-module. Associated to $M$, we define a sheaf $\mathcal{F}_M$ on $G-\pfsets$ by setting
\[
\mathcal{F}_M(S) = \Hom_{\cont,G}(S,M)\ .
\]
Checking that this is a sheaf is easy, using that the coverings maps are open to check that the continuity condition glues.

\begin{prop}\begin{altenumerate}
\item[{\rm (i)}] Any continuous open surjective map $S\rightarrow S^\prime$ of profinite sets admits a continuous splitting.
\item[{\rm (ii)}] For any $S\in G-\pfsets$ with free $G$-action, the functor $\mathcal{F}\mapsto \mathcal{F}(S)$ on sheaves over $G-\pfsets$ is exact.
\item[{\rm (iii)}] We have a canonical isomorphism
\[
H^i(\pt,\mathcal{F}_M) = H^i_{\cont}(G,M)
\]
for all $i\geq 0$. Here $\pt\in G-\pfsets$ is the one-point set with trivial $G$-action.
\end{altenumerate}
\end{prop}

\begin{proof}\begin{altenumerate}
\item[{\rm (i)}] Use Lemma \ref{LemSurjProfinite}. Using the notation from the lemma, the set of splittings of $A_i\rightarrow B_i$ is a nonempty finite set, hence the inverse limit is also nonempty, and any compatible system of splittings $B_i\rightarrow A_i$ gives rise to a continuous splitting $S^\prime\rightarrow S$.
\item[{\rm (ii)}] We have the projection map $S\rightarrow S/G$, an open surjective map of profinite topological spaces. By part (i), it admits a splitting, and hence $S=S/G\times G$. In particular, any $S$ with free $G$-action has the form $S=T\times G$ for a certain profinite set $T$ with trivial $G$-action. We have to check that if $\mathcal{F}\rightarrow \mathcal{F}^\prime$ is surjective, then so is $\mathcal{F}(S)\rightarrow \mathcal{F}^\prime(S)$. Let $x^\prime\in \mathcal{F}^\prime(S)$ be any section. Locally, it lifts to $\mathcal{F}$, i.e. there is a cover $\{S_i\rightarrow S\}$, which we may assume to be finite as $S$ is quasicompact, and lifts $x_i\in \mathcal{F}(S_i)$ of $x_i^\prime=x^\prime|_{S_i}\in \mathcal{F}^\prime(S_i)$. Let $T_i\subset S_i$ be the preimage of $T\subset S$; then $S_i = T_i\times G$ and the $T_i$ are profinite sets, with an open surjective family of maps $\{T_i\rightarrow T\}$. By part (i), this map splits continuously, hence $\{S_i\rightarrow S\}$ splits $G$-equivariantly, and by pullback we get $x\in \mathcal{F}(S)$ mapping to $x^\prime$.
\item[{\rm (iii)}] We use the cover $G\rightarrow \pt$ to compute the cohomology using the Cartan-Leray spectral sequence, cf. SGA 4 V Corollaire 3.3. Note that by our previous results,
\[
R\Gamma(G^n,\mathcal{F}_M)=\Hom_{\cont,G}(G^n,M) = \Hom_\cont(G^{n-1},M)
\]
for all $n\geq 1$. The left-hand side is a term of the complex computing $H^i(\pt,\mathcal{F}_M)$ via the Cartan-Leray spectral sequence, the right-hand side is a term of the complex computing $H^i_\cont(G,M)$. One easily identifies the differentials, giving the claim.
\end{altenumerate}
\end{proof}

\begin{cor} The site $G-\mathrm{pfsets}$ has enough points, given by $G$-profinite sets $S$ with free $G$-action.\hfill $\Box$
\end{cor}

Now we define the whole pro-\'{e}tale site $X_\proet$. Note that $U=\varprojlim U_i\rightarrow X$ in $\pro-X_\et$ has an underlying topological space $|U|=\varprojlim |U_i|$. This allows us to put topological conditions in the following.

\begin{definition} A morphism $U\rightarrow V$ of objects of $\pro-X_\et$ is called \'{e}tale, resp. finite \'{e}tale, if it is induced by an \'{e}tale, resp. finite \'{e}tale, morphism $U_0\rightarrow V_0$ of objects in $X_\et$, i.e. $U=U_0\times_{V_0} V$ via some morphism $V\rightarrow V_0$. A morphism $U\rightarrow V$ of objects of $\pro-X_\et$ is called pro-\'{e}tale if it can be written as a cofiltered inverse limit $U=\varprojlim U_i$ of objects $U_i\rightarrow V$ \'{e}tale over $V$, such that $U_i\rightarrow U_j$ is finite \'{e}tale and surjective for large $i>j$. Note that here $U_i$ is itself a pro-object of $X_\et$, and we use that the cofiltered inverse limit $\varprojlim U_i$ exists in $\pro-X_\et$. Such a presentation $U=\varprojlim U_i\rightarrow V$ is called a pro-\'{e}tale presentation.

The pro-\'{e}tale site $X_\proet$ has as underlying category the full subcategory of $\pro-X_\et$ of objects that are pro-\'{e}tale over $X$. Finally, a covering in $X_\proet$ is given by a family of pro-\'{e}tale morphisms $\{f_i: U_i\rightarrow U\}$ such that $|U| = \bigcup_i f_i(|U_i|)$.
\end{definition}

We have the following lemma, which in particular verifies that $X_\proet$ is indeed a site.

\begin{lem}\begin{altenumerate}
\item[{\rm (i)}] Let $U,V,W\in \pro-X_\et$, and assume that $U\rightarrow V$ is \'{e}tale, resp. finite \'{e}tale, resp. pro-\'{e}tale, and $W\rightarrow V$ is any morphism. Then $U\times_V W$ exists in $\pro-X_\et$, the map $U\times_V W\rightarrow W$ is \'{e}tale, resp. finite \'{e}tale, resp. pro-\'{e}tale, and the map $|U\times_V W|\rightarrow |U|\times_{|V|} |W|$ of underlying topological spaces is surjective.

\item[{\rm (ii)}] A composition of $U\rightarrow V\rightarrow W$ of two \'{e}tale, resp. finite \'{e}tale, morphisms in $\pro-X_\et$ is \'{e}tale, resp. finite \'{e}tale.

\item[{\rm (iii)}] Let $U\in \pro-X_\et$ and let $W\subset |U|$ be a quasicompact open subset. Then there is some $V\in \pro-X_\et$ with an \'{e}tale map $V\rightarrow U$ such that $|V|\rightarrow |U|$ induces a homeomorphism $|V|\cong W$. If $U\in X_\proet$, the following strengthening is true: One can take $V\in X_\proet$, and for any $V^\prime\in X_\proet$ such that $V^\prime\rightarrow U$ factors over $|W|$ on topological spaces, the map $V^\prime\rightarrow U$ factors over $V$.

\item[{\rm (iv)}] Any pro-\'{e}tale map $U\rightarrow V$ in $\pro-X_\et$ is open.

\item[{\rm (v)}] A surjective \'{e}tale, resp. surjective finite \'{e}tale, map $U\rightarrow V$ in $\pro-X_\et$ with $V\in X_\proet$ comes via pullback along $V\rightarrow V_0$ from a surjective \'{e}tale, resp. surjective finite \'{e}tale, map $U_0\rightarrow V_0$ of objects $U_0,V_0\in X_\et$.

\item[{\rm (vi)}] Let $U\rightarrow V\rightarrow W$ be pro-\'{e}tale morphisms in $\pro-X_\et$, and assume $W\in X_\proet$. Then $U,V\in X_\proet$ and the composition $U\rightarrow W$ is pro-\'{e}tale.

\item[{\rm (vii)}] Arbitrary finite projective limits exist in $X_\proet$.
\end{altenumerate}
\end{lem}

\begin{proof}\begin{altenumerate}
\item[{\rm (i)}] If $U\rightarrow V$ is \'{e}tale, resp. finite \'{e}tale, then by definition we reduce to the case that $U,V\in X_\et$. Writing $W$ as the inverse limit of $W_i$, we may assume that the map $W\rightarrow V$ comes from a compatible system of maps $W_i\rightarrow V$. Then $U\times_V W = \varprojlim U\times_V W_i$ exists, and $U\times_V W\rightarrow W$ is by definition again \'{e}tale, resp. finite \'{e}tale. On topological spaces, we have
\[
|U\times_V W| = \varprojlim |U\times_V W_i| \rightarrow \varprojlim |U|\times_{|V|} |W_i| = |U|\times_{|V|} |W|\ ,
\]
the first equality by definition, and the last because fibre products commute with inverse limits. But the middle map is surjective, because at each finite stage it is surjective with finite fibres, and inverse limits of nonempty finite sets are nonempty. In particular, the fibres are nonempty compact spaces.

In the general case, take a pro-\'{e}tale presentation $U=\varprojlim U_i\rightarrow V$. Then $U\times_V W = \varprojlim U_i\times_V W\rightarrow W$ is pro-\'{e}tale over $W$ by what we have just proved. On topological spaces, we have
\[
|U\times_V W| = \varprojlim |U_i\times_V W|\rightarrow \varprojlim |U_i|\times_{|V|} |W| = |U|\times_{|V|} |W|
\]
by similar reasoning. The middle map is surjective on each finite level by our previous results, with fibres compact. Thus the fibres of the middle map are inverse limits of nonempty compact topological spaces, hence nonempty.

\item[{\rm (ii)}] Write $V=V_0\times_{W_0} W$ as a pullback of an \'{e}tale, resp. finite \'{e}tale, map $V_0\rightarrow W_0$ in $X_\et$. Moreover, write $W=\varprojlim W_i$ as an inverse limit of $W_i\in X_\et$, with a compatible system of maps $W_i\rightarrow W_0$ for $i$ large. Set $V_i = V_0\times_{W_0} W_i$; then $V=\varprojlim V_i$.

Now write $U=U_0\times_{V_0} V$ as a pullback of an \'{e}tale, resp. finite \'{e}tale, map $U_0\rightarrow V_0$ in $X_\et$. The map $V\rightarrow V_0$ factors over $V_i\rightarrow V_0$ for $i$ large. Now let $U_i=U_0\times_{V_0} V_i\in X_\et$. This has an \'{e}tale, resp. finite \'{e}tale, map $U_i\rightarrow W_i$, and $U=U_i\times_{W_i} W$.

\item[{\rm (iii)}] Write $U$ as an inverse limit of $U_i\in X_\et$, let $W$ be the preimage of $W_i\subset |U_i|$ for $i$ sufficiently large. Then $W_i$ corresponds to an open subspace $V_i\subset U_i$, and we take $V=V_i\times_{U_i} U$. This clearly has the desired property. Moreover, if $U=\varprojlim U_i$ is a pro-\'{e}tale presentation, then so is the corresponding presentation of $V$. If $V^\prime\rightarrow U$ factors over $W$, then $V^\prime\rightarrow U_i$ factors over $W_i$. Choosing a pro-\'{e}tale presentation of $V^\prime$ as the inverse limit of $V^\prime_j$, the map $V^\prime\rightarrow U_i$ factors over $V^\prime_j\rightarrow U_i$ for $j$ large; moreover, as the transition maps are surjective for large $j$, the map $V^\prime_j\rightarrow U_i$ factors over $V_i\subset U_i$ for $j$ large. Then $V^\prime\rightarrow U$ factors over $V_i\times_{U_i} U=V$, as desired.

\item[{\rm (iv)}] Choose a pro-\'{e}tale presentation $U=\varprojlim U_i\rightarrow V$, and let $W\subset |U|$ be a quasicompact open subset. It comes via pullback from a quasicompact open subset $W_i\subset |U_i|$ for some $i$, and if $i$ is large enough so that all higher transition maps in the inverse limit are surjective, then the map $W\rightarrow W_i$ is surjective. Using parts (iii) and (ii), the image of $W_i\subset |U_i|\rightarrow |V|$ can be written as the image of an \'{e}tale map, so we are reduced to checking that the image of an \'{e}tale map is open.

Hence let $U\rightarrow V$ be any \'{e}tale map, written as a pullback of $U_0\rightarrow V_0$ along $V\rightarrow V_0$. Then the map $|U|\rightarrow |V|$ factors as the composite $|U_0\times_{V_0} V|\rightarrow |U_0|\times_{|V_0|} |V|\rightarrow |V|$. The first map is surjective by (i), so it suffices to check that the image of the second map is open. For this, it is enough to check that the image of $|U_0|$ in $|V_0|$ is open, but this is true because \'{e}tale maps are open, cf. \cite{Huber}, Proposition 1.7.8, in the adic case.

\item[{\rm (v)}] Write $U\rightarrow V$ as the pullback along $V\rightarrow V_0$ of some \'{e}tale, resp. finite \'{e}tale, map $U_0\rightarrow V_0$, and take a pro-\'{e}tale presentation of $V$ as an inverse limit of $V_i$. We get a compatible system of maps $V_i\rightarrow V_0$ for $i$ large. Because the composite $|U|=|U_0\times_{V_0} V|\rightarrow |U_0|\times_{|V_0|} |V|\rightarrow |V|$, and hence the second map, is surjective, we know that $|V|\rightarrow |V_0|$ factors over the image of $|U_0|$ in $|V_0|$. But $|V| = \varprojlim |V_i|$ with surjective transition maps for large $i$, hence also $|V_i|\rightarrow |V_0|$ factors over the image of $|U_0|$ in $|V_0|$ for some large $i$. Then $U_0\times_{V_0} V_i\rightarrow V_i$ is surjective and \'{e}tale, resp. finite \'{e}tale, as desired.

\item[{\rm (vi)}] We may write $U\rightarrow V$ as the composition $U\rightarrow U_0\rightarrow V$ of an inverse system $U=\varprojlim U_i\rightarrow U_0$ of finite \'{e}tale surjective maps $U_i\rightarrow U_j\rightarrow U_0$, and an \'{e}tale map $U_0\rightarrow V$. This reduces us to checking the assertion separately in the case that $U\rightarrow V$ is \'{e}tale, or an inverse system of finite \'{e}tale surjective maps.

First, assume that $U\rightarrow V$ is \'{e}tale. Then it comes via pull-back along $V\rightarrow V_0$ from some $U_0\rightarrow V_0$ of objects $U_0,V_0$ \'{e}tale over $X$. We may choose a pro-\'{e}tale presentation $V=\varprojlim V_i\rightarrow W$, and $V\rightarrow V_0$ is given by a compatible system of maps $V_i\rightarrow V_0$ for $i$ large. Then $U = \varprojlim U_0\times_{V_0} V_i$. This description shows that $U$ is pro-\'{e}tale over $W$, using part (i).

Using this reduction, we assume in the following that all maps $U\rightarrow V\rightarrow W\rightarrow X$ are inverse limits of finite \'{e}tale surjective maps, and that $X$ is connected. We want to show that all compositions are again inverse limits of finite \'{e}tale surjective maps. This reduces to a simple exercise in $X_\profet$.

\item[{\rm (vii)}] We have to check that direct products and equalizers exist. The first case follows from (i) and (vi). To check for equalizers, one reduces to proving that if $U\rightarrow X$ is pro-\'{e}tale and $V\subset U$ is an intersection of open and closed subsets, then $V\rightarrow X$ is pro-\'{e}tale. Writing $U=\varprojlim U_i$, we have $V=\varprojlim V_i$, where $V_i\subset U_i$ is the image of $V$. As $V$ is an intersection of open and closed subsets, and the transition maps are finite \'{e}tale for large $i$, it follows that $V_i\subset U_i$ is an intersection of open and closed subsets for large $i$. Since locally, $U_i$ has only a finite number of connected components, it follows that $V_i\subset U_i$ is open and closed for large $i$. Moreover, the transition maps $V_i\rightarrow V_j$ are by definition surjective, and finite \'{e}tale for large $i$, $j$, as they are unions of connected components of the map $U_i\rightarrow U_j$. This shows that $V$ is pro-\'{e}tale over $X$, as desired.
\end{altenumerate}
\end{proof}

It is part (vii) which is the most nonformal part: One needs that any $U\in X_\et$ has locally on $U$ only a finite number of connected components.

\begin{lem} Under the fully faithful embedding of categories $X_\profet\subset X_\proet$, a morphism $f: U\rightarrow V$ in $X_\profet$ is open if and only if it is pro-\'{e}tale as a morphism in $X_\proet$. In particular, the notions of coverings coincide, and there is a map of sites $X_\proet\rightarrow X_\profet$.
\end{lem}

\begin{proof} As pro-\'{e}tale maps are open, we only have to prove the converse. This follows directly from Lemma \ref{LemSurjProfinite}, under the equivalence of Proposition \ref{ProfinitePi1}.
\end{proof}

\begin{prop} Let $X$ be a locally noetherian scheme or locally noetherian adic space.
\begin{altenumerate}
\item[{\rm (i)}] Let $U=\varprojlim U_i\rightarrow X$ be a pro-\'{e}tale presentation of $U\in X_\proet$, such that all $U_i$ are affinoid. Then $U$ is a quasicompact object of $X_\proet$.
\item[{\rm (ii)}] The family of all objects $U$ as in {\rm (i)} is generating, and stable under fibre products.
\item[{\rm (iii)}] The topos associated to the site $X_{\proet}$ is algebraic, cf. SGA 4 VI, Definition 2.3, and all $U$ as in {\rm (i)} are coherent, i.e. quasicompact and quasiseparated.
\item[{\rm (iv)}] An object $U\in X_{\proet}$ is quasicompact, resp. quasiseparated, if and only if $|U|$ is quasicompact, resp. quasiseparated.
\item[{\rm (v)}] If $U\rightarrow V$ is an inverse limit of surjective finite \'{e}tale maps, then $U$ is quasicompact, resp. quasiseparated, if and only if $V$ is quasicompact, resp. quasiseparated.
\item[{\rm (vi)}] A morphism $f: U\rightarrow V$ of objects in $X_{\proet}$ is quasicompact, resp. quasiseparated, if and only if $|f|: |U|\rightarrow |V|$ is quasicompact, resp. quasiseparated.
\item[{\rm (vii)}] The site $X_{\proet}$ is quasiseparated, resp. coherent, if and only if $|X|$ is quasiseparated, resp. coherent.
\end{altenumerate}
\end{prop}

\begin{proof} \begin{altenumerate}
\item[{\rm (i)}] Each $|U_i|$ is a spectral space, and the transition maps are spectral. Hence the inverse limit $|U| = \varprojlim |U_i|$ is a spectral space, and in particular quasicompact. As pro-\'{e}tale maps are open, this gives the claim.
\item[{\rm (ii)}] Any $U\in X_\proet$ can be covered by smaller $U^\prime$ that are of the form given in (i), using that preimages of affinoids under finite \'{e}tale maps are again affinoids. This shows that the family is generating, and it is obviously stable under fibre products.
\item[{\rm (iii)}] Using the criterion of SGA 4 VI Proposition 2.1, we see that $X_\proet$ is locally algebraic and all $U$ as in (i) are coherent. We check the criterion of SGA 4 VI Proposition 2.2 by restricting to the class of $U$ as in (i) that have the additional property that $U\rightarrow X$ factors over an affinoid open subset $U_0$ of $X$. It consists of coherent objects and is still generating, and because $U\times_X U = U\times_{U_0} U$, one also checks property (i ter).
\item[{\rm (iv)}] Use SGA 4 VI Proposition 1.3 to see that if $|U|$ is quasicompact, then so is $U$, by covering $U$ by a finite number of open subsets $U_i\subset U$ of the form given in (i). Conversely, if $U$ is quasicompact, any open cover of $|U|$ induces a cover of $U$, which by definition has a finite subcover, inducing a finite subcover of $|U|$, so that $|U|$ is quasicompact.

Now take any $U$, and cover it by open subsets $U_i\subset U$, the $U_i$ as in (i). Using SGA 4 VI Corollaire 1.17, we see that $U$ is quasiseparated if and only if all $U_i\times_U U_j$ are quasicompact if and only if all $|U_i|\times_{|U|} |U_j|$ are quasicompact if and only if $|U|$ is quasiseparated.
\item[{\rm (v)}] Use SGA 4 VI Corollaire 2.6 to show that $U\rightarrow V$ is a coherent morphism, by covering $V$ by inverse limits of affinoids as in (i): they are coherent, and their inverse images are again inverse limits of affinoids, hence coherent. Hence Proposition SGA 4 VI Proposition 1.14 (ii) shows that $V$ quasicompact, resp. quasiseparated, implies $U$ quasicompact, resp. quasiseparated.

Conversely, if $U$ is quasicompact, take any open cover of $|V|$; this gives an open cover of $|U|$ which has a finite subcover. But the corresponding finite subcover of $|V|$ has to cover $|V|$, hence $|V|$ is quasicompact. Arguing similarly shows that $U$ quasiseparated implies $V$ quasiseparated.
\item[{\rm (vi)}] This follows from (iii), (iv) and SGA 4 VI Corollaire 2.6.
\item[{\rm (vii)}] This follows from (iii), (iv) and the definition of quasiseparated, resp. coherent, sites.
\end{altenumerate}
\end{proof}

Moreover, the site $X_\proet$ is clearly functorial in $X$. Let us denote by $T^\sim$ the topos associated to a site $T$. If $X$ is quasiseparated, one can also consider the subsite $X_\proetqc\subset X_\proet$ consisting of quasicompact objects in $X_\proet$; the associated topoi are the same.

\begin{prop} Let $x\in X$, corresponding to a map $Y=\Spa(K,K^+)\rightarrow X$, resp. $Y=\Spec(K)\rightarrow X$, into $X$.\footnote{We note that in the case of adic spaces, the image of $Y$ may be larger than $x$ itself.} Then there is a morphism of topoi
\[
i_x: Y_\profet^\sim\rightarrow X_\proet^\sim\ ,
\]
such that the pullback of $\mathcal{F}\in X_\proet^\sim$ is the sheafification $i_x^\ast \mathcal{F}$ of the functor
\[
V\mapsto \varinjlim_{V\rightarrow U} \mathcal{F}(U)\ ,
\]
where $U\in X_\proet$, and $V\rightarrow U$ is a map compatible with the projections to $Y\rightarrow X$.

If for all points $x$, $i_x^\ast \mathcal{F}=0$, then $\mathcal{F}=0$. In particular, $X_\proet$ has enough points, given by profinite covers of geometric points.
\end{prop}

\begin{rem} It is not enough to check stalks at geometric points: One has to include the profinite covers of geometric points to get a conservative family.
\end{rem}

\begin{proof} We leave the construction of the morphism of topoi to the reader: One reduces to the case that $X$ is affinoid, in particular quasiseparated. Then there is a morphism of sites $Y_\profet\rightarrow X_\proetqc$, induced from the taking the fibre of $U\in X_\proetqc$ above $x$.

Now let $\mathcal{F}$ be a sheaf on $X_\proet$ such that $i_x^\ast \mathcal{F}=0$ for all $x\in X$. Assume that there is some $U\in X_{\proet}$ with two distinct sections $s_1,s_2\in \mathcal{F}(U)$. We may assume that $U$ is quasicompact. Take any point $x\in X$ and let $S$ be the preimage of $x$ in $U$. It suffices to see that there is a pro-\'{e}tale map $V\rightarrow U$ with image containing $S$ such that $s_1$ and $s_2$ become identical on $V$. The preimage $S$ of $x$ corresponds to a pro-finite \'{e}tale cover $\tilde{S}\in Y_\profet$, $\tilde{S}\rightarrow U$.

Now we use that $s_1$ and $s_2$ become identical in $(i_x^\ast \mathcal{F})(S)$. This says that there is a pro-\'{e}tale cover $\tilde{S}^\prime\rightarrow \tilde{S}$ in $Y_\profet$ and some $V\in X_\proet$, $V\rightarrow U$, with a lift $\tilde{S}^\prime\rightarrow V$ such that $s_1$ and $s_2$ become identical in $\mathcal{F}(V)$. We get the following situation:

\[\xymatrix{
\tilde{S}^\prime\ar[d] \ar[r] & V\ar[d]\\
\tilde{S} \ar[d] \ar[r] & U\ar[d]\\
Y \ar[r] & X
}\]

Here both projections $V\rightarrow X$ and $U\rightarrow X$ are pro-\'{e}tale, and the map $\tilde{S}^\prime\rightarrow \tilde{S}$ is a pro-finite \'{e}tale cover.

By the usual arguments, one reduces to the case that $U$ and $V$ are cofiltered inverse limits of finite \'{e}tale surjective maps. Moreover, one may assume that $X$ is connected, and we choose a geometric point $\bar{x}$ above $x$. Let $S_{\bar{x}}$, $S^\prime_{\bar{x}}$ be the fibres of $\tilde{S}$ and $\tilde{S}^\prime$ above $\bar{x}$. Now $U, V\in X_\profet$ correspond to profinite $\pi_1(X,\bar{x})$-sets $S(U)$, $S(V)$. In fact, $S(U)=S_{\bar{x}}$, and $S^\prime_{\bar{x}}\subset S(V)$. Consider the subset $T = \pi_1(X,\bar{x}) S^\prime_{\bar{x}}\subset S(V)$, and let $V^\prime\subset V$ be the corresponding subset. Since $\tilde{S}^\prime\rightarrow \tilde{S}$ is a pro-finite \'{e}tale cover, $S^\prime_{\bar{x}}\rightarrow S_{\bar{x}}$ is an open surjective map, hence so is the map $S(V^\prime)=T\rightarrow S(U)=S_{\bar{x}}$. This means that the map $V^\prime\rightarrow U$ is pro-\'{e}tale. Since $s_1$ and $s_2$ are identical in $\mathcal{F}(V^\prime)$, this finishes the proof that $\mathcal{F}=0$.

For the last assertion, use that $Y_\profet$ has enough points.
\end{proof}

We will need a lemma about the behaviour of the pro-\'{e}tale site under change of base field. Assume that $X$ lives over a field $K$, i.e. $X\rightarrow \Spec K$, resp. $X\rightarrow \Spa(K,K^+)$, and let $L/K$ be a separable extension (with $L^+\subset L$ the integral closure of $K^+$ in the adic case). Let $X_L = X\times_{\Spec K} \Spec L$, resp. $X_L = X\times_{\Spa(K,K^+)} \Spa(L,L^+)$. We may also define an object of $X_\proet$, by taking the inverse limit of $X_{L_i}\in X_\et$ where $L_i\subset L$ runs through the finite extensions of $K$. By abuse of notation, we denote by the same symbol $X_L\in X_\proet$ this formal inverse limit.

Then, we may consider the localized site $X_\proet / X_L$ of objects with a structure map to $X_L$, and the induced covers. One immediately checks the following result.

\begin{prop}\label{ProetChangeBase} There is an equivalence of sites $X_{L,\proet}\cong X_\proet / X_L$.$\hfill \Box$
\end{prop}

There is the natural projection $\nu: X_\proet\rightarrow X_\et$. Using it, we state some general comparison isomorphisms between the \'{e}tale and pro-\'{e}tale site.

\begin{lem}\label{CompEtProet} Let $\mathcal{F}$ be an abelian sheaf on $X_\et$. For any quasicompact and quasiseparated $U=\varprojlim U_j\in X_\proet$ and any $i\geq 0$, we have
\[
H^i(U,\nu^\ast \mathcal{F}) = \varinjlim H^i(U_j,\mathcal{F})\ .
\]
\end{lem}

\begin{proof} One may assume that $\mathcal{F}$ is injective, and that $X$ is qcqs. Let us work with the site $X_\proetqc$; as it has the same associated topos, this is allowed. Let $\tilde{\mathcal{F}}$ be the presheaf on $X_\proetqc$ given by $\tilde{\mathcal{F}}(V) = \varinjlim \mathcal{F}(V_j)$, where $V=\varprojlim V_j$. Obviously, $\nu^\ast\mathcal{F}$ is the sheaf associated to $\tilde{\mathcal{F}}$. We have to show that $\tilde{\mathcal{F}}$ is a sheaf with $H^i(V,\tilde{\mathcal{F}})=0$ for all $V\in X_\proetqc$ and $i>0$. Using SGA 4 V Proposition 4.3, equivalence of (i) and (iii), we have to check that for any $U\in X_\proetqc$ with a pro-\'{e}tale covering by $V_k\rightarrow U$, $V_k\in X_\proetqc$, the corresponding Cech complex
\[
0\rightarrow \tilde{\mathcal{F}}(U)\rightarrow \prod_k \tilde{\mathcal{F}}(V_k)\rightarrow \prod_{k,k^\prime} \tilde{\mathcal{F}}(V_k\times_U V_{k^\prime})\rightarrow \ldots
\]
is exact. This shows in the first step that $\tilde{\mathcal{F}}$ is separated; in the second step that $\tilde{\mathcal{F}}$ is a sheaf; in the third step that all higher cohomology groups vanish.

We may pass to a finite subcover because $U$ is quasicompact; this may be combined into a single morphism $V\rightarrow U$. Now take a pro-\'{e}tale presentation $V=\varprojlim V_l\rightarrow U$. Then $V_l\rightarrow U$ is an \'{e}tale cover for large $l$. Since $\tilde{\mathcal{F}}(V) = \varinjlim \tilde{\mathcal{F}}(V_l)$, the Cech complex for the covering $V\rightarrow U$ is the direct limit of the Cech complexes for the coverings $V_l\rightarrow U$. This reduces us to the case that $V\rightarrow U$ is \'{e}tale.

Choose $V_j\rightarrow U_j$ \'{e}tale such that $V=V_j\times_{U_j} U$. Then denoting for $j^\prime \geq j$ by $V_{j^\prime}\rightarrow U_{j^\prime}$ the pullback of $V_j\rightarrow U_j$, $V_{j^\prime}\rightarrow U_{j^\prime}$ is an \'{e}tale cover for large $j^\prime$, and the Cech complex for $V\rightarrow U$ is the direct limit over $j^\prime$ of the Cech complexes for $V_{j^\prime}\rightarrow U_{j^\prime}$. This reduces us to checking exactness of the Cech complexes for the covers $V_{j^\prime}\rightarrow U_{j^\prime}$. But this is just the acyclicity of the injective sheaf $\mathcal{F}$ on $X_\et$.
\end{proof}

\begin{cor}\label{CompProetVSEt}\begin{altenumerate}
\item[{\rm (i)}] For any sheaf $\mathcal{F}$ on $X_\et$, the adjunction morphism $\mathcal{F}\rightarrow R\nu_\ast \nu^\ast \mathcal{F}$ is an isomorphism.

\item[{\rm (ii)}] Let $f: X\rightarrow Y$ be a quasicompact and quasiseparated morphism. Then for any sheaf $\mathcal{F}$ on $X_\et$, the base-change morphism
\[
\nu_Y^{\ast} Rf_{\et\ast} \mathcal{F}\rightarrow Rf_{\proet\ast} \nu_X^{\ast} \mathcal{F}
\]
associated to the diagram
\[\xymatrix{
X_{\proet}\ar[d]^{f_{\proet}} \ar[r]^{\nu_X} & X_{\et} \ar[d]^{f_{\et}} \\
Y_{\proet}\ar[r]^{\nu_Y} & Y_{\et}
}\]
is an isomorphism.
\end{altenumerate}
\end{cor}

\begin{proof}\begin{altenumerate}
\item[{\rm (i)}] We recall that for any $i\geq 0$, $R^i\nu_\ast \nu^\ast \mathcal{F}$ is the sheaf on $X_\et$ associated to the presheaf $U\mapsto H^i(U,\nu^\ast \mathcal{F})$, where in the last expression $U$ is considered as an element of $X_\proet$. Hence the last lemma already implies that we get an isomorphism for $i=0$. Moreover, in degree $i>0$, the lemma says that if $U$ is qcqs, then $H^i(U,\nu^\ast \mathcal{F})=H^i(U,\mathcal{F})$. But any section of $H^i(U,\mathcal{F})$ vanishes locally in the \'{e}tale topology, so that the associated sheaf is trivial.
\item[{\rm (ii)}] One checks that for any $i\geq 0$, the $i$-th cohomology sheaf of both sides is the sheafification of the presheaf taking a quasicompact and quasiseparated $U=\varprojlim U_j\rightarrow Y$ to
\[
\varinjlim H^i(U_j\times_Y X,\mathcal{F})\ .
\]
\end{altenumerate}
\end{proof}

Part (i) implies that $\nu^\ast$ gives a fully faithful embedding from abelian sheaves on $X_\et$ to abelian sheaves on $X_\proet$, and we will sometimes confuse a sheaf $\mathcal{F}$ on $X_\et$ with its natural extension $\nu^\ast \mathcal{F}$ to $X_\proet$.

One useful property of the pro-\'{e}tale site is that inverse limits are often exact. This is in stark contrast with the usual \'{e}tale site, the difference being that property (ii) of the following lemma is rarely true on the \'{e}tale site.

\begin{lem}\label{InverseLimitExact} Let $\mathcal{F}_i$, $i\in \mathbb{N}$, be an inverse system of abelian sheaves on a site $T$. Assume that there is a basis $B$ for the site $T$, such that for any $U\in B$, the following two conditions hold:
\begin{altenumerate}
\item[{\rm (i)}] The higher inverse limit $R^1\varprojlim \mathcal{F}_i(U)=0$ vanishes.
\item[{\rm (ii)}] All cohomology groups $H^j(U,\mathcal{F}_i)=0$ vanish for $j>0$.
\end{altenumerate}

Then $R^j\varprojlim \mathcal{F}_i = 0$ for $j>0$ and $(\varprojlim \mathcal{F}_i)(U) = \varprojlim \mathcal{F}_i(U)$ for all $U\in T$. Moreover, $H^j(U,\varprojlim \mathcal{F}_i)=0$ for $U\in B$ and $j>0$.
\end{lem}

\begin{proof} Consider the composition of functors
\[
\mathrm{Sh}^{\mathbb{N}}\rightarrow \mathrm{PreSh}^{\mathbb{N}}\rightarrow \mathrm{PreSh}\rightarrow \mathrm{Sh}\ .
\]
Here the first is the forgetful functor, the second is the inverse limit functor on presheaves, and the last is the sheafification functor. The first functor has the exact left adjoint given by sheafification, so that it preserves injectives and one sees that upon taking the derived functors, the conditions guarantee that for $U\in B$, all $(R^j\varprojlim \mathcal{F}_i)(U)$ vanish for $j>0$: They do before the last sheafification, hence they do after the sheafification. This already shows that all higher inverse limits $R^j\varprojlim \mathcal{F}_i$, $j>0$, vanish. Because an inverse limit of sheaves calculated as presheaves is again a sheaf, the description of $\varprojlim \mathcal{F}_i$ is always true.

For the last statement, consider the commutative diagram
\[\xymatrix{
\mathrm{Sh}^{\mathbb{N}}\ar[r]\ar[d] & \mathrm{PreSh}^{\mathbb{N}}\ar[d]\\
\mathrm{Sh}\ar[r] & \mathrm{PreSh}
}\]
expressing that the inverse limit of sheaves calculated on the level of presheaves is a sheaf. Going over the upper right corner, we have checked that the higher derived functors of the composite map are zero for $i>0$ on sections over $U\in B$. As the left vertical functor has an exact left adjoint giving by taking a sheaf to the constant inverse limit, it preserves injectives, and we have a Grothendieck spectral sequence for the composition over the lower left corner. There are no higher derived functors appearing for the left vertical functor, by what we have proved. Hence this gives $H^j(U,\varprojlim \mathcal{F}_i)=0$ for $j>0$ and $U\in B$.
\end{proof}

\section{Structure sheaves on the pro-\'{e}tale site}\label{StructureSheafSection}

\begin{definition} Let $X$ be a locally noetherian adic space over $\Spa(\Q_p,\Z_p)$. Consider the following sheaves on $X_{\proet}$.
\begin{altenumerate}
\item[{\rm (i)}] The (uncompleted) structure sheaf $\mathcal{O}_X = \nu^{\ast} \mathcal{O}_{X_\et}$, with subring of integral elements $\mathcal{O}_X^+ = \nu^\ast \mathcal{O}_{X_\et}^+$.
\item[{\rm (ii)}] The integral completed structure sheaf $\hat{\mathcal{O}}_X^+ = \varprojlim \mathcal{O}_X^+/p^n$, and the completed structure sheaf $\hat{\mathcal{O}}_X = \hat{\mathcal{O}}_X^+[\frac 1p]$.
\end{altenumerate}
\end{definition}

\begin{lem} Let $X$ be a locally noetherian adic space over $\Spa(\Q_p,\Z_p)$, and let $U\in X_\proet$.
\begin{altenumerate}
\item[{\rm (i)}] For any $x\in |U|$, we have a natural continuous valuation $f\mapsto |f(x)|$ on $\mathcal{O}_X(U)$.
\item[{\rm (ii)}] We have
\[
\mathcal{O}_X^+(U) = \{f\in \mathcal{O}_X(U)\mid \forall x\in |U|: |f(x)|\leq 1\}\ .
\]
\item[{\rm (iii)}] For any $n\geq 1$, the map of sheaves $\mathcal{O}_X^+/p^n\rightarrow \hat{\mathcal{O}}_X^+/p^n$ is an isomorphism, and $\hat{\mathcal{O}}_X^+(U)$ is flat over $\Z_p$ and $p$-adically complete.
\item[{\rm (iv)}] For any $x\in |U|$, the valuation $f\mapsto |f(x)|$ extends to a continuous valuation on $\hat{\mathcal{O}}_X(U)$.
\item[{\rm (v)}] We have
\[
\hat{\mathcal{O}}_X^+(U) = \{f\in \hat{\mathcal{O}}_X(U)\mid \forall x\in |U|: |f(x)|\leq 1\}\ .
\]
In particular, $\hat{\mathcal{O}}_X^+(U)\subset \hat{\mathcal{O}}_X(U)$ is integrally closed.
\end{altenumerate}
\end{lem}

\begin{proof} All assertions are local in $U$, so we may assume that $U$ is quasicompact and quasiseparated. We choose a pro-\'{e}tale presentation $U=\varprojlim U_i\rightarrow X$.
\begin{altenumerate}
\item[{\rm (i)}] A point $x\in |U|$ is given by a sequence of points $x_i\in |U_i|$, giving a compatible system of continuous valuations on $\mathcal{O}_{X_\et}(U_i)=\mathcal{O}_X(U_i)$. But
\[
\mathcal{O}_X(U) = (\nu^\ast \mathcal{O}_{X_\et})(U) = \varinjlim \mathcal{O}_X(U_i)\ ,
\]
so these valuations combine into a continuous valuation on $\mathcal{O}_X(U)$.

\item[{\rm (ii)}] Assume $i$ is large enough so that $|U|\rightarrow |U_i|$ is surjective, and $f\in \mathcal{O}_X(U)$ is the image of $f_i\in \mathcal{O}_X(U_i)$. Then the condition $|f(x)|\leq 1$ for all $x\in |U|$ implies $|f_i(x_i)|\leq 1$ for all $x_i\in |U_i|$, whence $f_i\in \mathcal{O}_X^+(U_i)$, so that $f\in \mathcal{O}_X^+(U)$. Conversely, if $f\in \mathcal{O}_X^+(U)$, then it comes as the image of some $f_i\in \mathcal{O}_X^+(U_i)$, and it lies in the right-hand side.

\item[{\rm (iii)}] This follows formally from flatness of $\mathcal{O}_X^+$ over $\Z_p$.

\item[{\rm (iv)}] To define the desired valuation on $f\in \hat{\OO}_X^+(U)$, represent it as the inverse system of $\overline{f}_n\in (\OO_X^+/p^n)(U)$. It makes sense to talk about $\max(|\overline{f}_n(x)|,|p|^n)$: Cover $U$ so that $\overline{f}_n$ lifts to some $f_n\in \OO_X^+$; then the valuation $|f_n(\tilde{x})|$ will depend on the preimage $\tilde{x}$ of $x$ in the cover, but the expression $\max(|f_n(\tilde{x})|,|p|^n)$ does not. If $\max(|\overline{f}_n(x)|,|p|^n)>|p|^n$ for some $n$, then we define $|f(x)| = |\overline{f}_n(x)|$; otherwise, we set $|f(x)|=0$. One easily checks that this is well-defined and continuous. Clearly, it extends to $\hat{\OO}_X(U)$.

\item[{\rm (v)}] By definition, the left-hand side is contained in the right-hand side. For the converse, note that since $U$ is quasicompact and quasiseparated, we have $\hat{\OO}_X(U) = \hat{\OO}_X^+(U)[\frac 1p]$. This reduces us to checking that if $f\in \hat{\OO}_X^+(U)$ satisfies $|f(x)|\leq |p|^n$ for all $x\in |U|$, then $f\in p^n\hat{\OO}_X^+(U)$. For this, one may use part (iii) to write $f=f_0+p^n g$ for some $f_0\in \OO_X^+(\tilde{U})$, $g\in p^n\hat{\OO}_X^+(\tilde{U})$ over some cover $\tilde{U}$ of $U$. Then we see that $|f_0(\tilde{x})|\leq |p|^n$ for all $\tilde{x}\in \tilde{U}$, and hence by part (ii), $f_0\in p^n\OO_X^+(\tilde{U})$.
\end{altenumerate}
\end{proof}

We caution the reader that we do not know whether $\mathcal{O}_X(U)\subset \hat{\mathcal{O}}_X(U)$ is always dense with respect to the topology on $\hat{\mathcal{O}}_X(U)$ having $p^n\hat{\mathcal{O}}_X^+(U)$, $n\geq 0$, as a basis of open neighborhoods of $0$. This amounts to asking whether $\mathcal{O}_X^+(U)/p^n\rightarrow \hat{\mathcal{O}}_X^+(U)/p^n$ is an isomorphism for all $n\geq 1$, or whether one could define $\hat{\mathcal{O}}_X(U)$ as the completion of $\mathcal{O}_X(U)$ with respect to the topology having $p^n\hat{\mathcal{O}}_X^+(U)$, $n\geq 0$, as a basis of open neighborhoods of $0$. In a similar vein, we ignore whether for all $U\in X_\proet$, the triple $(|U|,\hat{\mathcal{O}}_X|_{|U|},(|\bullet(x)|\mid x\in |U|))$ is an adic space. Here $\hat{\mathcal{O}}_X|_{|U|}$ denotes the restriction of $\hat{\OO}_X$ to the site of open subsets of $|U|$.

However, we will check next that there is a basis for the pro-\'{e}tale topology where these statements are true. For simplicity, let us work over a perfectoid field $K$ of characteristic $0$ with an open and bounded valuation subring $K^+\subset K$, and let $X$ over $\Spa(K,K^+)$ be a locally noetherian adic space. As in the Section \ref{AlmostMath}, we write $\Gamma = |K^\times|\subset \R_{>0}$, and identify $\R_{>0}$ with $\R$ using the logarithm with base $|p|$. For any $r\in \log \Gamma\subset \R$, we choose an element, written $p^r\in K$, such that $|p^r| = |p|^r$.

\begin{definition} Let $U\in X_{\proet}$.
\begin{altenumerate}
\item[{\rm (i)}] We say that $U$ is affinoid perfectoid if $U$ has a pro-\'{e}tale presentation $U=\varprojlim U_i\rightarrow X$ by affinoid $U_i=\Spa(R_i,R_i^+)$ such that, denoting by $R^+$ the $p$-adic completion of $\varinjlim R_i^+$, and $R=R^+[p^{-1}]$, the pair $(R,R^+)$ is a perfectoid affinoid $(K,K^+)$-algebra.
\item[{\rm (ii)}] We say that $U$ is perfectoid if it has an open cover by affinoid perfectoid $V\subset U$. Here, recall that quasicompact open subsets of $U\in X_\proet$ naturally give rise to objects in $X_\proet$.
\end{altenumerate}
\end{definition}

\begin{example} If
\[
X=\mathbb{T}^n = \Spa(K\langle T_1^{\pm 1},\ldots,T_n^{\pm 1}\rangle,K^+\langle T_1^{\pm 1},\ldots,T_n^{\pm 1}\rangle)\ ,
\]
then the inverse limit $\tilde{\mathbb{T}}^n\in X_\proet$ of the
\[
\Spa(K\langle T_1^{\pm 1/p^m},\ldots,T_n^{\pm 1/p^m}\rangle,K^+\langle T_1^{\pm 1/p^m},\ldots,T_n^{\pm 1/p^m}\rangle)\ ,
\]
$m\geq 0$, is affinoid perfectoid.
\end{example}

To an affinoid perfectoid $U$ as in (i), one associates $\hat{U} = \Spa(R,R^+)$, an affinoid perfectoid space over $K$. One immediately checks that it is well-defined, i.e. independent of the pro-\'{e}tale presentation $U=\varprojlim U_i$. Also, $U\mapsto \hat{U}$ defines a functor from affinoid perfectoid $U\in X_\proet$ to affinoid perfectoid spaces over $K$. Moreover, if $U$ is affinoid perfectoid and $U=\varprojlim U_i$ is a pro-\'{e}tale presentation, then $\hat{U}\sim \varprojlim U_i$ in the sense of \cite{ScholzePerfectoidSpaces1}, in particular $|\hat{U}| = |U|$.

\begin{lem}\label{PreciseLocalStructure} Let $U=\varprojlim U_i\in X_\proet$, $U_i=\Spa(R_i,R_i^+)$, be affinoid perfectoid, with a pro-\'{e}tale presentation. Let $(R,R^+)$ be the completion of the direct limit of the $(R_i,R_i^+)$, so that $\hat{U}=\Spa(R,R^+)$.

Assume that $V_i=\Spa(S_i,S_i^+)\rightarrow U_i$ is an \'{e}tale map which can be written as a composition of rational subsets and finite \'{e}tale maps. For $j\geq i$, write $V_j = V_i\times_{U_i} U_j = \Spa(S_j,S_j^+)$, and $V=V_i\times_{U_i} U = \varprojlim V_j\in X_\proet$. Let $A_j$ be the $p$-adic completion of the $p$-torsion free quotient of $S_j^+\otimes_{R_j^+} R^+$. Then
\begin{altenumerate}
\item[{\rm (i)}] The completion $(S,S^+)$ of the direct limit of the $(S_j,S_j^+)$ is a perfectoid affinoid $(K,K^+)$-algebra. In particular, $V$ is affinoid perfectoid. Moreover, $\hat{V} = V_j\times_{U_j} \hat{U}$ in the category of adic spaces over $K$, and $S = A_j[\frac 1p]$ for any $j\geq i$.
\item[{\rm (ii)}] For any $j\geq i$, the cokernel of the map $A_j\rightarrow S^+$ is annihilated by some power $p^N$ of $p$.
\item[{\rm (iii)}] Let $\epsilon>0$, $\epsilon\in \log \Gamma$. Then there exists some $j$ such that the cokernel of the map $A_j\rightarrow S^+$ is annihilated by $p^\epsilon$.
\end{altenumerate}
\end{lem}

\begin{proof} Clearly, one can separately treat the cases where $V_i\subset U_i$ is a rational subset and where $V_i\rightarrow U_i$ is finite \'{e}tale.

Assume first that $V_i\subset U_i$ is a rational subset, given by certain functions $f_1,\ldots,f_n,g\in R_i$. By pullback, it induces a rational subset $W$ of $\hat{U}$. Let
\[
(T,T^+) = (\OO_{\hat{U}}(W),\OO_{\hat{U}}^+(W))\ .
\]
Then $(T,T^+)$ is a perfectoid affinoid $(K,K^+)$-algebra by Theorem 6.3 (ii) of \cite{ScholzePerfectoidSpaces1}, $W=\Spa(T,T^+)$. Recall that for perfectoid affinoid $(K,K^+)$-algebras $(A,A^+)$, $A^+$ is open and bounded, i.e. carries the $p$-adic topology. Let $R_{j0}\subset R_j^+$ be an open and bounded subring. Then
\[
S_{j0} = R_{j0}\langle \frac{f_1}g,\ldots,\frac{f_n}g\rangle\subset S_j^+
\]
is an open and bounded subring. If we give $T_j = S_j\otimes_{R_j} R$ its natural topology making the image of $S_{j0}\otimes_{R_{j0}} R^+$ open and bounded, and let $T_j^+\subset T_j$ be the integral closure of the image of $S_j^+\otimes_{R_j^+} R^+$, then $\Spa(T_j,T_j^+)$ is the fibre product $V_j\times_{U_j} \hat{U}=W$. By the universal property of $(T,T^+)$, we find that $(T,T^+)$ is the completion of $(T_j,T_j^+)$. In particular, the natural topology on $T_j$ makes $T_j^+$ open and bounded (as this is true after completion, $T$ being perfectoid). Hence, the natural topology on $T_j$ agrees with the topology making the image of $S_j^+\otimes_{R_j^+} R^+$ open and bounded, as
\[
\im(S_{j0}\otimes_{R_{j0}} R^+\rightarrow T_j)\subset \im(S_j^+\otimes_{R_j^+} R^+\rightarrow T_j)\subset T_j^+\ .
\]
In particular, $T=A_j[p^{-1}]$. As $T^+\subset T$ is bounded, we also find that the cokernel of $A_j\rightarrow T^+$ is killed by some power of $p$.

Next, we claim that $S^+=T^+$, hence also $S=T$. We have to show that $S^+/p^n = \varinjlim S_j^+/p^n\rightarrow T^+/p^n$ is an isomorphism for all $n$. As $W\rightarrow V_j$ is surjective for $j$ large, the map from $S_j^+/p^n$ to $T^+/p^n$ is injective for large $j$, hence so is the map $S^+/p^n\rightarrow T^+/p^n$. For surjectivity, take $f\in T^+$. After multiplying by $p^N$, it is the image of some element of $S_i^+\hat{\otimes}_{R_i^+} R^+$, which we can approximate by $g\in S_j^+$ modulo $p^{n+N}$ if $j$ is large enough, as $R^+/p^{n+N} = \varinjlim R_j^+/p^{n+N}$. Then $g\in p^N S_j^+$ (by surjectivity of $W\rightarrow V_j$ for $j$ large), and writing $g=p^N h$, $h\in S_j^+$, we find that $h\equiv f$ modulo $p^n T^+$.

It remains to see that for any $\epsilon>0$, there exists some $j$ such that the cokernel of $A_j\rightarrow T^+$ is annihilated by $p^\epsilon$. For this, it suffices to exhibit a subalgebra $T_\epsilon^+\subset T^+$, topologically finitely generated over $R^+$, whose cokernel is annihilated by $p^\epsilon$: Any generator can be approximated modulo $p^N$ by an element of $S_j^+$ for $j$ large enough, so that for $j$ very large, $T_\epsilon^+$ will be in the image of $A_j$.

The existence of such subalgebras is an abstract question about perfectoid spaces. It follows from Lemma 6.4 of \cite{ScholzePerfectoidSpaces1}: In the notation of that lemma, the subalgebras
\[
R^+\langle \left(\frac{f_1^\sharp}{g^\sharp}\right)^{1/p^m},\ldots,\left(\frac{f_n^\sharp}{g^\sharp}\right)^{1/p^m} \rangle\subset \mathcal{O}_X(U^\sharp)^+
\]
for $m\geq 0$ have the desired property.

Now assume that $V\rightarrow U$ is finite \'{e}tale. In that case, $S_j$ is a finite \'{e}tale $R_j$-algebra for all $j$, and the almost purity theorem \cite{ScholzePerfectoidSpaces1}, Theorem 7.9 (iii), shows that $T = S_j\otimes_{R_j} R$ is a perfectoid $K$-algebra. Let $T^+$ be the integral closure of $R^+$ in $T$. By \cite{ScholzePerfectoidSpaces1}, Lemma 7.3 (iv) and Proposition 7.10, $\Spa(T,T^+)$ is the fibre product $V_j\times_{U_j} \hat{U}$. Then the proof of parts (i) and (ii) goes through as before. For part (iii), it suffices to check that one can find subalgebras $T_\epsilon^+\subset T^+$ topologically finitely generated over $R^+$, such that the cokernel is annihilated by $p^\epsilon$. In this case, this follows from the stronger statement that $T^{+a} = T^{\circ a}$ is a uniformly almost finitely generated $R^{+ a} = R^{\circ a}$-module, cf. \cite{ScholzePerfectoidSpaces1}, Theorem 7.9 (iii).
\end{proof}

In particular, this implies that the functor $U\mapsto \hat{U}$ is compatible with open embeddings, and one can extend the functor to a functor $U\mapsto \hat{U}$ from perfectoid $U\in X_\proet$ to perfectoid spaces over $K$.

\begin{lem} Let $U\in X_\proet$ be perfectoid. For any $V\rightarrow U$ pro-\'{e}tale, also $V$ is perfectoid.
\end{lem}

\begin{proof} We may assume that $U$ is affinoid perfectoid, given as the inverse limit of $U_i = \Spa(R_i,R_i^+)$. Moreover, factor $V\rightarrow U$ as the composition $V\rightarrow V_0\rightarrow U$ of an inverse limit of finite \'{e}tale surjective maps $V=\varprojlim V_j\rightarrow V_0$ and an \'{e}tale map $V_0\rightarrow U$. The latter is locally the composite of a rational subset of a finite \'{e}tale cover. These cases have already been dealt with, so we can assume that $V_0=U$, i.e. $V$ is an inverse limit of finite \'{e}tale surjective maps.

For any $j$, $V_j$ comes as the pullback of some finite \'{e}tale cover $V_{ij} = \Spa(S_{ij},S_{ij}^+)$ of $U_i$. For any $j$, the completion $(S_j,S_j^+)$ of the direct limit over $i$ of the $(S_{ij},S_{ij}^+)$ is perfectoid affinoid by the previous lemma. It follows that the completion of the direct limit over $i$ and $j$ of $(S_{ij},S_{ij}^+)$ is the completion of the direct limit of the $(S_j,S_j^+)$. But the completion of a direct limit of perfectoid affinoid $(K,K^+)$-algebras is again perfectoid affinoid.
\end{proof}

\begin{cor} Assume that $X$ is smooth over $\Spa(K,K^+)$. Then the set of $U\in X_\proet$ which are affinoid perfectoid form a basis for the topology.
\end{cor}

\begin{proof} If $X=\mathbb{T}^n$, then we have constructed an explicit cover of $X$ by an affinoid perfectoid $\tilde{\mathbb{T}}^n\in X_\proet$. By the last lemma, anything pro-\'{e}tale over $\tilde{\mathbb{T}}^n$ again has a basis of affinoid perfectoid subspaces, giving the claim in this case: Any $U\in X_\proet$ is covered by $U\times_{\mathbb{T}^n} \tilde{\mathbb{T}}^n$, which is pro-\'{e}tale over $\tilde{\mathbb{T}}^n$.

In general, $X$ locally admits an \'{e}tale map to $\mathbb{T}^n$, cf. \cite{Huber}, Corollary 1.6.10, reducing us to this case.
\end{proof}

This corollary is all we will need, but a result of Colmez, \cite{ColmezBanach}, shows that the statement is true in full generality.

\begin{prop} Let $X$ be a locally noetherian adic space over $\Spa(K,K^+)$. Then the set of $U\in X_\proet$ which are affinoid perfectoid form a basis for the topology.
\end{prop}

\begin{proof} We may assume that $X=\Spa(A,A^+)$, where $A$ has no nontrivial idempotents. It is enough to find a sequence $A_i/A$ of finite \'{e}tale extensions, such that, denoting by $A_i^+$ the integral closure of $A^+$ in $A_i^+$, the completion $(B,B^+)$ of the direct limit of $(A_i,A_i^+)$ is perfectoid affinoid. Here, we put the $p$-adic topology on the direct limit of the $A_i^+$, even though all $A_i^+$ may not carry the $p$-adic topology. Using Proposition \ref{ProetChangeBase}, we may assume that $K$ is algebraically closed.

Now we follow the construction of Colmez, \cite{ColmezBanach}, \S 4.4. The construction is to iterate adjoining $p$-th roots of all $1$-units $1+A^{\circ\circ}$; here $A^{\circ\circ}\subset A^+$ denotes the subset of topologically nilpotent elements. Note that Colmez works in a setup which amounts to assuming that $A^+$ has the $p$-adic topology; however,
\[
(1+A^{\circ\circ})/(1+A^{\circ\circ})^p\cong (1+\hat{A}^{\circ\circ})/(1+\hat{A}^{\circ\circ})^p\ ,
\]
where $\hat{A}=\hat{A}^+[p^{-1}]$, with $\hat{A}^+$ the $p$-adic completion of $A^+$. This means that adjoining the $p$-th roots to $A$ first and then completing is the same as first completing and then adjoining the $p$-th roots. Colmez shows that repeating this construction produces a sympathetic $K$-algebra, and sympathetic $K$-algebras are perfectoid by \cite{ColmezBanach}, Lemma 1.15 (iii).
\end{proof}

Using the construction of the previous proposition, we prove the following theorem.

\begin{thm}\label{TheoremKPi1} Let $X=\Spa(A,A^+)$ be an affinoid connected noetherian adic space over $\Spa(\Q_p,\Z_p)$. Then $X$ is a $K(\pi,1)$ for $p$-torsion coefficients, i.e. for all $p$-torsion local systems $\mathbb{L}$ on $X_\et$, the natural map
\[
H^i_\cont(\pi_1(X,\bar{x}),\mathbb{L}_{\bar{x}})\to H^i(X_\et,\mathbb{L})
\]
is an isomorphism, where $\bar{x}\in X$ is a geometric point, and $\pi_1(X,\bar{x})$ denotes the profinite \'etale fundamental group.
\end{thm}

\begin{proof} We have to show that the natural map
\[
H^i(X_\fet,\mathbb{L})\to H^i(X_\et,\mathbb{L})
\]
is an isomorphism. For this, let $f:X_\et\rightarrow X_\fet$ be the natural map of sites; then it is enough to show that $R^if_\ast \mathbb{L} = 0$ for $i>0$, and $f_\ast \mathbb{L} = \mathbb{L}$. The second property is clear. It remains to show that for any $U\to X$, which we may assume to be connected, any cohomology class of $H^i(U_\et,\mathbb{L})$ can be killed by a finite \'etale cover of $U$. Renaming $U=X$, we have to show that any cohomology class of $H^i(X_\et,\mathbb{L})$ can be killed by a finite \'etale cover. For this, we can assume that $\mathbb{L}$ is trivial (by passing to the cover trivializing $\mathbb{L}$), and then that $\mathbb{L} = \mathbb{F}_p$.

Now we argue with the universal cover of $X$. Let $A_\infty$ be a direct limit of faithfully flat finite \'etale $A$-algebras $A_i\subset A_\infty$, such that $A_\infty$ has no nontrivial idempotents, and such that any faithfully flat finite \'etale $A_\infty$-algebra has a section. Let $A_\infty^+\subset A_\infty$ be the integral closure of $A^+$, and let $(\hat{A}_\infty,\hat{A}_\infty^+)$ be the completion of $(A_\infty,A_\infty^+)$, for the $p$-adic topology on $A_\infty^+$. Then $(\hat{A}_\infty,\hat{A}_\infty^+)$ is a perfectoid affinoid $\mathbb{C}_p$-algebra (which can either be easily checked directly, or deduced from the proof of the previous proposition and the almost purity theorem). Let $X_\infty = \Spa(\hat{A}_\infty,\hat{A}_\infty^+)$, which is a perfectoid space over $\mathbb{C}_p$. Moreover, $X_\infty\sim \varprojlim X_i$ in the sense of \cite[Definition 7.14]{ScholzePerfectoidSpaces1}, where $X_i=\Spa(A_i,A_i^+)\to X$ is finite \'etale. Then by \cite[Corollary 7.18]{ScholzePerfectoidSpaces1}, we have
\[
H^j(X_{\infty,\et},\mathbb{L}) = \varinjlim H^j(X_{i,\et},\mathbb{L})
\]
for all $j\geq 0$. We see that it is enough to prove that $H^j(X_{\infty,\et},\mathbb{L}) = 0$ for $j>0$.

For this, we argue with the tilt $X_\infty^\flat$ of $X_\infty$. Recall that we have reduced to the case $\mathbb{L} = \mathbb{F}_p$. We have the Artin -- Schreier sequence
\[
0\to \mathbb{F}_p\to \mathcal{O}_{X_\infty^\flat}\to \mathcal{O}_{X_\infty^\flat}\to 0
\]
on $X^\flat_{\infty,\et}\cong X_{\infty,\et}$. Taking cohomology, we see that $H^j(X_{\infty,\et},\mathbb{F}_p) = 0$ for $j\geq 2$, and
\[
0\to \mathbb{F}_p\to \hat{A}_\infty^\flat\to \hat{A}_\infty^\flat\to H^1(X_{\infty,\et},\mathbb{F}_p)\to 0
\]
However, as $A_\infty$ has no nontrivial finite \'etale covers, also $\hat{A}_\infty$ has no nontrivial finite \'etale covers (cf. e.g. \cite[Lemma 7.5 (i)]{ScholzePerfectoidSpaces1}), and thus $\hat{A}_\infty^\flat$ has no nontrivial finite \'etale covers, by \cite[Theorem 5.25]{ScholzePerfectoidSpaces1}. This implies that the Artin -- Schreier map $\hat{A}_\infty^\flat\to \hat{A}_\infty^\flat$ is surjective, giving $H^1(X_{\infty,\et},\mathbb{F}_p)=0$, as desired.
\end{proof}

\begin{lem}\label{ContOXplusonAffPerf} Assume that $U\in X_\proet$ is affinoid perfectoid, with $\hat{U} = \Spa(R,R^+)$. In the following, use the almost setting with respect to $K^+$ and the ideal of topologically nilpotent elements.
\begin{altenumerate}
\item[{\rm (i)}] For any nonzero $b\in K^+$, we have $\mathcal{O}_X^+(U)/b = R^+/b$, and this is almost equal to $(\mathcal{O}_X^+/b)(U)$.
\item[{\rm (ii)}] The image of $(\mathcal{O}_X^+/b_1)(U)$ in $(\mathcal{O}_X^+/b_2)(U)$ is equal to $R^+/b_2$ for any nonzero $b_1,b_2\in K^+$ such that $|b_1|<|b_2|$.
\item[{\rm (iii)}] We have $\hat{\mathcal{O}}_X^+(U) = R^+$, $\hat{\mathcal{O}}_X(U) = R$.
\item[{\rm (iv)}] The ring $\hat{\mathcal{O}}_X^+(U)$ is the $p$-adic completion of $\mathcal{O}_X^+(U)$.
\item[{\rm (v)}] The cohomology groups $H^i(U,\hat{\mathcal{O}}_X^+)$ are almost zero for $i>0$.
\end{altenumerate}

In particular, for any perfectoid $U\in X_\proet$, $(|U|,\hat{\mathcal{O}}_X|_{|U|},(|\bullet(x)|\mid x\in |U|))$ is an adic, in fact perfectoid, space, given by $\hat{U}$.
\end{lem}

\begin{rem} This gives the promised base of the topology on which the sheaves $\hat{\mathcal{O}}_X^+$ and $\hat{\mathcal{O}}_X$ behave as expected. Note that by Proposition \ref{ProetChangeBase}, such a base of the topology exists for all locally noetherian adic spaces over $\Spa(\Q_p,\Z_p)$.
\end{rem}

\begin{proof}\begin{altenumerate}
\item[{\rm (i)}] The equality $\mathcal{O}_X^+(U)/b = R^+/b$ follows from the definition of $(R,R^+)$. By the last proposition, giving a sheaf on $X_\proet$ is equivalent to giving a presheaf on the set of affinoid perfectoid $U\in X_\proet$, satisfying the sheaf property for pro-\'{e}tale coverings by such. We claim that
\[
U\mapsto \mathcal{F}(U) = (\mathcal{O}_{\hat{U}}^+(\hat{U})/b)^a=(\mathcal{O}_X^+(U)/b)^a
\]
is a sheaf $\mathcal{F}$ of almost $K^+$-algebras, with $H^i(U,\mathcal{F})=0$ for $i>0$. Indeed, let $U$ be covered by $V_k\rightarrow U$. We may assume that each $V_k$ is pro-finite \'{e}tale over $V_{k0}\rightarrow U$, and that $V_{k0}\rightarrow U$ factors as a composite of rational embeddings and finite \'{e}tale maps. By quasicompactness of $U$, we can assume that there are only finitely many $V_k$, or just one $V$ by taking the union. Then $V=\varprojlim V_j\rightarrow V_{j_0}\rightarrow U$, where $V_{j_0}$ is a composite of rational embeddings and finite \'{e}tale maps, and $V_j\rightarrow V_{j^\prime}$ is finite \'{e}tale surjective for $j,j^\prime\geq j_0$. We have to see that the complex
\[
\mathcal{C}(U,V): 0\rightarrow \mathcal{F}(U)\rightarrow \mathcal{F}(V)\rightarrow \mathcal{F}(V\times_U V)\rightarrow \ldots
\]
is exact. Now $\mathcal{F}(V) = \varinjlim \mathcal{F}(V_j)$ etc., so that
\[
\mathcal{C}(U,V) = \varinjlim \mathcal{C}(U,V_j)\ ,
\]
and one reduces to the case that $V\rightarrow U$ is a composite of rational embeddings and finite \'{e}tale maps. In that case, $V$ and $U$ are affinoid perfectoid, giving rise to perfectoid spaces $\hat{U}$ and $\hat{V}$, and an \'{e}tale cover $\hat{V}\rightarrow \hat{U}$. Then Lemma \ref{PreciseLocalStructure} implies that
\[
\mathcal{C}(U,V): 0\rightarrow (\mathcal{O}_{\hat{U}_\et}^+(\hat{U})/b)^a\rightarrow (\mathcal{O}_{\hat{U}_\et}^+(\hat{V})/b)^a\rightarrow (\mathcal{O}_{\hat{U}_\et}^+(\hat{V}\times_{\hat{U}} \hat{V})/b)^a\rightarrow \ldots\ ,
\]
so the statement follows from the vanishing of $H^i(W_\et,\OO_{W_\et}^{+ a})$, $i>0$, for any affinoid perfectoid space $W$, proved in \cite{ScholzePerfectoidSpaces1}, Proposition 7.13.

This shows in particular that $\mathcal{F} = (\mathcal{O}_X^+/b)^a$, giving part (i).

\item[{\rm (ii)}] Define $c=\frac{b_1}{b_2}\in K^+$, with $|c|<1$. Let $f\in (\mathcal{O}_X^+/b_1)(U)$, and take $g\in R^+$ such that $cf=g$ in $(\OO_X^+/b_1)(U)$, which exists by part (i). Looking at valuations, one finds that $g=ch$ for some $h\in R^+$. As multiplication by $c$ induces an injection $\OO_X^+/b_2\rightarrow \OO_X^+/b_1$, it follows that $f-h=0$ in $(\OO_X^+/b_2)(U)$. Hence the image of $f$ in $(\OO_X^+/b_2)(U)$ is in the image of $R^+/b_2$, as desired.

\item[{\rm (iii)}] Using part (ii), one sees that
\[
\hat{\mathcal{O}}_X^+(U) = \varprojlim (\mathcal{O}_X^+/p^n)(U) = \varprojlim R^+/p^n = R^+\ .
\]
Inverting $p$ gives $\hat{\mathcal{O}}_X(U) = R$.

\item[{\rm (iv)}] This follows from part (iii) and the definition of $R^+$.

\item[{\rm (v)}] We have checked in the proof of part (i) that the cohomology groups $H^i(U,\mathcal{O}_X^+/p^n)$ are almost zero for $i>0$ and all $n$. Now it follows for $\hat{\mathcal{O}}_X^+$ from the almost version of Lemma \ref{InverseLimitExact}.
\end{altenumerate}
\end{proof}

\begin{lem}\label{TwistedOXPlus} In the situation of the previous lemma, assume that $\mathbb{L}$ is an $\mathbb{F}_p$-local system on $X_\et$. Then for all $i>0$, the cohomology group
\[
H^i(U,\mathbb{L}\otimes \OO_X^+/p)^a = 0
\]
is almost zero, and it is an almost finitely generated projective $R^{+a}/p$-module $M(U)$ for $i=0$. If $U^\prime\in X_\proet$ is affinoid perfectoid, corresponding to $\hat{U^\prime} = \Spa(R^\prime,R^{\prime +})$, and $U^\prime\rightarrow U$ some map in $X_\proet$, then $M(U^\prime) = M(U)\otimes_{R^{+a}/p} R^{\prime +a}/p$.
\end{lem}

\begin{proof} We may assume that $X$ is connected; in particular, one can trivialize $\mathbb{L}$ to a constant sheaf $\mathbb{F}_p^k$ after a surjective finite \'{e}tale Galois cover. Let $V\rightarrow U$ be such a surjective finite \'{e}tale Galois cover trivializing $\mathbb{L}$, with Galois group $G$, and let $\hat{V}=\Spa(S,S^+)$. Then $\mathbb{L}\cong \mathbb{F}_p^k$ over $V$. Let $V^{j/U}$ be the $j$-fold fibre product of $V$ over $U$, for $j\geq 1$. Then the previous lemma implies that $H^i(V^{j/U},\mathbb{L}\otimes \OO_X^+/p)$ is almost zero for $i>0$, and almost equal to $(S_j^+/p)^k$ for $i=0$, where $S_j^+$ is the $j$-fold tensor product of $S^+$ over $R^+$. But $S^{+a}/p$ is almost finite \'{e}tale over $R^{+a}/p$ by the almost purity theorem, and then faithfully flat (as $V\rightarrow U$ is surjective), so the result follows from faithfully flat descent in the almost setting, cf. Section 3.4 of \cite{GabberRamero}.
\end{proof}

\section{Finiteness of \'{e}tale cohomology}\label{FinitenessSection}

In this section, we prove the following result.

\begin{thm}\label{Finiteness} Let $K$ be an algebraically closed complete extension of $\Q_p$ with an open and bounded valuation subring $K^+\subset K$, let $X$ be a proper smooth adic space over $\Spa(K,K^+)$, and let $\mathbb{L}$ be an $\mathbb{F}_p$-local system on $X_\et$. Then $H^i(X_\et,\mathbb{L})$ is a finite-dimensional $\mathbb{F}_p$-vector space for all $i\geq 0$, which vanishes for $i>2\dim X$. Moreover, there is an isomorphism of almost $K^+$-modules for all $i\geq 0$,
\[
H^i(X_\et,\mathbb{L})\otimes K^{+a}/p\cong H^i(X_\et,\mathbb{L}\otimes \mathcal{O}_X^{+a}/p)\ .
\]
\end{thm}

We start with some lemmas. Here and in the following, for any nonarchimedean field $K$, we denote by $\OO_K=K^\circ\subset K$ its ring of integers.

\begin{lem} Let $K$ be a complete nonarchimedean field, let $V$ be an affinoid smooth adic space over $\Spa(K,\OO_K)$, and let $x\in V$ with closure $M=\overline{\{x\}}\subset V$. Then there exists a rational subset $U\subset V$ containing $M$, together with an \'{e}tale map $U\rightarrow \mathbb{T}^n$ which factors as a composite of rational embeddings and finite \'{e}tale maps.
\end{lem}

\begin{proof} We note that $M$ is the intersection of all rational subsets $U\subset V$ that contain $M$.

Now, first, one may replace $V$ by a rational subset such that there exists an \'{e}tale map $f: V\rightarrow \mathbb{B}^n$, where $\mathbb{B}^n$ denotes the $n$-dimensional unit ball. This follows from Corollary 1.6.10 of \cite{Huber}, once one observes that the open subset constructed there may be assumed to contain $M$. Let $y=f(x)\in \mathbb{B}^n$, with closure $N\subset \mathbb{B}^n$. Then by Lemma 2.2.8 of \cite{Huber} (with similar analysis of its proof) one may find a rational subset $W\subset \mathbb{B}^n$ containing $N$, such that $f^{-1}(W)\rightarrow W$ factors as an open embedding $f^{-1}(W) \rightarrow Z$ and a finite \'{e}tale map $Z\rightarrow W$. Note that $M\subset f^{-1}(W)$. Choose some open subset $U\subset f^{-1}(W)$ containing $M$, such that $U$ is rational in $Z$ (and hence in $f^{-1}(W)$). Then $U\subset f^{-1}(W)\subset V$ is a rational subset, and $U\rightarrow W$ factors as a composite of a rational embedding and a finite \'{e}tale map. Finally, embed $\mathbb{B}^n\rightarrow \mathbb{T}^n$ as a rational subset, e.g. as the locus where $|T_i-1|\leq |p|$ for all $i=1,\ldots,n$. Then $W\subset \mathbb{B}^n\subset \mathbb{T}^n$ is a rational subset, so that $U\rightarrow \mathbb{T}^n$ gives the desired \'{e}tale map. 
\end{proof}

\begin{lem}\label{ChoiceOfCovers} Let $K$ be a complete nonarchimedean field and let $X$ be a proper smooth adic space over $\Spa(K,\OO_K)$. For any integer $N\geq 1$, one may find $N$ finite covers $V_i^{(1)},\ldots,V_i^{(N)}$ of $X$ by affinoid open subsets, such that the following conditions are satisfied.
\begin{altenumerate}
\item[{\rm (i)}] For all $i$, $k=1,\ldots,N-1$, the closure $\overline{V}_i^{(k+1)}$ of $V_i^{(k+1)}$ in $X$ is contained in $V_i^{(k)}$.
\item[{\rm (ii)}] For all $i$, $V_i^{(N)}\subset \ldots\subset V_i^{(1)}$ is a chain of rational subsets.
\item[{\rm (iii)}] For all $i$, $j$, the intersection $V_i^{(1)}\cap V_j^{(1)}\subset V_i^{(1)}$ is a rational subset.
\item[{\rm (iv)}] For all $i$, there is an \'{e}tale map $V_i^{(1)}\rightarrow \mathbb{T}^n$ that factors as a composite of rational subsets and finite \'{e}tale maps.
\end{altenumerate}
\end{lem}

\begin{proof} For any $x\in X$, there is some affinoid open subset $V\subset X$ such that $\overline{\{x\}}\subset V$, by \cite{Temkin}. Inside $V$, we may find a rational subset $U\subset V$ such that $\overline{\{x\}}\subset U$ and the closure $\overline{U}$ of $U$ in $X$ is contained in $V$. Taken together, the $U$'s cover $X$, so we may find finite covers $U_i^0, V_i^0\subset X$ by affinoid open subsets such that for all $x\in X$, there is some $i$ for which $\overline{\{x\}}\subset U_i^0$, and such that $\overline{U}_i^0\subset V_i^0$, with $U_i^0\subset V_i^0$ a rational subset.

Next, we may find for any $x\in X$ an affinoid open subset $V\subset X$ such that $\overline{\{x\}}\subset V$, such that $V\cap U_i^0\subset U_i^0$ is a rational subset for all $i$, and such that $V\subset U_i^0$ for one $i$. It suffices to do this for maximal points $x$. Now, let $I_x = \{i | \overline{\{x\}}\subset V_i^0\}$. There is some open affinoid subset $W\subset X$ containing $\overline{\{x\}}$ such that $W\cap U_i^0 = \emptyset$ if $i\not\in I_x$, and $W\subset V_i^0$ for $i\in I_x$. For each $i\in I_x$, choose some rational subset $W_i\subset V_i^0$ such that $\overline{\{x\}}\subset W_i\subset W$, and let $V^\prime$ be the intersection of all $W_i$ for $i\in I_x$. There is some $i\in I_x$ such that $\overline{\{x\}}\subset U_i^0$; we set $V=V^\prime\cap U_i^0$, which is a rational subset of $V^\prime$. Then $V\subset V_i^0$ is rational for all $i\in I_x$, hence $V\cap U_i^0\subset U_i^0$ is rational. If $i\not\in I_x$, then $V\cap U_i^0\subset W\cap U_i^0=\emptyset$, which is also a rational subset of $U_i^0$. Finally, $V\subset U_i^0$ for one $i$.

Replacing $V$ by a further rational subset, we may assume that there is an \'{e}tale map $V\rightarrow \mathbb{T}^n$ that factors as a composite of rational embeddings and finite \'{e}tale maps. Also, we may find a chain of rational subsets $V^{(N)}\subset \ldots \subset V=V^{(1)}$ such that $\overline{\{x\}}\subset V^{(N)}$ and $\overline{V}^{(k+1)}\subset V^{(k)}$ for $k=1,\ldots,N-1$. Then the $V^{(N)}$'s cover $X$, so we may find finite covers $V_j^{(N)},\ldots,V_j^{(1)}\subset X$ by affinoid open subsets such that (i), (ii) and (iv) are satisfied. In fact, (iii) is satisfied as well: For each $j$, there is some $i$ such that $V_j^{(1)}\subset U_i^0$. Then for all $j^\prime$, $V_j^{(1)}\cap V_{j^\prime}^{(1)} = V_j^{(1)}\cap (U_i^0\cap V_{j^\prime}^{(1)})$. Now, $U_i^0\cap V_{j^\prime}^{(1)}\subset U_i^0$ is a rational subset by construction of the $V^{(1)}$'s. It follows that $V_j^{(1)}\cap (U_i^0\cap V_{j^\prime}^{(1)})\subset V_j^{(1)}$ is a rational subset as well, as desired.
\end{proof}

Let us record the following lemma on tracing finiteness results on images of maps through spectral sequences.

\begin{lem}\label{SpecSeqFiniteness} Let $K$ be a nondiscretely valued complete nonarchimedean field, and let $E_{\ast,(i)}^{p,q}\Rightarrow M_{(i)}^{p+q}$, $i=1,...,N$, be upper-right quadrant spectral sequences of almost $\OO_K$-modules, together with maps of spectral sequences $E_{\ast,(i)}^{p,q}\rightarrow E_{\ast,(i+1)}^{p,q}$, $M_{(i)}^{p+q}\rightarrow M_{(i+1)}^{p+q}$ for $i=1,...,N-1$. Assume that for some $r$, the image on the $r$-th sheet, $E_{r,(i)}^{p,q}\rightarrow E_{r,(i+1)}^{p,q}$ is almost finitely generated over $\OO_K$ for all $i$, $p$, $q$. Then the image of $M_{(1)}^k$ in $M_{(N)}^k$ is an almost finitely generated $\OO_K$-module for $k\leq N-2$.
\end{lem}

\begin{proof} Fix $k\leq N-2$. Each spectral sequence defines the decreasing separated and exhaustive abutment filtration $\Fil_{(i)}^p$ on $M_{(i)}^k$, such that
\[
\Fil_{(i)}^p / \Fil_{(i)}^{p+1} = E_{\infty,(i)}^{p,k-p}\ .
\]
In particular, $\Fil_{(i)}^0 = M_{(i)}^k$ and $\Fil_{(i)}^{k+1} = 0$. We note that the existence of the maps of spectral sequences means that $\Fil_{(i)}^p$ maps into $\Fil_{(i+1)}^p$ for all $p$, and the induced maps
\[
E_{\infty,(i)}^{p,k-p} = \Fil_{(i)}^p / \Fil_{(i)}^{p+1}\rightarrow \Fil_{(i+1)}^p / \Fil_{(i+1)}^{p+1} = E_{\infty,(i+1)}^{p,k-p}
\]
agree with the map on the spectral sequence.

By induction on $i=1,\ldots,k+2$, we claim that the image of $M_{(1)}^k$ in $M_{(i)}^k / \Fil_{(i)}^{i-1}$ is an almost finitely generated $\OO_K$-module. For $i=1$, there is nothing to show. Now assume that the image of $M_{(1)}^k$ in $M_{(i)}^k / \Fil_{(i)}^{i-1}$ is an almost finitely generated $\OO_K$-module. There is some $r$ such that the image of $E_{r,(i)}^{p,q}\rightarrow E_{r,(i+1)}^{p,q}$ is an almost finitely generated $\OO_K$-module for all $p$, $q$. It follows that the same stays true on the $E_\infty$-page, so that in particular the image of
\[
E_{\infty,(i)}^{i-1,k+1-i}\rightarrow E_{\infty,(i+1)}^{i-1,k+1-i}
\]
is almost finitely generated. Under the identification of these terms with the abutment filtration, this image is precisely
\[
\Fil_{(i)}^{i-1} / (\Fil_{(i)}^{i-1}\cap \Fil_{(i+1)}^i)\ .
\]
Now we use the exact sequence
\[\begin{aligned}
0\rightarrow (M_{(1)}^k\cap \Fil_{(i)}^{i-1}) / (M_{(1)}^k\cap \Fil_{(i)}^{i-1}\cap \Fil_{(i+1)}^i)&\rightarrow M_{(1)}^k / (M_{(1)}^k\cap \Fil_{(i+1)}^i)\\
&\rightarrow M_{(1)}^k / (M_{(1)}^k\cap (\Fil_{(i)}^{i-1} + \Fil_{(i+1)}^i))\rightarrow 0\ ,
\end{aligned}\]
where the intersection $M_{(1)}^k\cap F$ for $F\subset M_{(i)}^k$ means taking those elements of $M_{(1)}^k$ whose image in $M_{(i)}^k$ lies in $F$. The right-most term is almost finitely generated by induction, and we have just seen that the left-most term is almost finitely generated. The middle term is isomorphic to the image of $M_{(1)}^k$ in $M_{(i+1)}^k / \Fil_{(i+1)}^i$, so we get the claim.

Now we use this information for $i=k+2$, where it says that the image of $M_{(1)}^k$ in $M_{(k+2)}^k = M_{(k+2)}^k / \Fil_{(k+2)}^{k+1}$ is almost finitely generated. In particular, the same stays true for the image in $M_{(N)}^k$, as desired.
\end{proof}

\begin{lem}\label{BasicLocalComputation} Let $K$ be a complete nonarchimedean field of characteristic $0$ that contains all $p$-power roots of unity; choose a compatible system $\zeta_{p^\ell}\in K$ of $p^\ell$-th roots of unity. Let $R_0=\OO_K\langle T_1^{\pm 1},\ldots,T_n^{\pm 1}\rangle$, and $R=\OO_K\langle T_1^{\pm 1/p^\infty},\ldots,T_n^{\pm 1/p^\infty}\rangle$. Let $\Z_p^n$ act on $R$, such that the $k$-th basis vector $\gamma_k\in \Z_p^n$ acts on the basis via
\[
T_1^{i_1}\cdots T_n^{i_n}\mapsto \zeta^{i_k} T_1^{i_1}\cdots T_n^{i_n}\ ,
\]
where $\zeta^{i_k} = \zeta_{p^\ell}^{i_k p^\ell}$ whenever $i_k p^\ell\in \Z$. Then $H^q_\cont(\Z_p^n,R/p^m)$ is an almost finitely presented $R_0$-module for all $m$, and the map
\[
\bigwedge^q R_0^n = H^q_\cont(\Z_p^n,R_0)\rightarrow H^q_\cont(\Z_p^n,R)
\]
is injective with cokernel killed by $\zeta_p-1$.

Moreover, if $S_0$ is a $p$-adically complete flat $\Z_p$-algebra with the $p$-adic topology, then
\[
H^q_\cont(\Z_p^n,S_0/p^m\otimes_{R_0/p^m} R/p^m) = S_0/p^m\otimes_{R_0/p^m} H^q_\cont(\Z_p^n,R/p^m)
\]
for all $m$, and
\[
H^q_\cont(\Z_p^n,S_0\hat{\otimes}_{R_0} R) = S_0\hat{\otimes}_{R_0} H^q_\cont(\Z_p^n,R)\ .
\]
\end{lem}

\begin{proof} Recall that in general, if $M$ is a topological $\Z_p^n$-module such that $M=\varprojlim M/p^m$, with $M$ carrying the inverse limit topology of the discrete topologies on $M/p^m$, then continuous $\Z_p^n$-cohomology with values in $M$ is computed by the Koszul complex
\[
0\rightarrow M \rightarrow M^n \rightarrow \ldots \rightarrow \bigwedge^q M^n\rightarrow \ldots \rightarrow M^n \rightarrow M \rightarrow 0\ ,
\]
where the first map is $(\gamma_1-1,\ldots,\gamma_n-1)$. To check this, consider the Iwasawa algebra $\Lambda = \Z_p[[\Z_p^n]]\cong \Z_p[[T_1,\ldots,T_n]]$, with $T_i$ corresponding to $\gamma_i - 1$, use the Koszul complex
\[
0\rightarrow \Lambda\rightarrow \Lambda^n\rightarrow \ldots \rightarrow \bigwedge^q \Lambda^n\rightarrow \ldots \rightarrow \Lambda^n \rightarrow \Lambda\rightarrow 0
\]
associated to $(T_1,\ldots,T_n)$, which resolves $\Z_p$. Now take $\Hom_\cont(-,M)$, which gives
\[
0\rightarrow \Hom_\cont(\Z_p^n,M)\rightarrow \Hom_\cont(\Z_p^n,M)^n\rightarrow \ldots\ ,
\]
and resolves $M$ as a topological $\Z_p^n$-module. Then taking continuous $\Z_p^n$-cohomology gives the result.

Let us compute $H^q(\Z_p^n,R/p^m)$ for all $m$. It is the direct sum of
\[
H^q(\Z_p^n,R_0/p^m \cdot T_1^{i_1}\cdots T_n^{i_n})
\]
over $i_1,\ldots,i_n\in [0,1)\cap \Z[p^{-1}]$. This is computed by the tensor product of the complexes
\[
0\rightarrow R_0/p^m \buildrel{\zeta^{i_k}-1}\over\rightarrow R_0/p^m\rightarrow 0
\]
over $k=1,\ldots,n$. If $i_k\neq 0$, the cohomology of $0\rightarrow R_0/p^m\buildrel{\zeta^{i_k}-1}\over\rightarrow R_0/p^m\rightarrow 0$ is annihilated by $\zeta_p - 1$. It follows that
\[
\bigwedge^q (R_0/p^m)^n = H^q(\Z_p^n,R_0/p^m)\rightarrow H^q(\Z_p^n,R/p^m)
\]
is injective with cokernel killed by $\zeta_p - 1$, for all $m$. Taking the inverse limit over $m$, we get the statement about $H^q_\cont(\Z_p^n,R)$.

More precisely, if $i_k$ has denominator $p^\ell$, then the cohomology of $0\rightarrow R_0/p^m\buildrel{\zeta^{i_k}-1}\over\rightarrow R_0/p^m\rightarrow 0$ is annihilated by $\zeta_{p^\ell}-1$. If $\epsilon>0$, $\epsilon\in \log\Gamma$, then $\zeta_{p^\ell}-1$ divides $p^\epsilon$ for almost all $\ell$, hence only finitely many $n$-tuples $(i_1,\ldots,i_n)$ contribute cohomology which is not $p^\epsilon$-torsion. The cohomology for each such tuple is finitely presented, hence the cohomology group $H^q(\Z_p^n,R/p^m)$ is almost finitely presented.

The compatibility with base-change is immediate from the calculations.
\end{proof}

\begin{lem}\label{LocalComputation} Let $K$ be a perfectoid field of characteristic $0$ containing all $p$-power roots of unity. Let $V$ be an affinoid smooth adic space over $\Spa(K,\OO_K)$ with an \'{e}tale map $V\rightarrow \mathbb{T}^n$ that factors as a composite of rational embeddings and finite \'{e}tale maps. Let $\mathbb{L}$ be an $\mathbb{F}_p$-local system on $V_\et$.
\begin{altenumerate}
\item[{\rm (i)}] For $i>n=\dim V$, the cohomology group
\[
H^i(V_\et,\mathbb{L}\otimes \OO_V^+/p)
\]
is almost zero as $\OO_K$-module.
\item[{\rm (ii)}] Assume that $V^\prime\subset V$ is a rational subset such that $V^\prime$ is strictly contained in $V$. Then the image of
\[
H^i(V_\et,\mathbb{L}\otimes \OO_V^+/p)\rightarrow H^i(V^\prime_\et,\mathbb{L}\otimes \OO_V^+/p)
\]
is an almost finitely generated $\OO_K$-module.
\end{altenumerate}
\end{lem}

\begin{rem} It is probably true that $H^i(V_\et,\mathbb{L}\otimes \OO_V^+/p)$ is an almost finitely generated $\OO_V^+(V)/p$-module. If one assumes that $K$ is algebraically closed, then using Theorem \ref{BGRStuff}, the arguments of the following proof would show this result if $\OO_K/p[T_1,\ldots,T_n]$ is 'almost noetherian' for all $n$; the problem occurs in the Hochschild-Serre spectral sequence, where certain subquotients have to be understood.
\end{rem}

\begin{proof} We may use the pro-\'{e}tale site to compute these cohomology groups. Let $\tilde{V} = V\times_{\mathbb{T}^n} \tilde{\mathbb{T}}^n\in V_\proet$. Let $\tilde{V}^{j/V}$ be the $j$-fold fibre product of $\tilde{V}$ over $V$, for $j\geq 1$. Recall that the category underlying $\Spa(K,\OO_K)_\profet$ contains the category of profinite sets (with trivial Galois action); in particular, we can make sense of $\tilde{V}\times \Z_p^{n(j-1)}\in V_\proet$, by considering $\Z_p^{n(j-1)}$ as an object of $\Spa(K,\OO_K)_\profet$, pulled back to $V_\proet$. As $\tilde{V}\rightarrow V$ is a Galois cover with Galois group $\Z_p^n$, we see that $\tilde{V}^{j/V}\cong \tilde{V}\times \Z_p^{n(j-1)}$. Then by Lemma \ref{CompEtProet},
\[
H^i(\tilde{V}^{j/V},\mathbb{L}\otimes \OO_V^+/p)\cong \Hom_\cont(\Z_p^{n(j-1)},H^i(\tilde{V},\mathbb{L}\otimes \OO_V^+/p))
\]
for all $i\geq 0$, $j\geq 1$. But Lemma \ref{TwistedOXPlus} implies that
\[
H^i(\tilde{V},\mathbb{L}\otimes \OO_V^+/p)^a = 0
\]
for $i>0$, and is an almost finitely generated projective $S^{+a}/p$-module $M$ for $i=0$, where $\hat{\tilde{V}} = \Spa(S,S^+)$. Taken together, the Cartan-Leray spectral sequence shows that
\[
H^i(V_\proet,\mathbb{L}\otimes \OO_V^+/p)^a\cong H^i_\cont(\Z_p^n,M)\ .
\]
As $\Z_p^n$ has cohomological dimension $n$, we get part (i).

For part (ii), we have to check that the image of
\[
H^i_\cont(\Z_p^n,M)\rightarrow H^i_\cont(\Z_p^n,M\otimes_{S^{+a}/p} S^{\prime+a}/p)
\]
is an almost finitely generated $\OO_K$-module, where $(S^\prime,S^{\prime +})$ is defined as $(S,S^+)$, with $V^\prime$ in place of $V$. Recall that $M$ is an almost finitely generated projective $S^{+a}/p$-module. By Lemma 2.4.15 of \cite{GabberRamero}, this means that for any $\epsilon>0$, $\epsilon\in \log \Gamma$, there exists some $k$ and maps $f_\epsilon: M\rightarrow (S^{+a}/p)^k$, $g_\epsilon: (S^{+a}/p)^k\rightarrow M$ with $g_\epsilon f_\epsilon = p^\epsilon$. One sees that in order to prove (ii), we may replace $M$ by $S^{+a}/p$, i.e. we have to show that the image of
\[
H^i_\cont(\Z_p^n,S^{+a}/p)\rightarrow H^i_\cont(\Z_p^n,S^{\prime+a}/p)
\]
is an almost finitely generated $\OO_K$-module.

Now, choose $N=n+2$ rational subsets $V^{(N)} = V^\prime\subset \ldots\subset V^{(1)} = V$, such that $V^{(j+1)}$ is strictly contained in $V^{(j)}$ for $j=1,\ldots,N-1$. Let $\tilde{V}^{(j)}$ and $(S^{(j)},S^{(j)+})$ be defined as for $V$, using $V^{(j)}$ in place of $V$. We need to show that the image of
\[
H^i_\cont(\Z_p^n,S^{(1)+}/p)\rightarrow H^i_\cont(\Z_p^n,S^{(N)+}/p)
\]
is almost finitely generated over $\OO_K$. Now we use Lemma \ref{PreciseLocalStructure}, applied to $X=\mathbb{T}^n$, $U=\tilde{\mathbb{T}}^n$, giving rise to $\hat{U} = \Spa(R,R^+)$, where
\[
R^+ = \OO_K\langle T_1^{\pm 1/p^\infty},\ldots,T_n^{\pm 1/p^\infty}\rangle\ .
\]
Also, in the notation of that lemma, let $U_m = \Spa(R_m,R_m^+)$ with
\[
R_m^+ = \OO_K\langle T_1^{\pm 1/p^m},\ldots,T_n^{\pm 1/p^m}\rangle\ ,
\]
giving rise to $V_m^{(j)}=V^{(j)}\times_X U_m = \Spa(S_m^{(j)},S_m^{(j)+})$. By Lemma \ref{PreciseLocalStructure}, it is enough to show that the image of
\[
H^i_\cont(\Z_p^n,(S_m^{(1)+}\otimes_{R_m^+} R^+)/p)\rightarrow H^i_\cont(\Z_p^n,(S_m^{(N)+}\otimes_{R_m^+} R^+)/p)
\]
is almost finitely generated over $\OO_K$ for all $m$. These groups can be computed via the Hochschild-Serre spectral sequence
\[
H^{i_1}((\Z/p^m\Z)^n,H^{i_2}_\cont((p^m\Z_p)^n,(S_m^{(j)+}\otimes_{R_m^+} R^+)/p))\Rightarrow H^{i_1+i_2}_\cont(\Z_p^n,(S_m^{(j)+}\otimes_{R_m^+} R^+)/p)\ ,
\]
as the coefficients carry the discrete topology. But by Lemma \ref{BasicLocalComputation},
\[
H^i_\cont((p^m\Z_p)^n,(S_m^{(j)+}\otimes_{R_m^+} R^+)/p) = S_m^{(j)+}/p\otimes_{R_m^+/p} H^i_\cont((p^m\Z_p)^n,R^+/p)\ .
\]
Now as $N=n+2\geq i+2$, Lemma \ref{SpecSeqFiniteness} shows that it is enough to prove that the image of
\[
S_m^{(j)+}/p\otimes_{R_m^+/p} H^i_\cont((p^m\Z_p)^n,R^+/p)\rightarrow S_m^{(j+1)+}/p\otimes_{R_m^+/p} H^i_\cont((p^m\Z_p)^n,R^+/p)
\]
is almost finitely generated for $j=1,\ldots,N-1$. The image $A$ of $S_m^{(j)+}/p\rightarrow S_m^{(j+1)+}/p$ is almost finitely generated over $\OO_K$: As $V^{(j+1)}_m$ is strictly contained in $V^{(j)}_m$, the map $S_m^{(j)}\rightarrow S_m^{(j+1)}$ is completely continuous, and hence one easily checks that $A$ is a subquotient of a bounded subset of a finite-dimensional $K$-vector space. Such $\OO_K$-modules are almost finitely generated as $\OO_K$ is almost noetherian. Then the image of
\[
S_m^{(j)+}/p\otimes_{R_m^+/p} H^i_\cont((p^m\Z_p)^n,R^+/p)\rightarrow S_m^{(j+1)+}/p\otimes_{R_m^+/p} H^i_\cont((p^m\Z_p)^n,R^+/p)
\]
is a quotient of $A\otimes_{R_m^+/p} H^i_\cont((p^m\Z_p)^n,R^+/p)$. The group $H^i_\cont((p^m\Z_p)^n,R^+/p)$ is almost finitely generated over $R_m^+$ by Lemma \ref{BasicLocalComputation}. Choosing a map
\[
(R_m^+/p)^{N(\epsilon)}\rightarrow H^i_\cont((p^m\Z_p)^n,R^+/p)
\]
with cokernel annihilated by $p^\epsilon$, we find a map
\[
A^{N(\epsilon)}\rightarrow A\otimes_{R_m^+/p} H^i_\cont((p^m\Z_p)^n,R^+/p)
\]
with cokernel annihilated by $p^\epsilon$. It follows that $A\otimes_{R_m^+/p} H^i_\cont((p^m\Z_p)^n,R^+/p)$ is almost finitely generated over $\OO_K$, giving the claim.
\end{proof}

Now we can prove the crucial statement.

\begin{lem}\label{LemmaFiniteness} Let $K$ be a perfectoid field of characteristic $0$ containing all $p$-power roots of unity. Let $X$ be a proper smooth adic space over $\Spa(K,\OO_K)$, and let $\mathbb{L}$ be an $\mathbb{F}_p$-local system on $X_\et$. Then
\[
H^j(X_\et,\mathbb{L}\otimes \OO_X^+/p)
\]
is an almost finitely generated $\OO_K$-module, which is almost zero for $j>2\dim X$.
\end{lem}

\begin{proof} Let $X_\an$ be the site of open subsets of $X$. Lemma \ref{LocalComputation} shows that under the projection $\lambda: X_\et\rightarrow X_\an$, $R^j\lambda_\ast(\mathbb{L}\otimes \OO_X^+/p)$ is almost zero for $j>\dim X$. As the cohomological dimension of $X_\an$ is $\leq\dim X$ by Proposition 2.5.8 of \cite{deJongvanderPut}, we get the desired vanishing result.

To see that $H^j(X_\et,\mathbb{L}\otimes \OO_X^+/p)$ is an almost finitely generated $\OO_K$-module, choose $N=j+2$ covers $V_i^{(j+2)},\ldots,V_i^{(1)}\subset X$ as in Lemma \ref{ChoiceOfCovers}. Let $I$ be the finite index set. For any nonempty subset $J\subset I$, let $V_J^{(k)} = \cap_{i\in J} V_i^{(k)}$. Then the conditions of Lemma \ref{ChoiceOfCovers} ensure that each $V_J^{(k)}$ admits an \'{e}tale map $V_J^{(k)}\rightarrow \mathbb{T}^n$ that factors as a composite of rational embeddings and finite \'{e}tale maps. For each $k=1,\ldots,j+2$, we get a spectral sequence
\[
E_{1,(k)}^{m_1,m_2} = \bigoplus_{|J|=m_1+1} H^{m_2}(V_{J\et}^{(k)},\mathbb{L}\otimes \OO_X^+/p)\Rightarrow H^{m_1+m_2}(X_\et,\mathbb{L}\otimes \OO_X^+/p)\ ,
\]
together with maps between these spectral sequences $E_{\ast,(k)}^{m_1,m_2}\rightarrow E_{\ast,(k+1)}^{m_1,m_2}$ for $k=1,\ldots,j+1$. Then Lemma \ref{SpecSeqFiniteness} combined with Lemma \ref{LocalComputation} (ii) shows the desired finiteness.
\end{proof}

To finish the proof, we have to introduce the 'tilted' structure sheaf.

\begin{definition} Let $X$ be a locally noetherian adic space over $\Spa(\Q_p,\Z_p)$. The tilted integral structure sheaf is $\hat{\OO}_{X^\flat}^+ = \varprojlim_\Phi \OO_X^+/p$, where the inverse limit is taken along the Frobenius map.

If $X$ lives over $\Spa(K,K^+)$, where $K$ is a perfectoid field with an open and bounded valuation subring $K^+\subset K$, we set $\hat{\OO}_{X^\flat} = \hat{\OO}_{X^\flat}^+\otimes_{K^{\flat +}} K^\flat$.
\end{definition}

\begin{lem}\label{ContOXplusonAffPerfTilt} Let $K$ be a perfectoid field of characteristic $0$ with an open and bounded valuation subring $K^+\subset K$, let $X$ be a locally noetherian adic space over $\Spa(K,K^+)$, and $U\in X_\proet$ be affinoid perfectoid, with $\hat{U} = \Spa(R,R^+)$, where $(R,R^+)$ is a perfectoid affinoid $(K,K^+)$-algebra. Let $(R^\flat,R^{\flat +})$ be its tilt.

\begin{altenumerate}
\item[{\rm (i)}] We have $\hat{\OO}_{X^\flat}^+(U) = R^{\flat +}$, $\hat{\OO}_{X^\flat}(U) = R^\flat$.
\item[{\rm (ii)}] The cohomology groups $H^i(U,\hat{\OO}_{X^\flat}^+)$ are almost zero for $i>0$, with respect to the almost setting defined by $K^{\flat +}$ and its ideal of topologically nilpotent elements.
\end{altenumerate}
\end{lem}

\begin{proof} This follows from Lemma \ref{ContOXplusonAffPerf} by repeating part of its proof in the tilted situation.
\end{proof}

\begin{proof}{\it (of Theorem \ref{Finiteness}.)} Let $X^\prime$ be the fibre product $X\times_{\Spa(K,K^+)} \Spa(K,\OO_K)$, which is an open subset of $X$. Then the induced morphism
\[
H^i(X_\et,\mathbb{L}\otimes \OO_X^+/p)\rightarrow H^i(X^\prime_\et,\mathbb{L}\otimes \OO_X^+/p)
\]
is an almost isomorphism of $K^+$-modules. Indeed, take a simplicial cover $U_\bullet$ of $X$ by affinoid perfectoid $U_k\rightarrow X$. Then $U_\bullet\times_X X^\prime\rightarrow X^\prime$ is a simplicial cover of $X^\prime$ by affinoid perfectoid $U_k\times_X X^\prime$. Then for all $i,k\geq 0$,
\[
H^i(U_k,\mathbb{L}\otimes \OO_X^+/p)\rightarrow H^i(U_k\times_X X^\prime,\mathbb{L}\otimes \OO_X^+/p)
\]
is an almost isomorphism by Lemma \ref{TwistedOXPlus}, which implies the same result for $X$ compared to $X^\prime$.

Now recall that $K^\flat$ is an algebraically closed field of characteristic $p$. Fix an element $\pi\in \OO_{K^\flat}$ such that $\pi^\sharp = p$. Let $M_k = H^i(X^\prime_\proet,\mathbb{L}\otimes \hat{\OO}_{X^\flat}^+/\pi^k)$. As $\hat{\OO}_{X^\flat}^+|_{X^\prime}$ is a sheaf of perfect flat $\OO_{K^\flat}$-algebras with $\hat{\OO}_{X^\flat}^+/\pi = \OO_X^+/p$, we see that Lemma \ref{LemmaFiniteness} implies that the $M_k$ satisfy the hypotheses of Lemma \ref{KeyAlmostLemma}. It follows that there is some integer $r\geq 0$ such that
\[
H^i(X_\proet,\mathbb{L}\otimes \hat{\OO}_{X^\flat}^+/\pi^k)^a\cong H^i(X^\prime_\proet,\mathbb{L}\otimes \hat{\OO}_{X^\flat}^+/\pi^k)^a\cong (\OO_{K^\flat}^a/\pi^k)^r
\]
as almost $\OO_{K^\flat}$-modules, compatibly with the Frobenius action. By Lemma \ref{InverseLimitExact} and Lemma \ref{ContOXplusonAffPerfTilt}, we have
\[
R\varprojlim (\mathbb{L}\otimes \hat{\OO}_{X^\flat}^+/\pi^k)^a = (\mathbb{L}\otimes \hat{\OO}_{X^\flat}^+)^a\ .
\]
Therefore,
\[
H^i(X_\proet,\mathbb{L}\otimes \hat{\OO}_{X^\flat}^+)^a\cong (\OO_{K^\flat}^a)^r\ .
\]
Inverting $\pi$, we see that
\[
H^i(X_\proet,\mathbb{L}\otimes \hat{\OO}_{X^\flat})\cong (K^\flat)^r\ ,
\]
still compatible with the action of Frobenius $\varphi$. Now we use the Artin-Schreier sequence
\[
0\rightarrow \mathbb{L}\rightarrow \mathbb{L}\otimes \hat{\OO}_{X^\flat}\rightarrow \mathbb{L}\otimes \hat{\OO}_{X^\flat}\rightarrow 0\ ,
\]
where the second map is $v\otimes f\mapsto v\otimes (f^p - f)$. This is an exact sequence of sheaves on $X_\proet$: It suffices to check locally on $U\in X_\proet$ which is affinoid perfectoid and over which $\mathbb{L}$ is trivial, and only the surjectivity is problematic. To get surjectivity, one has to realize a finite \'{e}tale cover of $\hat{U}^\flat$, but $\hat{U}^\flat_\fet\cong \hat{U}_\fet$, and finite \'{e}tale covers of $\hat{U}$ come via pullback from finite \'{e}tale covers in $X_\proet$, by Lemma 7.5 (i) of \cite{ScholzePerfectoidSpaces1}.

On cohomology, the Artin-Schreier sequence gives
\[
\ldots\rightarrow H^i(X_\proet,\mathbb{L})\rightarrow H^i(X_\proet,\mathbb{L}\otimes \hat{\OO}_{X^\flat})\rightarrow H^i(X_\proet,\mathbb{L}\otimes \hat{\OO}_{X^\flat})\rightarrow \ldots\ .
\]
But the second map is the same as $(K^\flat)^r\rightarrow (K^\flat)^r$, which is coordinate-wise $x\mapsto x^p-x$. This is surjective as $K^\flat$ is algebraically closed, hence the long-exact sequence reduces to short exact sequences, and (using Lemma \ref{CompProetVSEt} (i))
\[
H^i(X_\et,\mathbb{L}) = H^i(X_\proet,\mathbb{L}) = H^i(X_\proet,\mathbb{L}\otimes \hat{\OO}_{X^\flat})^{\varphi = 1}\cong \F_p^r\ ,
\]
which implies all desired statements.
\end{proof}

\begin{cor}\label{RelPrimComp} Let $f:X\rightarrow Y$ be a proper smooth morphism of locally noetherian adic spaces over $\Spa(\Q_p,\Z_p)$. Let $\mathbb{L}$ be an $\mathbb{F}_p$-local system on $X_\et$. Then for all $i\geq 0$, there is an isomorphism of sheaves of almost $\OO_Y^+$-modules
\[
(R^if_{\et\ast} \mathbb{L})\otimes \OO_Y^{+a}/p\cong R^if_{\et\ast}(\mathbb{L}\otimes \OO_X^{+a}/p)\ .
\]
\end{cor}

Here, we use the almost setting relative to the site $Y_\et$, the sheaf of algebras $\OO_Y^+$, and the ideal of elements of valuation $<1$ everywhere. If $Y$ lives over $\Spa(K,K^+)$ for some nondiscretely valued extension $K$ of $\Q_p$, this is the same as the almost setting with respect to $K^+$ and the ideal of topologically nilpotent elements in $K^+$.

\begin{proof} It suffices to check at stalks at all geometric points $y$ of $Y$. Let $y$ correspond to $\Spa(L,L^+)\rightarrow Y$, where $L$ is an algebraically closed complete extension of $K$, and let $X_y = X\times_Y \Spa(L,L^+)$. By Proposition 2.6.1 of \cite{Huber}, we have
\[
((R^if_{\et\ast} \mathbb{L})\otimes \OO_Y^{+a}/p)_y = (R^if_{\et\ast} \mathbb{L})_y\otimes (\OO_Y^{+a}/p)_y = H^i(X_{y\et},\mathbb{L})\otimes L^{+a}/p\ ,
\]
and
\[
(R^if_{\et\ast}(\mathbb{L}\otimes \OO_X^{+a}/p))_y = H^i(X_{y\et},\mathbb{L}\otimes \OO_{X_y}^{+a}/p)\ .
\]
Now the result follows from Theorem \ref{Finiteness}.
\end{proof}

\section{Period sheaves}\label{PeriodSheavesSection}

\begin{definition} Let $X$ be a locally noetherian adic space over $\Spa(\Q_p,\Z_p)$. Consider the following sheaves on $X_\proet$.
\begin{altenumerate}
\item[{\rm (i)}] The sheaf $\A_\inf = W(\hat{\OO}_{X^{\flat}}^+)$, and its rational version $\B_\inf = \A_\inf[\frac 1p]$. Note that we have $\theta: \A_\inf\rightarrow \hat{\OO}_X^+$ extending to $\theta: \B_\inf \rightarrow \hat{\OO}_X$.
\item[{\rm (ii)}] The positive deRham sheaf
\[
\B_\dR^+ = \varprojlim \B_\inf/(\ker \theta)^n\ ,
\]
with its filtration $\Fil^i \B_\dR^+ = (\ker \theta)^i \B_\dR^+$.
\item[{\rm (iii)}] The deRham sheaf
\[
\B_\dR = \B_\dR^+[t^{-1}]\ ,
\]
where $t$ is any element that generates $\Fil^1 \B_\dR^+$. It has the filtration
\[
\Fil^i \B_\dR = \sum_{j\in \mathbb{Z}} t^{-j} \Fil^{i+j} \B_\dR^+\ .
\]
\end{altenumerate}
\end{definition}

\begin{rem} We will see that locally on $X_\proet$, the element $t$ exists, is unique up to a unit and is not a zero-divisor. This shows that the sheaf $\B_\dR$ and its filtration are well-defined.
\end{rem}

Before we describe these period sheaves explicitly, we first study them abstractly for a perfectoid field $K$ with open and bounded valuation subring $K^+\subset K$ of characteristic $0$ and a perfectoid affinoid $(K,K^+)$-algebra $(R,R^+)$. Fix $\pi\in K^\flat$ such that $\pi^\sharp/p\in (K^+)^\times$. We make the following definitions.
\[\begin{aligned}
\A_\inf(R,R^+) &= W(R^{\flat +})\ ,\\
\B_\inf(R,R^+) &=  \A_\inf(R,R^+)[p^{-1}]\ ,\\
\B_\dR^+(R,R^+) &= \varprojlim \B_\inf(R,R^+)/(\ker \theta)^i\ .
\end{aligned}\]
Moreover, we know that $\theta: \A_\inf(R,R^+) = W(R^{\flat +})\rightarrow R^+$ is surjective.

\begin{lem}\label{ExistenceXi} There is an element $\xi\in \A_\inf(K,K^+)$ that generates $\ker \theta$, where $\theta: \A_\inf(K,K^+)\rightarrow K^+$. The element $\xi$ is not a zero-divisor, and hence is unique up to a unit.

In fact, for any perfectoid affinoid $(K,K^+)$-algebra $(R,R^+)$, the element $\xi$ generates $\ker (\theta: \A_\inf(R,R^+)\rightarrow R^+)$, and is not a zero-divisor in $\A_\inf(R,R^+)$.
\end{lem}

\begin{proof} We will choose the element $\xi$ of the form $\xi=[\pi] - \sum_{i=1}^\infty p^i [x_i]$ for certain elements $x_i\in \OO_{K^\flat}$. In fact, $[\pi]$ maps via $\theta$ to some element $\pi^\sharp$ in $p(K^+)^\times$, and any such element can be written as a sum $\sum_{i=1}^\infty p^i \theta([x_i])$, as desired.

Let $y=\sum_{i=0}^\infty p^i [y_i]\in W(R^{\flat +})$, and assume $\xi y=0$, but $y\neq 0$. Because $W(R^{\flat +})$ is flat over $\Z_p$, we may assume that $y_0\neq 0$. Reducing modulo $p$, we see that $\pi y_0 = 0$, which implies $y_0=0$, as $R^{\flat +}$ is flat over $K^{\flat +}$, contradiction.

Assume now that $y=\sum_{i=0}^\infty p^i [y_i]\in \ker (\theta: W(R^{\flat +})\rightarrow R^+)$. We want to show that it is divisible by $\xi$. Because $R^+$ is flat over $\Z_p$, we may assume that $y_0\neq 0$. As a first step, we will find $z_0\in W(R^{\flat +})$ so that $y-z_0 \xi$ is divisible by $p$. Indeed, $W(R^{\flat +})/(\xi,p) = R^{\flat +}/\pi = R^+/p$, so that $f$ is mapped to zero in this quotient, which amounts to the existence of $z_0$ as desired.

Continuing in this fashion gives us a sequence $z_0,z_1,\ldots\in W(R^{\flat +})$ such that $y-(\sum_{i=0}^k p^i z_i)\xi$ is divisible by $p^{k+1}$, for all $k\geq 0$. But $W(R^{\flat +})$ is $p$-adically complete and separated, hence $y=(\sum_{i=0}^\infty p^i z_i) \xi$, as desired.
\end{proof}

In particular, we can also define $\B_\dR(R,R^+) = \B_\dR^+(R,R^+)[\xi^{-1}]$, with the filtration given by $\Fil^i \B_\dR(R,R^+) = \xi^i \B_\dR^+(R,R^+)$, $i\in \Z$.

\begin{cor} For any $i\in \Z$, we have $\gr^i \B_\dR(R,R^+)\cong \xi^i R$, which is a free $R$-module of rank $1$. In particular, $\gr^\bullet \B_\dR(R,R^+)\cong R[\xi^{\pm 1}]$.
\end{cor}

\begin{proof} The element $\xi$ has the same properties in $\B_\inf$ as in $\A_\inf$, i.e. it generates $\ker \theta$ and is not a zero-divisor. The corollary follows.
\end{proof}

Note that all of these rings are $\A_\inf(K,K^+)$-algebras, e.g. $K^+$ via the map $\theta$. In the following, we consider the almost-setting with respect to this ring and the ideal generated by all $[\pi^{1/p^N}]$.

\begin{thm}\label{DescriptionPeriodSheaves} Let $X$ be a locally noetherian adic space over $\Spa(K,K^+)$. Assume that $U\rightarrow X_\proet$ is affinoid perfectoid, with $\hat{U} = \Spa(R,R^+)$.
\begin{altenumerate}
\item[{\rm (i)}] We have a canonical isomorphism
\[
\A_\inf(U) = \A_\inf(R,R^+)\ ,
\]
and analogous statements for $\B_\inf$, $\B_\dR^+$ and $\B_\dR$. In particular, there is an element $\xi$ generating $\Fil^1 \B_\dR^+(U)$, unique up to a unit, and it is not a zero-divisor.
\item[{\rm (ii)}] All $H^i(U,\mathcal{F})$ are almost zero for $i>0$, where $\mathcal{F}$ is any of the sheaves just considered.
\item[{\rm (iii)}] In $\B_\dR^+(U)$ and $\B_\dR(U)$ the element $[\pi]$ becomes invertible, in particular the cohomology groups $H^i(U,\mathcal{F})$ vanish for these sheaves.
\end{altenumerate}
\end{thm}

\begin{proof} By induction on $m$, we get a description of $W(\hat{\OO}_{X^\flat}^+)/p^m$, together with almost vanishing of cohomology. Now we use Lemma \ref{InverseLimitExact} to get the description of $\A_\inf$. Afterwards, one passes to $\B_\inf$ by taking a direct limit, which is obviously exact. This proves parts (i) and (ii) for these sheaves.

In order to pass to $\B_\dR^+$, one has to check that the exact sequence of sheaves on $X_\proet$
\[
0\rightarrow \B_\inf\buildrel \xi^i\over\rightarrow \B_\inf\rightarrow \B_\inf / (\ker \theta)^i \rightarrow 0
\]
stays exact after taking sections over $U$. We know that the defect is controlled by $H^1(U,\B_\inf)$, which is almost zero. We see that all statements follow once we know that $[\pi]$ is invertible in $\B_\inf / (\ker \theta)$. But the latter is a sheaf of $\B_\inf(K,K^+)/(\ker \theta)= K$-modules, and $[\pi]$ maps to the unit $\pi^\sharp\in K^\times$.
\end{proof}

\begin{cor}\label{TrivialCovers} Let $X$ be a locally noetherian adic space over $\Spa(K,K^+)$, and assume that $U\in X_\proet$ is affinoid perfectoid. Further, let $S$ be some profinite set, and $V=U\times S\in X_\proet$, which is again affinoid perfectoid. Then
\[
\mathcal{F}(V) = \Hom_\cont(S,\mathcal{F}(U))
\]
for any of the sheaves
\[
\mathcal{F}\in \{\hat{\OO}_X,\hat{\OO}_X^+,\hat{\OO}_{X^\flat},\hat{\OO}_{X^\flat}^+,\A_\inf,\B_\inf,\B_\dR^+,\B_\dR,\gr^i \B_\dR\}\ .
\]
Here, $\hat{\OO}_X^+(U)$ is given the $p$-adic topology, and all other period sheaves are given the induced topology: For example, $\A_\inf(U)$ the inverse limit topology, $\B_\inf(U)$ the direct limit topology, the quotients $(\B_\inf/(\ker \theta)^n)(U)$ the quotient topology, and then finally $\B_\dR^+(U)$ the inverse limit topology and $\B_\dR(U)$ the direct limit topology.
\end{cor}

\begin{proof} Go through all identifications.
\end{proof}

This proposition shows that even though we defined our sheaves completely abstractly without any topology, their values on certain profinite covers naturally involve the topology. This will later imply the appearance of continuous group cohomology.

\begin{prop} Let $X$ be a locally noetherian adic space over $\Spa(\Q_p,\Z_p)$. For all $i\in \Z$, we have $\gr^i \B_\dR\cong \hat{\OO}_X(i)$, where $(i)$ denotes a Tate twist: Let $\hat{\Z}_p = \varprojlim \Z/p^n\Z$ as sheaves on $X_\proet$, and $\hat{\Z}_p(1) = \varprojlim \mu_{p^n}$. Then for any sheaf $\mathcal{F}$ of $\hat{\Z}_p$-modules on $X_\proet$, we set $\mathcal{F}(1) = \mathcal{F}\otimes_{\hat{\Z}_p} \hat{\Z}_p(1)$.
\end{prop}

\begin{proof} Let $K$ be the completion of $\Q_p(\mu_{p^\infty})$. A choice of $p^n$-th roots of unity gives rise to an element $\epsilon\in \OO_{K^\flat}$. Recall the element
\[
t=\log([\epsilon])\in \Fil^1 \B_\dR^+(K,K^+)\ ,
\]
which generates $\Fil^1$, so that we have $\gr^i \B_\dR = t^i \hat{\OO}_X$ over $X_{K,\proet}\cong X_\proet / X_K$, cf. Proposition \ref{ProetChangeBase}. Because the action of the Galois group $\Gal(\Q_p(\mu_{p^\infty})/\Q_p)$ on $t$ is through the cyclotomic character, the isomorphism descends to an isomorphism $\gr^i \B_\dR\cong \hat{\OO}_X(i)$ on $X_\proet$.
\end{proof}

\begin{definition} Let $X$ be a smooth adic space over $\Spa(k,\OO_k)$, where $k$ is a discretely valued complete nonarchimedean extension of $\Q_p$ with perfect residue field $\kappa$. Consider the following sheaves on $X_\proet$.
\begin{altenumerate}
\item[{\rm (i)}] The sheaf of differentials $\Omega_X^1 = \nu^\ast \Omega_{X_\et}^1$, and its exterior powers $\Omega_X^i$.
\item[{\rm (ii)}] The tensor product $\OO\B_\inf = \OO_X\otimes_{W(\kappa)} \B_\inf$. Here $W(\kappa)=\nu^{\ast} W(\kappa)$ is the constant sheaf associated to $W(\kappa)$. It still admits $\theta: \OO\B_\inf\rightarrow \hat{\OO}_X$.
\item[{\rm (iii)}] The positive structural deRham sheaf
\[
\OO\B_\dR^+ = \varprojlim \OO\B_\inf / (\ker \theta)^n\ ,
\]
with its filtration $\Fil^i \OO\B_\dR^+ = (\ker \theta)^i \OO\B_\dR^+$.
\item[{\rm (iv)}] The structural deRham sheaf
\[
\OO\B_\dR = \OO\B_\dR^+[t^{-1}]\ ,
\]
where $t$ is a generator of $\Fil^1 \B_\dR^+$, with the filtration
\[
\Fil^i \OO\B_\dR = \sum_{j\in \mathbb{Z}} t^{-j} \Fil^{i+j} \OO\B_\dR^+\ .
\]
\end{altenumerate}
\end{definition}

\begin{rem} Because locally on $X_\proet$, the element $t$ exists and is unique up to a unit and not a zero-divisor, the sheaf $\OO\B_\dR$ and its filtrations are well-defined.
\end{rem}

Also note that the sheaf $\OO\B_\inf$ admits a unique $\B_\inf$-linear connection
\[
\nabla: \OO\B_\inf\rightarrow \OO\B_\inf\otimes_{\OO_X} \Omega^1_X\ ,
\]
extending the one on $\OO_X$. This connection extends uniquely to the completion
\[
\nabla: \OO\B_\dR^+\rightarrow \OO\B_\dR^+\otimes_{\OO_X} \Omega^1_X\ ,
\]
and this extension is $\B_\dR^+$-linear. Because $t\in \B_\dR^+$, it further extends to a $\B_\dR$-linear connection
\[
\nabla: \OO\B_\dR\rightarrow \OO\B_\dR\otimes_{\OO_X} \Omega^1_X\ .
\]

We want to describe $\OO\B_\dR^+$. For this, choose an algebraic extension of $k$ whose completion $K$ is perfectoid. We get the base-change $X_K$ of $X$ to $\Spa(K,\OO_K)$, and again consider $X_K\in X_\proet$ by slight abuse of notation. We assume given an \'{e}tale map $X\rightarrow \mathbb{T}^n$; such a map exists locally on $X$. Let $\tilde{X} = X\times_{\mathbb{T}^n} \tilde{\mathbb{T}}^n$. Taking a further-base change to $K$, $\tilde{X}_K\in X_{K,\proet}\cong X_\proet / X_K$ is perfectoid.

In the following, we look at the localized site $X_\proet / \tilde{X}$. We get the elements
\[
u_i = T_i\otimes 1 - 1\otimes [T_i^\flat]\in \OO\B_\inf|_{\tilde{X}} = (\OO_X\otimes_{W(\kappa)} W(\hat{\OO}_{X^\flat}^+))|_{\tilde{X}}
\]
in the kernel of $\theta$, where $T_i^\flat\in \hat{\OO}_{X^\flat}^+ = \varprojlim \OO_X^+/p$ is given by the sequence $(T_i,T_i^{\frac 1p},\ldots)$ in the inverse limit.

\begin{prop}\label{DescrBdR} The map
\[
\B_\dR^+|_{\tilde{X}}[[X_1,\ldots,X_n]]\rightarrow \OO\B_\dR^+|_{\tilde{X}}
\]
sending $X_i$ to $u_i$ is an isomorphism of sheaves over $X_\proet / \tilde{X}$.
\end{prop}

\begin{proof} It suffices to check this over $X_\proet / \tilde{X}_K$. The crucial point is to show that $\B_\dR^+|_{\tilde{X}_K}[[X_1,\ldots,X_n]]$ admits a unique $\OO_X|_{\tilde{X}_K}$-algebra structure, sending $T_i$ to $[T_i^\flat]+X_i$ and compatible with the structure on
\[
\B_\dR^+[[X_1,\ldots,X_n]]/(\ker \theta) = \hat{\OO}_X\ .
\]
This being granted, we get a natural map
\[
(\OO_X\otimes_{W(\kappa)} W(\hat{\OO}_{X^\flat}^+))|_{\tilde{X}_K}\rightarrow \B_\dR^+|_{\tilde{X}_K}[[X_1,\ldots,X_n]]\ ,
\]
which induces a map $\OO\B_\dR^+|_{\tilde{X}_K}\rightarrow \B_\dR^+|_{\tilde{X}_K}[[X_1,\ldots,X_n]]$ which is easily seen to be inverse to the map above, giving the desired isomorphism.

In order to check that $\B_\dR^+|_{\tilde{X}_K}[[X_1,\ldots,X_n]]$ admits a natural $\OO_X|_{\tilde{X}_K}$-algebra structure, we need the following lemma.

\begin{lem} Let $(R,R^+)$ be a perfectoid affinoid $(K,\OO_K)$-algebra, so that we get $\B_\dR^+(R,R^+)$. Let $S$ be a finitely generated $\OO_k$-algebra. Then any morphism
\[
f: S\rightarrow \B_\dR^+(R,R^+)[[X_1,\ldots,X_n]]
\]
such that $\theta(f(S))\subset R^+$ extends to the $p$-adic completion of $S$.
\end{lem}

\begin{proof} It suffices to check modulo $(\ker \theta)^i$ for all $i$. There it follows from the fact that any finitely generated $R^+$-submodule of $\gr^i \B_\dR^+(R,R^+)\cong R$ is $p$-adically complete: In fact, the image of $S$ will be contained in
\[
(W(R^{\flat +})[[X_1,\ldots,X_n]]/(\ker \theta)^i)[\frac {\xi}{p^k},\frac{X_1}{p^k},\ldots,\frac{X_n}{p^k}]
\]
for some $k$, and this algebra is $p$-adically complete. Here, $\xi$ is as in Lemma \ref{ExistenceXi}.
\end{proof}

Moreover, we have the following lemma about \'{e}tale maps of adic spaces, specialized to the case $\mathbb{T}^n$.

\begin{lem} Let $U=\Spa(R,R^+)$ over $\Spa(W(\kappa)[p^{-1}],W(\kappa))$ be an affinoid adic space of finite type with an \'{e}tale map $U\rightarrow \mathbb{T}^n$. Then there exists a finitely generated $W(\kappa)[T_1^{\pm 1},\ldots,T_n^{\pm 1}]$-algebra $R_0^+$, such that $R_0 = R_0^+[\frac 1p]$ is \'{e}tale over
\[
W(\kappa)[p^{-1}][T_1^{\pm 1},\ldots,T_n^{\pm 1}]
\]
and $R^+$ is the $p$-adic completion of $R_0^+$.
\end{lem}

\begin{proof} We use \cite{Huber}, Corollary 1.7.3 (iii), to construct the affinoid ring $(R_0,R_0^+)$, denoted $B$ there. We have to see that $R_0^+$ is a finitely generated $W(\kappa)[T_1^{\pm 1},\ldots,T_n^{\pm 1}]$-algebra. But \cite{Huber}, Remark 1.2.6 (iii), implies that it is the integral closure of a finitely generated $W(\kappa)[T_1^{\pm 1},\ldots,T_n^{\pm 1}]$-algebra $R_1^+\subset R_0^+$ inside $R_0$, with $R_1^+[\frac 1p]=R_0$. But $W(\kappa)$ is excellent, in particular for any reduced flat finitely generated $W(\kappa)$-algebra $S^+$, the normalization of $S^+$ inside $S^+[p^{-1}]$ is finite over $S^+$, giving the desired result.
\end{proof}

First note that one has a map
\[
W(\kappa)[p^{-1}][T_1^{\pm 1},\ldots,T_n^{\pm 1}]\rightarrow \B_\dR^+|_{\tilde{X}}[[X_1,\ldots,X_n]]
\]
sending $T_i$ to $[T_i^\flat]+X_i$. For this, note that $T_i\mod (\ker \theta)$ is $[T_i^\flat]$, which is invertible, hence $T_i$ is itself invertible.

Now take some affinoid perfectoid $U\in X_\proet / \tilde{X}_K$, and write it as the inverse limit of affinoid $U_i\in X_\et$. In particular, $\OO_X(U) = \varinjlim \OO_X(U_i)$, and we may apply the last lemma to $U_i\rightarrow \mathbb{T}^n$. This gives algebras $R_{i0}^+$ whose generic fibre $R_{i0}$ is \'{e}tale over $W(\kappa)[p^{-1}][T_1^{\pm 1},\ldots,T_n^{\pm 1}]$. By Hensel's lemma, we can lift $R_{i0}$ uniquely to $\B_\dR^+(U)[[X_1,\ldots,X_n]]$, hence we get lifts of $R_{i0}^+$. These extend to the $p$-adic completion, hence we get lifts of $\mathcal{O}_X^+(U_i)$, and thus of $\mathcal{O}_X(U_i)$. Take the direct limits of these lifts to conclude.
\end{proof}

Let us collect some corollaries. First off, we have the following version of the Poincar\'{e} lemma.

\begin{cor}\label{DeepPoincareLemma} Let $X$ be an $n$-dimensional smooth adic space over $\Spa(k,\OO_k)$. The following sequence of sheaves on $X_\proet$ is exact.
\[
0\rightarrow \B_\dR^+\rightarrow \OO\B_\dR^+\buildrel\nabla\over\rightarrow \OO\B_\dR^+\otimes_{\OO_X} \Omega_X^1 \buildrel\nabla\over\rightarrow \ldots \buildrel\nabla\over\rightarrow \OO\B_\dR^+\otimes_{\OO_X} \Omega_X^n\rightarrow 0\ .
\]
Moreover, the derivation $\nabla$ satisfies Griffiths transversality with respect to the filtration on $\OO\B_\dR^+$, and with respect to the grading giving $\Omega_X^i$ degree $i$, the sequence is strict exact.
\end{cor}

\begin{proof} Using the description of Proposition \ref{DescrBdR}, this is obvious.
\end{proof}

In particular, we get the following short exact sequence, often called Faltings's extension.

\begin{cor}\label{FaltingssExtension} Let $X$ be a smooth adic space over $\Spa(k,\OO_k)$. Then we have a short exact sequence of sheaves over $X_\proet$,
\[
0\rightarrow \hat{\OO}_X(1)\rightarrow \gr^1 \OO\B_\dR^+\rightarrow \hat{\OO}_X\otimes_{\OO_X} \Omega_X^1\rightarrow 0\ .
\]
\end{cor}

\begin{proof} This is the first graded piece of the Poincar\'{e} lemma.
\end{proof}

\begin{cor} Let $X\rightarrow \mathbb{T}^n$, $\tilde{X}$, etc., be as above. For any $i\in \mathbb{Z}$, we have an isomorphism of sheaves over $X_\proet / \tilde{X}_K$,
\[
\gr^i \OO\B_\dR\cong \xi^i \hat{\OO}_X[\frac{X_1}{\xi},\ldots,\frac{X_n}{\xi}]\ .
\]
In particular,
\[
\gr^\bullet \OO\B_\dR\cong \hat{\OO}_X[\xi^{\pm 1},X_1,\ldots,X_n]\ ,
\]
where $\xi$ and all $X_i$ have degree $1$.$\hfill \Box$
\end{cor}

\begin{prop}\label{CohomOOBdR} Let $X=\Spa(R,R^+)$ be an affinoid adic space of finite type over $\Spa(k,\OO_k)$ with an \'{e}tale map $X\rightarrow \mathbb{T}^n$ that factors as a composite of rational embeddings and finite \'{e}tale maps.
\begin{altenumerate}
\item[{\rm (i)}] Assume that $K$ contains all $p$-power roots of unity. Then
\[
H^q(X_K,\gr^0 \OO\B_\dR)=0
\]
unless $q=0$, in which case it is given by $R\hat{\otimes}_k K$.
\item[{\rm (ii)}] We have
\[
H^q(X,\gr^i \OO\B_\dR) = 0
\]
unless $i=0$ and $q=0,1$. If $i=0$, we have $(\gr^0 \OO\B_\dR)(X) = R$ and $H^1(X,\gr^0 \OO\B_\dR) = R \log \chi$. Here, $\chi: \Gal(\bar{k}/k)\rightarrow \Z_p^\times$ is the cyclotomic character and
\[
\log \chi\in \Hom_\cont(\Gal(\bar{k}/k),\Q_p) = H^1_\cont(\Gal(\bar{k}/k),\Q_p)
\]
is its logarithm.
\end{altenumerate}
\end{prop}

\begin{proof}\begin{altenumerate}
\item[{\rm (i)}] We use the cover $\tilde{X}_K\rightarrow X_K$ to compute the cohomology using the Cartan-Leray spectral sequence. This is a $\Z_p^n$-cover, and all fibre products $\tilde{X}_K\times_{X_K} \cdots \times_{X_K} \tilde{X}_K$ are affinoid perfectoid, and hence we know that all higher cohomology groups of the sheaves considered vanish. The version of Corollary \ref{TrivialCovers} for $\gr^0 \OO\B_\dR$ stays true, so we find that
\[
H^q(X_K,\gr^0 \OO\B_\dR) = H^q_\cont(\Z_p^n,\gr^0 \OO\B_\dR(\tilde{X}_K))\ .
\]
Now we follow the computation of this Galois cohomology group given in \cite{BrinonRepresentations}, Proposition 4.1.2. First, note that we may write
\[
\gr^0 \OO\B_\dR(\tilde{X}_K) = \tilde{R}[V_1,\ldots,V_n]\ ,
\]
where $\hat{\tilde{X}}_K = \Spa(\tilde{R},\tilde{R}^+)$, and the $V_i$ are given by $t^{-1} \log([T_i^\flat]/T_i)$, where $t=\log([\epsilon])$ as usual. Let $\gamma_i\in \Z_p^n$ be the $i$-th basis vector.

\begin{lem} The action of $\gamma_i$ on $V_j$ is given by $\gamma_i(V_j) = V_j$ if $i\neq j$ and $\gamma_i(V_i) = V_i + 1$.
\end{lem}

\begin{proof} By definition, $\gamma_i$ acts on $T_j^\flat$ trivially if $i\neq j$, and by multiplication by $\epsilon$ if $i=j$. This gives the claim.
\end{proof}

We claim that the inclusion
\[
(R\hat{\otimes}_k K)[V_1,\ldots,V_n]\subset \tilde{R}[V_1,\ldots,V_n]
\]
induces an isomorphism on continuous $\Z_p^n$-cohomology. It is enough to check this on associated gradeds for the filtration given by the degree of polynomials. On associated gradeds, the action of $\Z_p^n$ on the variables $V_i$ is trivial by the previous lemma, and it suffices to see that $R\hat{\otimes}_k K\subset \tilde{R}$ induces an isomorphism on continuous $\Z_p^n$-cohomology. The following lemma reduces the computation to Lemma \ref{BasicLocalComputation}.

\begin{lem}\label{RoughLocalStructure} The map
\[
R^+\hat{\otimes}_{\OO_k\langle T_1^{\pm 1},\ldots,T_n^{\pm 1}\rangle} \OO_K\langle T_1^{\pm 1/p^\infty},\ldots,T_n^{\pm 1/p^\infty}\rangle\rightarrow \tilde{R}^+
\]
is injective with cokernel killed by some power of $p$. In particular, we have
\[
\tilde{R} = R\hat{\otimes}_{k\langle T_1^{\pm 1},\ldots,T_n^{\pm 1}\rangle} K\langle T_1^{\pm 1/p^\infty},\ldots,T_n^{\pm 1/p^\infty}\rangle\ .
\]
\end{lem}

\begin{proof} This is an immediate consequence of Lemma \ref{PreciseLocalStructure} (ii).
\end{proof}

Now we have to compute
\[
H^q_\cont(\Z_p^n,(R\hat{\otimes}_k K)[V_1,\ldots,V_n])\ .
\]
We claim that inductively $H^q_\cont(\mathbb{Z}_p \gamma_i,(R\hat{\otimes}_k K)[V_1,\ldots,V_i]) = 0$ for $q>0$ and equal to $(R\hat{\otimes}_k K)[V_1,\ldots,V_{i-1}]$ for $q=0$. For this purpose, note that the cohomology is computed by the complex
\[
(R\hat{\otimes}_k K)[V_1,\ldots,V_i]\buildrel {\gamma_i - 1}\over\rightarrow (R\hat{\otimes}_k K)[V_1,\ldots,V_i]\ .
\]
If we set $S=(R\hat{\otimes}_k K)[V_1,\ldots,V_{i-1}]$, then the map $\gamma_i - 1$ sends a polynomial $P\in S[V_i]$ to $P(V_i+1) - P(V_i)$. One sees that the kernel of $\gamma_i - 1$ consists precisely of the constant polynomials, i.e. $S$, and the cokernel of $\gamma_i - 1$ is trivial.

\item[{\rm (ii)}] First note that in part (i), we have calculated $H^q(X_K,\gr^i \OO\B_\dR)$ for any $i\in \Z$, as all of these sheaves are isomorphic on $X_\proet / X_K$ to $\gr^0 \OO\B_\dR$.

We take $K$ as the completion of $k(\mu_{p^\infty})$, and we let $\Gamma_k = \Gal(k(\mu_{p^\infty})/k)$. We want to use the Cartan-Leray spectral sequence for the cover $X_K\rightarrow X$. For this, we have to know
\[
H^q(X_K^{m/X},\gr^i \OO\B_\dR)\ ,
\]
where we set $X_K^{m/X} = X_K\times_X\cdots\times_X X_K$. Inspection of the proof shows that they are given by
\[
H^q(X_K^{m/X},\gr^i \OO\B_\dR) = \Hom_\cont(\Gamma_k^{m-1},H^q(X_K,\gr^i \OO\B_\dR))\ :
\]
In fact, using the cover $\tilde{X}_K\times_{X_K} X_K^{m/X}$ of $X_K^{m/X}$ to compute the cohomology via the Cartan-Leray spectral sequence, the version of Corollary \ref{TrivialCovers} for $\gr^i \OO\B_\dR$ says that at each step in the proof, one has to take $\Hom_\cont(\Gamma_k^{m-1},\bullet)$. This shows that we have an identity
\[
H^q(X,\gr^i \OO\B_\dR) = H^q_\cont(\Gamma_k,R\hat{\otimes}_k K(i))\ .
\]
Similarly to Lemma \ref{BasicLocalComputation}, the map $R(i)\rightarrow R\hat{\otimes}_k K(i)$ induces an isomorphism on continuous $\Gamma_k$-cohomology. But then we get
\[
H^q_\cont(\Gamma_k,R\hat{\otimes}_k K(i)) = H^q_\cont(\Gamma_k,R(i)) = R\otimes_{\mathbb{Q}_p} H^q_\cont(\Gamma_k,\mathbb{Q}_p(i))\ ,
\]
and the latter groups are well-known, cf. \cite{TatePDivGroups}.
\end{altenumerate}
\end{proof}

\begin{cor}\label{CompDeepStructureSheaf} Let $X$ be a smooth adic space over $\Spa(k,\OO_k)$. Then $\nu_\ast \OO\B_\dR = \OO_{X_\et}$. Moreover, $\nu_\ast \hat{\OO}_X = \OO_{X_\et}$, $\nu_\ast \hat{\OO}_X(n)=0$ for $n\geq 1$,
\[
R^1\nu_\ast \hat{\OO}_X(1)\cong \Omega_{X_\et}^1
\]
via the connecting map in Faltings's extension, and $R^1\nu_\ast \hat{\OO}_X(n)=0$ for $n\geq 2$.
\end{cor}

\begin{rem} One could compute all $R^i \nu_\ast \hat{\OO}_X(j)$. They are $0$ if $i<j$ or $i>j+1$, and they are $\Omega_{X_\et}^i$ if $i=j$, and $\Omega_{X_\et}^i \log \chi$ if $i=j+1$.
\end{rem}

\begin{proof} The first part is clear. For the second, note that after inverting $t$ in the Poincar\'{e} lemma, we get the exact sequence
\[
0\rightarrow \B_\dR\rightarrow \OO\B_\dR\buildrel\nabla\over\rightarrow \ldots \ ,
\]
whose $0$-th graded piece is an exact sequence
\[
0\rightarrow \hat{\OO}_X\rightarrow \gr^0 \OO\B_\dR\rightarrow \ldots\ ,
\]
giving in particular an injection $\nu_\ast \hat{\OO}_X\rightarrow \nu_\ast \gr^0 \OO\B_\dR$. But we know that $\mathcal{O}_{X_\et}$ maps isomorphically into $\nu_\ast \gr^0 \OO\B_\dR$.

Similarly, we have a long exact sequence
\[
0\rightarrow \hat{\OO}_X(n)\rightarrow \gr^n \OO\B_\dR\rightarrow \gr^{n-1} \OO\B_\dR\otimes_{\mathcal{O}_X} \Omega_X^1\rightarrow \ldots\ ,
\]
which shows that for $n\geq 1$, $\nu_\ast \hat{\OO}_X(n) = 0$ and for $n\geq 2$, $R^1\nu_\ast \hat{\OO}_X(n)=0$, whereas for $n=1$ we get an isomorphism $R^1\nu_\ast \hat{\OO}_X(1)\cong \Omega_{X_\et}^1$. One directly checks that it is the boundary map in Faltings's extension.
\end{proof}

\section{Filtered modules with integrable connection}\label{DeRhamSection}

Let $X$ be a smooth adic space over $\Spa(k,\OO_k)$, with $k$ a discretely valued complete nonarchimedean extension of $\Q_p$ with perfect residue field $\kappa$.

\begin{definition}
\begin{altenumerate}
\item[{\rm (i)}] A $\B_\dR^+$-local system is a sheaf of $\B_\dR^+$-modules $\mathbb{M}$ that is locally on $X_\proet$ free of finite rank.
\item[{\rm (ii)}] An $\OO\B_\dR^+$-module with integrable connection is a sheaf of $\OO\B_\dR^+$-modules $\mathcal{M}$ that is locally on $X_\proet$ free of finite rank, together with an integrable connection $\nabla_{\mathcal{M}}: \mathcal{M}\rightarrow \mathcal{M}\otimes_{\OO_X} \Omega_X^1$, satisfying the Leibniz rule with respect to the derivation $\nabla$ on $\OO\B_\dR^+$.
\end{altenumerate}
\end{definition}

\begin{thm}\label{DeepEquivCat} The functor $\mathbb{M}\mapsto (\mathcal{M},\nabla_{\mathcal{M}})$ given by $\mathcal{M} = \mathbb{M}\otimes_{\B_\dR^+} \OO\B_\dR^+$, $\nabla_{\mathcal{M}} = \id\otimes \nabla$ induces an equivalence of categories between the category of $\B_\dR^+$-local systems and the category of $\OO\B_\dR^+$-modules with integrable connection. The inverse functor is given by $\mathbb{M}=\mathcal{M}^{\nabla_{\mathcal{M}}=0}$.
\end{thm}

\begin{proof} It is obvious that one composition is the identity. One needs to check that any $\OO\B_\dR^+$-module with integrable connection admits enough horizontal sections. This can be checked locally, i.e. in the case $X$ \'{e}tale over $\mathbb{T}^n$. Then it follows from Proposition \ref{DescrBdR} and the fact for any $\Q$-algebra $R$, any module with integrable connection over $R[[X_1,\ldots,X_n]]$ has enough horizontal sections.
\end{proof}

We want to compare those with more classical objects. We have the following lemma:

\begin{lem} Let $X_\an$ be the site of open subsets of $X$. Then the following categories are naturally equivalent:
\begin{altenumerate}
\item[{\rm (i)}] The category of $\OO_{X_\an}$-modules $\mathcal{E}_\an$ over $X_\an$ that are locally on $X_\an$ free of finite rank.
\item[{\rm (ii)}] The category of $\OO_{X_\et}$-modules $\mathcal{E}_\et$ over $X_\et$ that are locally on $X_\et$ free of finite rank.
\item[{\rm (iii)}] The category of $\OO_X$-modules $\mathcal{E}$ over $X_\proet$ that are locally on $X_\proet$ free of finite rank.
\end{altenumerate}
\end{lem}

\begin{proof} We have the morphisms of sites $\nu: X_\proet\rightarrow X_\et$, $\lambda: X_\et\rightarrow X_\an$. We know that $\OO_{X_\an}\cong \lambda_\ast \OO_{X_\et}$ and  $\OO_{X_\et}\cong \nu_\ast \OO_X$. This implies that the pullback functors are fully faithful.

To see that pullback from the analytic to the \'{e}tale topology is essentially surjective, we have to see that the stack in the analytic topology sending some $X$ to the category of locally free sheaves for the analytic topology, is also a stack for the \'{e}tale topology. It suffices to check for finite \'{e}tale covers $Y\rightarrow X$ by the proof of \cite{deJongvanderPut}, Proposition 3.2.2. Moreover, we can assume that $X=\Spa(R,R^+)$ is affinoid, hence so is $Y=\Spa(S,S^+)$. In that case, a locally free sheaf for the analytic topology is equivalent to a projective module over $R$ of finite rank. But the map $R\rightarrow S$ is faithfully flat, hence usual descent works; note that the fibre product $Y\times_X Y$ has global sections $S\otimes_R S$ etc. .

Similarly, if $\mathcal{E}$ on $X_\proet$ becomes trivial on some $U\in X_\proet$, then write $U$ has an inverse limit of finite \'{e}tale surjective maps $U_i\rightarrow U_0$, $U_0\in X_\et$. We assume again that $U_0=\Spa(R,R^+)$ is affinoid, hence so are all $U_i=\Spa(R_i,R_i^+)$. Then $\OO_X(U)$ is the direct limit $R_\infty$ of all $R_i$, which is faithfully flat over $R$. Applying descent for this morphism of rings shows that $\mathcal{E}$ descends to a projective module of finite rank over $U_0$; hence after passage to some smaller open subset of $U_0$, it will be free of finite rank.
\end{proof}

\begin{definition} A filtered $\OO_X$-module with integrable connection is a locally free $\OO_X$-module $\mathcal{E}$ on $X$, together with a separated and exhaustive decreasing filtration $\Fil^i \mathcal{E}$, $i\in \mathbb{Z}$, by locally direct summands, and an integrable connection $\nabla$ satisfying Griffiths transversality with respect to the filtration.
\end{definition}

\begin{definition} We say that an $\OO\B_\dR^+$-module with integrable connection $\mathcal{M}$ and a filtered $\OO_X$-module with integrable connection $\mathcal{E}$ are associated if there is an isomorphism of sheaves on $X_\proet$
\[
\mathcal{M}\otimes_{\OO\B_\dR^+} \OO\B_\dR\cong \mathcal{E}\otimes_{\OO_X} \OO\B_\dR
\]
compatible with filtrations and connections.
\end{definition}

\begin{thm}\begin{altenumerate}
\item[{\rm (i)}] If $\mathcal{M}$ is an $\OO\B_\dR^+$-module with integrable connection and horizontal sections $\mathbb{M}$, which is associated to a filtered $\OO_X$-module with integrable connection $\mathcal{E}$, then
\[
\mathbb{M} = \Fil^0(\mathcal{E}\otimes_{\OO_X} \OO\B_\dR)^{\nabla = 0}\ .
\]
Similarly, one can reconstruct $\mathcal{E}$ with filtration and connection via
\[
\mathcal{E}_\et\cong \nu_{\ast}(\mathbb{M}\otimes_{\B_\dR^+} \OO\B_\dR)\ .
\]
\item[{\rm (ii)}] For any filtered $\OO_X$-module with integrable connection $\mathcal{E}$, the sheaf
\[
\mathbb{M} = \Fil^0(\mathcal{E}\otimes_{\OO_X} \OO\B_\dR)^{\nabla = 0}
\]
is a $\B_\dR^+$-local system such that $\mathcal{E}$ is associated to $\mathcal{M} = \mathbb{M}\otimes_{\B_\dR^+} \OO\B_\dR^+$.
\end{altenumerate}
In particular, the notion of being associated gives rise to a fully faithful functor from filtered $\OO_X$-modules with integrable connection to $\B_\dR^+$-local systems.
\end{thm}

\begin{proof}\begin{altenumerate}
\item[{\rm (i)}] We have
\[
\mathbb{M} = \Fil^0(\mathbb{M}\otimes_{\B_\dR^+} \B_\dR) = \Fil^0(\mathcal{M}\otimes_{\OO\B_\dR^+} \OO\B_\dR)^{\nabla = 0} =  \Fil^0(\mathcal{E}\otimes_{\OO_X} \OO\B_\dR)^{\nabla = 0}\ .
\]
Similarly, lemma \ref{CompDeepStructureSheaf} shows that
\[
\mathcal{E}_\et = \nu_{\ast} (\mathcal{E}\otimes_{\OO_X} \OO\B_\dR) = \nu_{\ast}(\mathcal{M}\otimes_{\OO\B_\dR^+} \OO\B_\dR) = \nu_\ast(\mathbb{M}\otimes_{\B_\dR^+} \OO\B_\dR)\ .
\]
One also recovers the filtration and the connection.

\item[{\rm (ii)}] Let $(\mathcal{E},\nabla,\Fil^\bullet)$ be any filtered $\OO_X$-module with integrable connection. We have to show that there is some $\B_\dR^+$-local system $\mathbb{M}$ associated to $\mathcal{E}$, i.e. such that
\[
\mathcal{E}\otimes_{\OO_X} \OO\B_\dR\cong \mathbb{M}\otimes_{\B_\dR^+} \OO\B_\dR
\]
compatible with filtrations and connection.

We start by constructing some $\B_\dR^+$-local system $\mathbb{M}_0$ such that
\[
\mathcal{E}\otimes_{\OO_X} \OO\B_\dR^+\cong \mathbb{M}_0\otimes_{\B_\dR^+} \OO\B_\dR^+
\]
compatible with the connection (but not necessarily with the filtration). To this end, consider the $\OO\B_\dR^+$-module $\mathcal{M}_0 = \mathcal{E}\otimes_{\OO_X} \OO\B_\dR^+$, with the induced connection $\nabla_{\mathcal{M}_0}$, and take $\mathbb{M}_0 = \mathcal{M}_0^{\nabla_{\mathcal{M}_0}=0}$. The desired isomorphism follows from Theorem \ref{DeepEquivCat}.

Let $n$ be maximal with $\Fil^n \mathcal{E} = \mathcal{E}$ and $m$ minimal with $\Fil^{m+1} \mathcal{E} = 0$. We prove the proposition by induction on $m-n$.

If $n=m$, then we choose $\mathbb{M} = \Fil^{-n} (\mathbb{M}_0\otimes_{\B_\dR^+} \B_\dR)$. This obviously satisfies all conditions.

Now assume that $n<m$, and let $\Fil^{\prime\bullet}$ be the filtration on $\mathcal{E}$ with $\Fil^{\prime k} \mathcal{E} = \Fil^k \mathcal{E}$ unless $k=m$, in which case we set $\Fil^{\prime m} \mathcal{E}=0$. Associated to $(\mathcal{E},\nabla,\Fil^{\prime\bullet})$ we get a $\B_\dR^+$-local system $\mathbb{M}^\prime$, by induction. Twisting everything, we may assume that $m=0$.

We need to show that $\mathbb{M} := \mathrm{Fil}^0 (\mathcal{E}\otimes_{\OO_X} \OO\B_\dR)^{\nabla=0}$ is large. To this end, we consider the following diagram:
\[\xymatrixcolsep{2.7mm}\xymatrix{
&0\ar[d]&0\ar[d]&0\ar[d]&\\
0\ar[r]&\mathbb{M}^\prime\ar[d]\ar[r] & \mathbb{M} \ar[d]\ar[r] & \Fil^0 \mathcal{E}\otimes \hat{\OO}_X\ar[d]\ar[r] & 0\\
0\ar[r]&\Fil^{\prime 0}(\mathcal{E}\otimes \OO\B_\dR)\ar[d]^\nabla \ar[r] & \Fil^0(\mathcal{E}\otimes \OO\B_\dR)\ar[d]^\nabla \ar[r] & \Fil^0 \mathcal{E}\otimes \gr^0 \OO\B_\dR \ar[d]^{\id\otimes \nabla} \ar[r]&0\\
0\ar[r]&\Fil^{\prime -1}(\mathcal{E}\otimes \OO\B_\dR)\otimes \Omega_X^1\ar[d]^\nabla \ar[r] & \Fil^{-1}(\mathcal{E}\otimes \OO\B_\dR)\otimes \Omega_X^1\ar[d]^\nabla \ar[r] & \Fil^0 \mathcal{E}\otimes \gr^{-1} \OO\B_\dR\otimes \Omega_X^1 \ar[d]^{\id\otimes \nabla} \ar[r]&0\\
&\vdots\ar[d]^\nabla&\vdots\ar[d]^\nabla&\vdots\ar[d]^{\id\otimes \nabla}&\\
0\ar[r]&\Fil^{\prime -d}(\mathcal{E}\otimes \OO\B_\dR)\otimes \Omega_X^d\ar[d] \ar[r] & \Fil^{-d}(\mathcal{E}\otimes \OO\B_\dR)\otimes \Omega_X^d\ar[d] \ar[r] & \Fil^0 \mathcal{E}\otimes \gr^{-d} \OO\B_\dR\otimes \Omega_X^d \ar[d] \ar[r]&0\\
&0&0&0&
}\]
Here $d$ is the dimension of $X$ (which we may assume to be connected), and all tensor products are taken over $\OO_X$.

\begin{lem}\begin{altenumerate}
\item[{\rm (i)}] All rows and columns of this diagram are complexes, and the diagram commutes.
\item[{\rm (ii)}] All but the first row are exact.
\item[{\rm (iii)}] The left and right column are exact.
\item[{\rm (iv)}] In the middle column, $\mathbb{M}$ is the kernel of the first map $\nabla$.
\end{altenumerate}
\end{lem}

\begin{proof} Parts (i), (ii) and (iv) are clear: To check that $\nabla$ in the middle column actually is a map compatible with the filtration as claimed, use Griffiths transversality. The left column is isomorphic to $\mathbb{M}^\prime$ tensored with the exact sequence
\[
0\rightarrow \B_\dR^+ \rightarrow \Fil^0 \OO\B_\dR\buildrel\nabla\over\rightarrow \Fil^{-1} \OO\B_\dR\otimes \Omega_X^1\rightarrow \ldots \ .
\]
The right column is $\Fil^0 \mathcal{E}$ tensored with $\gr^0$ of this sequence.
\end{proof}

It follows that the whole diagram is exact, e.g. by considering all but the first row of this diagram as a short exact sequence of complexes and looking at the associated long exact sequence of cohomology groups.

Tensoring the first row with $\Fil^0 \OO\B_\dR$ over $\B_\dR^+$, which is flat, we get a diagram
\[\xymatrix{
0\ar[r] & \mathbb{M}^\prime\otimes_{\B_\dR^+} \Fil^0 \OO\B_\dR\ar[d]^{\cong}\ar[r] & \mathbb{M}\otimes_{\B_\dR^+} \Fil^0 \OO\B_\dR\ar[d]\ar[r] & \Fil^0 \mathcal{E}\otimes_{\OO_X} \gr^0 \OO\B_\dR\ar[d]^{=}\ar[r] & 0\\
0\ar[r]&\Fil^{\prime 0}(\mathcal{E}\otimes_{\OO_X} \OO\B_\dR)\ar[r] & \Fil^0(\mathcal{E}\otimes_{\OO_X} \OO\B_\dR)\ar[r] & \Fil^0 \mathcal{E}\otimes_{\OO_X} \gr^0 \OO\B_\dR \ar[r]&0
}\]
It follows that the middle vertical is an isomorphism as well, which shows that $\mathbb{M}$ is as desired.
\end{altenumerate}
\end{proof}

Although not necessary for our applications, it may be interesting to investigate the relationship between $\mathbb{M}$ and $\mathcal{E}$ further. We recall the $\B_\dR^+$-local system $\mathbb{M}_0$ from the proof, associated to $(\mathcal{E},\nabla)$ with the trivial filtration. It comes with an isomorphism
\[
\mathcal{E}\otimes_{\OO_X} \OO\B_\dR^+\cong \mathbb{M}_0\otimes_{\B_\dR^+} \OO\B_\dR^+\ .
\]
In particular, reducing the isomorphism modulo $\ker \theta$, we get an isomorphism
\[
\mathcal{E}\otimes_{\OO_X} \hat{\OO}_X\cong \gr^0 \mathbb{M}_0\ .
\]
Now the short exact sequence
\[
0\rightarrow \mathcal{E}\otimes_{\OO_X} \hat{\OO}_X(1)\rightarrow \mathbb{M}_0/\Fil^2 \mathbb{M}_0\rightarrow \mathcal{E}\otimes_{\OO_X} \hat{\OO}_X \rightarrow 0
\]
induces via $\nu_\ast$ a boundary map
\[
\mathcal{E}_\et\rightarrow \mathcal{E}_\et\otimes R^1\nu_\ast(\hat{\OO}_X(1))\ .
\]

\begin{lem} Under the canonical isomorphism $R^1\nu_\ast(\hat{\OO}_X(1))\cong \Omega_{X_\et}^1$, this map is identified with $-1$ times the connection on $\mathcal{E}$.
\end{lem}

\begin{proof} Define a sheaf $\mathcal{F}$ via pullback $\mathcal{E}\rightarrow \mathcal{E}\otimes_{\OO_X} \hat{\OO}_X$ as in the diagram
\[\xymatrix{
0\ar[r] & \mathcal{E}\otimes_{\OO_X} \hat{\OO}_X(-1) \ar[r] & \mathbb{M}_0/\Fil^2 \mathbb{M}_0 \ar[r] & \mathcal{E}\otimes_{\OO_X} \hat{\OO}_X\ar[r] & 0\\
0\ar[r] & \mathcal{E}\otimes_{\OO_X} \hat{\OO}_X(-1) \ar[r]\ar[u] & \mathcal{F} \ar[r]\ar[u] & \mathcal{E}\ar[r]\ar[u] & 0
}\]
Then $\mathcal{F}$ admits two maps
\[
\mathcal{F}\rightarrow \mathcal{E}\otimes_{\OO_X} \OO\B_\dR^+/ \Fil^2\cong \mathbb{M}_0\otimes_{\B_\dR^+} \OO\B_\dR^+ / \Fil^2\ :
\]
One via $\mathcal{F}\rightarrow \mathbb{M}_0 / \Fil^2$, the other via $\mathcal{F}\rightarrow \mathcal{E}$. The two maps agree modulo $\Fil^1$, hence their difference gives a map $\mathcal{F}\rightarrow \mathcal{E}\otimes_{\OO_X} \gr^1 \OO\B_\dR^+$. This gives a commutative diagram
\[\xymatrix{
0\ar[r] & \mathcal{E}\otimes_{\OO_X} \hat{\OO}_X(1) \ar[r]\ar[d] & \mathcal{F} \ar[r]\ar[d] & \mathcal{E}\ar[r]\ar[d]^{-\nabla} & 0\\
0\ar[r] & \mathcal{E}\otimes_{\OO_X} \hat{\OO}_X(1) \ar[r] & \mathcal{E}\otimes_{\OO_X} \gr^1 \OO\B_\dR^+ \ar[r] & \mathcal{E}\otimes_{\OO_X} \Omega_X^1\otimes_{\OO_X} \hat{\OO}_X \ar[r] & 0
}\]
Here the lower sequence is Faltings's extension tensored with $\mathcal{E}$. We know that the boundary map of the lower line induces the isomorphism $\mathcal{E}_\et\otimes \Omega^1_{X_\et}\cong \mathcal{E}_\et\otimes R^1\nu_\ast \hat{\OO}_X(1)$, giving the claim.
\end{proof}

\begin{prop}\label{AlternateCharacterization} The $\B_\dR^+$-local system $\mathbb{M}$ associated to $(\mathcal{E},\nabla,\Fil^\bullet)$ is contained in $\mathbb{M}_0\otimes_{\B_\dR^+} \B_\dR$, and it has the property
\[
(\mathbb{M}\cap \Fil^i \mathbb{M}_0)/(\mathbb{M}\cap \Fil^{i+1} \mathbb{M}_0) = \Fil^{-i} \mathcal{E}\otimes_{\OO_X} \hat{\OO}_X(i)\subset \gr^i \mathbb{M}_0\cong \mathcal{E}\otimes_{\OO_X} \hat{\OO}_X(i)
\]
for all $i\in \mathbb{Z}$.

Conversely, there is a unique such $\B_\dR^+$-submodule in $\mathbb{M}_0\otimes_{\B_\dR^+} \B_\dR$.
\end{prop}

\begin{proof} The first assertion is clear. For the other two assertions, we follow the proof of existence of $\mathbb{M}$ and argue by induction. So, let $n$ and $m$ be as above. Again, the case $n=m$ is trivial, so we assume $n<m$ and (by twisting) $m=0$. Moreover, we define the filtration $\Fil^{\prime \bullet}$ as before, and get a unique $\mathbb{M}^\prime$.

Note that $\mathbb{M}$ has the desired interaction with $\mathbb{M}_0$ if and only if there is a commutative diagram as follows:
\[\xymatrix{
0\ar[r] & \mathbb{M}^\prime \ar[r]\ar[d] & \mathbb{M} \ar[r]\ar[d] & \Fil^0 \mathcal{E}\otimes_{\OO_X} \hat{\OO}_X \ar[r]\ar[d] & 0\\
0\ar[r] & \Fil^1 \mathbb{M}_0 \ar[r] & \mathbb{M}_0 \ar[r] & \mathcal{E}\otimes_{\OO_X} \hat{\OO}_X \ar[r] & 0\\
}\]

It is immediate to check that our construction of $\mathbb{M}$ fulfils this requirement. Conversely, it is enough to check that there is a unique sheaf $\mathbb{N}$ that fits into a diagram
\[\xymatrix{
0\ar[r] & \mathbb{M}^\prime \ar[r]\ar[d] & \mathbb{N} \ar[r]\ar[d] & \Fil^0 \mathcal{E} \ar[r]\ar[d] & 0\\
0\ar[r] & \Fil^1 \mathbb{M}_0 \ar[r] & \mathbb{M}_0 \ar[r] & \mathcal{E}\otimes_{\OO_X} \hat{\OO}_X \ar[r] & 0\\
}\]

\begin{lem} We have injections
\[\begin{aligned}
\Ext^1(\Fil^0 \mathcal{E},\mathbb{M}^\prime)&\hookrightarrow \Hom(\Fil^0 \mathcal{E},\Fil^{-1} \mathcal{E}\otimes_{\OO_X} \Omega_X^1)\ ,\\
\Ext^1(\Fil^0 \mathcal{E},\Fil^1 \mathbb{M}_0)&\hookrightarrow \Hom(\Fil^0 \mathcal{E},\mathcal{E}\otimes_{\OO_X} \Omega_X^1)\ ,
\end{aligned}\]
where we calculate the $\Ext^1$ in the category of abelian sheaves on $X_\proet$.
\end{lem}

\begin{proof} Both $\mathbb{M}^\prime$ and $\Fil^1 \mathbb{M}_0$ are successive extensions of sheaves of the form $\mathcal{F}\otimes_{\OO_X} \hat{\OO}_X(k)$ with $k\geq 1$, where $\mathcal{F}$ is some locally free $\OO_X$-module. One readily reduces the lemma to proving that for two locally free $\OO_X$-modules $\mathcal{F}_1$ and $\mathcal{F}_2$, we have
\[
\Ext^1(\mathcal{F}_1,\mathcal{F}_2\otimes_{\OO_X} \hat{\OO}_X(k)) = \Hom(\mathcal{F}_1,\mathcal{F}_2\otimes_{\OO_X} \Omega_X^1)
\]
if $k=1$, and $=0$ if $k\geq 2$. For this, note that
\[
R\Hom(\nu^\ast \mathcal{F}_{1\et},\mathcal{F}_2\otimes_{\OO_X} \hat{\OO}_X(k)) = R\Hom(\mathcal{F}_{1\et}, R\nu_\ast (\mathcal{F}_2\otimes_{\OO_X} \hat{\OO}_X(k)))\ ,
\]
and we know by Corollary \ref{CompDeepStructureSheaf} that the term $R\nu_\ast (\mathcal{F}_2\otimes_{\OO_X} \hat{\OO}_X(k))$ vanishes in degrees $0$ and $1$ if $k\geq 2$, and vanishes in degree $0$ if $k=1$. This gives vanishing if $k\geq 2$, and if $k=1$, we get the desired identification because $R^1\nu_\ast ( \mathcal{F}_2\otimes_{\OO_X} \hat{\OO}_X(k))\cong \mathcal{F}_{2\et}\otimes_{\OO_X} \Omega_X^1$.
\end{proof}

This shows that the pullback of the lower sequence along $\Fil^0 \mathcal{E}\rightarrow \mathcal{E}\otimes_{\mathcal{O}_X} \hat{\mathcal{O}}_X$ comes in at most one way as the pushout from a sequence on the top, giving the desired uniqueness. Let us remark that the existence of this extension is related to Griffiths transversality once again.
\end{proof}

\begin{thm}\label{DeRhamComparison} Let $X$ be a proper smooth adic space over $\Spa(k,\OO_k)$, let $(\mathcal{E},\nabla,\Fil^\bullet)$ be a filtered module with integrable connection, giving rise to a $\B_\dR^+$-local system $\mathbb{M}$, and let $\bar{k}$ be an algebraic closure of $k$, with completion $\hat{\bar{k}}$. Then there is a canonical isomorphism
\[
H^i(X_{\bar{k}},\mathbb{M})\otimes_{B_\dR^+} B_\dR\cong H^i_\dR(X,\mathcal{E})\otimes_k B_\dR\ ,
\]
compatible with filtration and $\Gal(\bar{k}/k)$-action. Here $B_\dR=\B_\dR(\hat{\bar{k}},\OO_{\hat{\bar{k}}})$ is Fontaine's field of $p$-adic periods.

Moreover, there is a $\Gal(\bar{k}/k)$-equivariant isomorphism
\[
H^i(X_{\bar{k}},\gr^0 \mathbb{M})\cong \bigoplus_j H^{i-j,j}_{\mathrm{Hodge}}(X,\mathcal{E})\otimes_k \hat{\bar{k}}(-j)\ ,
\]
where
\[
H^{i-j,j}_{\mathrm{Hodge}}(X,\mathcal{E}) = H^i(X,\gr^j(DR(\mathcal{E})))
\]
denotes the Hodge cohomology in bidegree $(i-j,j)$ of $\mathcal{E}$, using the de Rham-complex $DR(\mathcal{E})$ of $\mathcal{E}$ with its natural filtration.
\end{thm}

\begin{rem} It does not matter whether the de Rham and Hodge cohomology groups are computed on the pro-\'{e}tale, the \'{e}tale, or the analytic site. If $\mathcal{E}=\OO_X$ with trivial filtration and connection, then $\gr^j(DR(\mathcal{E})) = \Omega_X^j[-j]$, and hence
\[
H^{i-j,j}_{\mathrm{Hodge}}(X,\mathcal{E}) = H^{i-j}(X,\Omega_X^j)\ .
\]
\end{rem}

\begin{proof} The de Rham complex of $\mathcal{E}$ is
\[
DR(\mathcal{E}) = ( 0\rightarrow \mathcal{E}\buildrel\nabla\over\rightarrow \mathcal{E}\otimes \Omega_X^1\buildrel\nabla\over\rightarrow\ldots )\ ,
\]
and it is filtered by the subcomplexes
\[
\Fil^m DR(\mathcal{E}) = ( 0\rightarrow \Fil^m \mathcal{E}\buildrel\nabla\over\rightarrow \Fil^{m-1}\mathcal{E}\otimes \Omega_X^1\buildrel\nabla\over\rightarrow\ldots )\ .
\]
On the other hand, one may replace $\mathbb{M}\otimes_{\B_\dR^+} \B_\dR$ by the quasiisomorphic complex
\[
DR(\mathcal{E})\otimes_{\OO_X} \OO\B_\dR = ( 0\rightarrow \mathcal{E}\otimes \OO\B_\dR\buildrel\nabla\over\rightarrow \mathcal{E}\otimes \OO\B_\dR\otimes \Omega_X^1\buildrel\nabla\over\rightarrow \ldots )
\]
with its natural filtration. There is a natural map of filtered complexes,
\[
DR(\mathcal{E})\rightarrow DR(\mathcal{E})\otimes_{\OO_X} \OO\B_\dR\ .
\]
One gets an induced morphism in the filtered derived category
\[
R\Gamma(X_{\bar{k}},DR(\mathcal{E}))\otimes_{\bar{k}} B_\dR \rightarrow R\Gamma(X_{\bar{k}},DR(\mathcal{E})\otimes_{\OO_X} \OO\B_\dR)\ .
\]
We claim that this map is a quasiisomorphism in the filtered derived category. It suffices to check this on gradeds. Further filtering by using the naive filtration of $DR(\mathcal{E})$, one reduces to checking the following statement.

\begin{lem} Let $\mathcal{A}$ be a locally free $\OO_X$-module of finite rank. Then for all $i\in \Z$, the map
\[
R\Gamma(X_{\bar{k}},\mathcal{A})\otimes_{\bar{k}} \gr^i B_\dR \rightarrow R\Gamma(X_{\bar{k}},\mathcal{A}\otimes_{\OO_X} \gr^i \OO\B_\dR)
\]
is a quasiisomorphism.
\end{lem}

\begin{proof} Twisting, one reduces to $i=0$. Then $\gr^0 B_\dR = \hat{\bar{k}}$. The statement becomes the identity
\[
H^i(X_{\bar{k}},\mathcal{A})\otimes_{\bar{k}} \hat{\bar{k}} \cong H^i(X_{\bar{k}},\mathcal{A}\otimes_{\OO_X} \gr^0 \OO\B_\dR)
\]
for all $i\geq 0$. For this, cover $X$ by affinoid open subsets on which $\mathcal{A}$ becomes free; one sees that the left-hand side is $H^i(X_\an,\mathcal{A})\otimes_k \hat{\bar{k}}$, and the right-hand side is $H^i(X_{\hat{\bar{k}},\an},\mathcal{A}_{\hat{\bar{k}}})$, by Proposition \ref{CohomOOBdR} (i). But coherent cohomology of proper adic spaces is finite-dimensional and commutes with extension of the base-field, so we get the desired result.
\end{proof}

Now we get the comparison results. For example, for the Hodge-Tate comparison, note that the $i$-th cohomology of $\gr^0 \mathbb{M}$ is computed by
\[
H^i(X_{\bar{k}},\gr^0(DR(\mathcal{E})\otimes_{\OO_X} \OO\B_\dR))\ ,
\]
which is identified with
\[
\bigoplus_j H^i(X_{\bar{k}},\gr^j(DR(\mathcal{E})))\otimes_{\bar{k}} \gr^{-j} B_\dR = \bigoplus_j H^{i-j,j}_{\mathrm{Hodge}}(X,\mathcal{E})\otimes_k \hat{\bar{k}}(-j)
\]
under the previous quasiisomorphism.
\end{proof}

\section{Applications}\label{ApplicationsSection}

\begin{definition} Let $X$ be a locally noetherian adic space or locally noetherian scheme. A lisse $\Z_p$-sheaf $\mathbb{L}_\bullet$ on $X_\et$ is an inverse system of sheaves of $\Z/p^n$-modules $\mathbb{L}_n$ on $X_\et$, such that each $\mathbb{L}_n$ is locally on $X_\et$ a constant sheaf associated to a finitely generated $\Z/p^n$-module, and such that this inverse system is isomorphic in the pro-category to an inverse system for which $\mathbb{L}_{n+1}/p^n\cong \mathbb{L}_n$.

Let $\hat{\Z}_p=\varprojlim \Z/p^n$ as sheaves on $X_\proet$. Then a lisse $\hat{\Z}_p$-sheaf on $X_\proet$ is a sheaf $\mathbb{L}$ of $\hat{\Z}_p$-modules on $X_\proet$, such that locally in $X_\proet$, $\mathbb{L}$ is isomorphic to $\hat{\Z}_p\otimes_{\Z_p} M$, where $M$ is a finitely generated $\Z_p$-module.
\end{definition}

Using Theorem \ref{TheoremKPi1} and Lemma \ref{InverseLimitExact}, one immediately verifies the following proposition.

\begin{prop} Let $X$ be a locally noetherian adic space over $\Spa(\Q_p,\Z_p)$, and let $\mathbb{L}_\bullet$ be a lisse $\Z_p$-sheaf on $X_\et$. Then $\mathbb{L} = \widehat{\mathbb{L}_\bullet} = \varprojlim \nu^\ast \mathbb{L}_n$ is a lisse sheaf of $\hat{\Z}_p$-modules on $X_\proet$. This functor is an equivalence of categories. Moreover, $R^j \varprojlim \nu^\ast \mathbb{L}_n = 0$ for $j>0$.

In particular, $\mathbb{L}\otimes_{\hat{\Z}_p} \hat{\Q}_p$ is a $\hat{\Q}_p=\hat{\Z}_p[p^{-1}]$-local system on $X_\proet$. $\hfill \Box$
\end{prop}

Let $f: X\rightarrow Y$ be a proper smooth morphism of locally noetherian adic spaces or locally noetherian schemes, and let $\mathbb{L}_\bullet$ be a lisse $\Z_p$-sheaf on $X_\et$. If $f$ is a morphism of schemes of characteristic prime to $p$, then the inverse system $R^if_{\et\ast} \mathbb{L}_\bullet$ of the $R^if_{\et\ast} \mathbb{L}_n$ is a lisse $\Z_p$-sheaf on $Y_\et$.\footnote{Even if the $\mathbb{L}_n$ satisfy $\mathbb{L}_{n+1}/p^n\cong \mathbb{L}_n$, this may not be true for the $R^if_{\et\ast} \mathbb{L}_n$.} Moreover, as the higher $R^j \varprojlim \nu^\ast \mathbb{L}_n$ vanish, we have
\[
\widehat{R^if_{\et\ast} \mathbb{L}_\bullet} = R^if_{\proet\ast} \widehat{\mathbb{L}_\bullet}\ .
\]

\begin{definition} Let $k$ be a discretely valued complete nonarchimedean extension of $\Q_p$ with perfect residue field $\kappa$ and ring of integers $\OO_k$, and let $X$ be a proper smooth adic space over $\Spa(k,\OO_k)$. A lisse $\hat{\Z}_p$-sheaf $\mathbb{L}$ is said to be de Rham if the associated $\B_\dR^+$-local system $\mathbb{M} = \mathbb{L}\otimes_{\hat{\Z}_p} \B_\dR^+$ is associated to some filtered module with integrable connection $(\mathcal{E},\nabla,\Fil^\bullet)$.
\end{definition}

In the following, we write $A_\inf$, $B_\inf$, etc., for the 'absolute' period rings as defined by Fontaine.

\begin{thm}\label{CompAbsolute} Let $k$ be a discretely valued complete nonarchimedean extension of $\Q_p$ with perfect residue field $\kappa$, ring of integers $\OO_k$, and algebraic closure $\bar{k}$, and let $X$ be a proper smooth adic space over $\Spa(k,\OO_k)$. For any lisse $\hat{\Z}_p$-sheaf $\mathbb{L}$ on $X_\proet$ with associated $\B_\dR^+$-local system $\mathbb{M} = \mathbb{L}\otimes_{\hat{\Z}_p} \B_\dR^+$, we have a $\Gal(\bar{k}/k)$-equivariant isomorphism
\[
H^i(X_{\bar{k}},\mathbb{L})\otimes_{\Z_p} B_\dR^+\cong H^i(X_{\bar{k}},\mathbb{M})\ .
\]
If $\mathbb{L}$ is de Rham, with associated filtered module with integrable connection $(\mathcal{E},\nabla,\Fil^\bullet)$, then the Hodge-de Rham spectral sequence
\[
H^{i-j,j}_{\mathrm{Hodge}}(X,\mathcal{E})\Rightarrow H^i_\dR(X,\mathcal{E})
\]
degenerates, and there is a $\Gal(\bar{k}/k)$-equivariant isomorphism
\[
H^i(X_{\bar{k}},\mathbb{L})\otimes_{\Z_p} B_\dR\cong H^i_\dR(X,\mathcal{E})\otimes_k B_\dR
\]
preserving filtrations. In particular, there is also a $\Gal(\bar{k}/k)$-equivariant isomorphism
\[
H^i(X_{\bar{k}},\mathbb{L})\otimes_{\Z_p} \hat{\bar{k}}\cong \bigoplus_j H^{i-j,j}_{\mathrm{Hodge}}(X,\mathcal{E})\otimes_k \hat{\bar{k}}(-j)\ .
\]
\end{thm}

\begin{proof} First, note that for any $n$, $H^i(X_{\bar{k}},\mathbb{L}_n)$ is a finitely generated $\Z/p^n$-module, and we have an almost isomorphism
\[
H^i(X_{\bar{k}},\mathbb{L}_n)\otimes_{\Z_p} A_\inf^a\cong H^i(X_{\bar{k}},\mathbb{L}_n\otimes_{\hat{\Z}_p} \A_\inf^a)\ .
\]
This follows inductively from Theorem \ref{Finiteness} (using Proposition \ref{ProetChangeBase}): For $n=1$, the desired statement was already proved in the proof of Theorem \ref{Finiteness}. Now the sheaves $\mathcal{F}_n = \mathbb{L}_n\otimes_{\hat{\Z}_p} \A_\inf^a$ satisfy the hypotheses of the almost version of Lemma \ref{InverseLimitExact}. Therefore we may pass to the inverse limit $\varprojlim \mathcal{F}_n = \mathbb{L}\otimes_{\hat{\Z}_p} \A_\inf^a$, and get almost isomorphisms
\[
H^i(X_{\bar{k}},\mathbb{L})\otimes_{\Z_p} A_\inf^a\cong H^i(X_{\bar{k}},\mathbb{L}\otimes_{\hat{\Z}_p} \A_\inf^a)\ .
\]
Now we invert $p$ and get almost isomorphisms
\[
H^i(X_{\bar{k}},\mathbb{L})\otimes_{\Z_p} B_\inf^a\cong H^i(X_{\bar{k}},\mathbb{L}\otimes_{\hat{\Z}_p} \B_\inf^a)\ .
\]
Multiplication by $\xi^k$ (using $\xi$ as in Lemma \ref{ExistenceXi}) then shows that
\[
H^i(X_{\bar{k}},\mathbb{L})\otimes_{\Z_p} B_\inf/(\ker \theta)^k\cong H^i(X_{\bar{k}},\mathbb{L}\otimes_{\hat{\Z}_p} \B_\inf/(\ker \theta)^k)\ ,
\]
as the ideal defining the almost setting becomes invertible in $B_\inf/(\ker \theta)^k$. Again, the sheaves $\mathcal{F}_k = \mathbb{L}\otimes_{\hat{\Z}_p} \B_\inf/(\ker \theta)^k$ satisfy the hypothesis of Lemma \ref{InverseLimitExact}, and we deduce that
\[
H^i(X_{\bar{k}},\mathbb{L})\otimes_{\Z_p} B_\dR^+\cong H^i(X_{\bar{k}},\mathbb{L}\otimes_{\hat{\Z}_p} \B_\dR^+)\ ,
\]
as desired.

In particular, $H^i(X_{\bar{k}},\mathbb{M})$ is a free $B_\dR^+$-module of finite rank. This implies that
\[
\dim_{\hat{\bar{k}}} H^i(X_{\bar{k}},\gr^0 \mathbb{M}) = \dim_{B_\dR} (H^i(X_{\bar{k}},\mathbb{M})\otimes_{B_\dR^+} B_\dR)\ .
\]
But Theorem \ref{DeRhamComparison} translates this into the equality
\[
\sum_j \dim_k H^{i-j,j}_{\mathrm{Hodge}}(X,\mathcal{E}) = \dim_k H^i_\dR(X,\mathcal{E})\ ,
\]
so that the Hodge-de Rham spectral sequence degenerates. The final statement follows directly from Theorem \ref{DeRhamComparison}.
\end{proof}

Our final application is a relative version of these results. First, we need a relative Poincar\'{e} lemma. Here and in the following, we use subscripts to denote the space giving rise to a period sheaf.

\begin{prop}\label{RelPoincareLemma} Let $f:X\rightarrow Y$ be a smooth morphism of smooth adic spaces over $\Spa(k,\OO_k)$, of relative dimension $d$. By composition with the projection $\Omega^1_X\rightarrow \Omega^1_{X/Y}$, we get the relative derivation $\nabla_{X/Y}: \OO\B_{\dR,X}^+\rightarrow \OO\B_{\dR,X}^+\otimes_{\OO_X} \Omega_{X/Y}^1$, and the following sequence is exact:
\[\begin{aligned}
0&\rightarrow \B_{\dR,X}^+\otimes_{f_{\proet}^{\ast} \B_{\dR,Y}^+} f_{\proet}^{\ast} \OO\B_{\dR,Y}^+\rightarrow\\
&\rightarrow \OO\B_{\dR,X}^+\buildrel{\nabla_{X/Y}}\over\rightarrow \OO\B_{\dR,X}^+\otimes_{\OO_X} \Omega_{X/Y}^1 \buildrel{\nabla_{X/Y}}\over\rightarrow \ldots \buildrel{\nabla_{X/Y}}\over\rightarrow \OO\B_{\dR,X}^+\otimes_{\OO_X} \Omega_{X/Y}^d\rightarrow 0\ .
\end{aligned}\]
It is strict exact with respect to the filtration giving $\Omega^i_{X/Y}$ degree $i$.
\end{prop}

\begin{proof} It is enough to check the assertion in the case $X=\mathbb{T}^n$, $Y=\mathbb{T}^{n-d}$, with the evident projection. There, the explicit description of Proposition \ref{DescrBdR} does the job again.
\end{proof}

\begin{lem}\label{RelCompCoherent} Let $f:X\rightarrow Y$ be a proper smooth morphism of smooth adic spaces over $\Spa(k,\OO_k)$. Let $\mathcal{A}$ be a locally free $\OO_X$-module of finite rank. Then the morphism
\[
(Rf_{\proet\ast} \mathcal{A})\otimes_{\OO_Y} \gr^0 \OO\B_{\dR,Y}\rightarrow Rf_{\proet\ast}(\mathcal{A}\otimes_{\OO_X} \gr^0\OO\B_{\dR,X})
\]
is an isomorphism in the derived category.
\end{lem}

\begin{proof} We need to check that the sheaves agree over any $U$ pro-\'{e}tale over $Y$. Note that this map factors as the composite of a pro-finite \'{e}tale map $U\rightarrow Y^\prime$ and an \'{e}tale map $Y^\prime\rightarrow Y$. Replacing $Y$ by $Y^\prime$, we may assume that $U$ is pro-finite \'{e}tale over $Y$. Moreover, we can assume that $Y$ is affinoid and that there is an \'{e}tale map $Y\rightarrow \mathbb{T}^m$ that factors as a composite of rational embeddings and finite \'{e}tale maps. Let $K$ be the completed algebraic closure of $k$. Let $\tilde{Y}=Y\times_{\mathbb{T}^m} \tilde{\mathbb{T}}_K^m$. One can also assume that $U$ is a pro-finite \'{e}tale cover of $\tilde{Y}$, which we do.

Finally, we choose an open simplicial cover $X_{\bullet}$ of $X$ such that each $X_i$ admits an \'{e}tale map $X_i\rightarrow \mathbb{T}^n$ that factors as a composite of rational embeddings and finite \'{e}tale maps, fitting into a commutative diagram
\[\xymatrix{
X_i \ar[d] \ar[r] &\mathbb{T}^n\ar[d]\\
Y \ar[r] & \mathbb{T}^m
}\]
where the right vertical map is the projection to the first coordinates. Let $\tilde{X}_i=X_i\times_{\mathbb{T}^n} \tilde{\mathbb{T}}_K^n$.

In this situation, we can control everything. The technical ingredients are summarized in the following lemma.

\begin{lem}\label{TechnicalPoints} All completed tensor products are completed tensor products of Banach spaces in the following.
\begin{altenumerate}
\item[{\rm (i)}] Let $W_i = U\times_Y X_i$ and $\tilde{W}_i = U\times_{\tilde{Y}} \tilde{X}_i$. Then $\tilde{W}_i$ is pro-finite \'{e}tale over $\tilde{X}_i$,
\[
\hat{\mathcal{O}}_X(W_i) = \hat{\mathcal{O}}_Y(U)\hat{\otimes}_{\mathcal{O}_Y(Y)} \mathcal{O}_X(X_i)\ ,
\]
and
\[
\hat{\mathcal{O}}_X(\tilde{W}_i) = \hat{\mathcal{O}}_X(W_i)\hat{\otimes}_{K\langle T_{m+1}^{\pm 1},\ldots,T_n^{\pm 1}\rangle} K\langle T_{m+1}^{\pm 1/p^\infty},\ldots,T_n^{\pm 1/p^\infty}\rangle\ .
\]
\item[{\rm (ii)}] The ring $\hat{\mathcal{O}}_Y^+(U)$ is flat over $\mathcal{O}_Y^+(Y)$ up to a bounded $p$-power, i.e. there is some integer $N$ such that for all $\mathcal{O}_Y^+(Y)$-modules $M$, the group $\Tor^1_{\mathcal{O}_Y^+(Y)}(M,\hat{\mathcal{O}}_Y^+(U))$ is annihilated by $p^N$.
\item[{\rm (iii)}] In the complex
\[
C: 0\rightarrow \mathcal{F}(X_1)\rightarrow \mathcal{F}(X_2)\rightarrow \ldots
\]
associated to the simplicial covering $X_{\bullet}$ of $X$, computing $(Rf_{\proet\ast} \mathcal{F})(Y)$, all boundary maps have closed image.
\item[{\rm (iv)}] We have
\[
(\gr^0\OO\B_{\dR,X})(W_i) = \hat{\mathcal{O}}_X(W_i)[V_1,\ldots,V_m]\ ,
\]
where $V_a$ is the image of $t^{-1}\log([T_a^\flat]/T_a)$ in $\gr^0\OO\B_{\dR,X}$. Moreover,
\[
H^q(W_i,\gr^0\OO\B_{\dR,X})=0
\]
for $q>0$.
\end{altenumerate}
\end{lem}

\begin{proof}
\begin{altenumerate}
\item[{\rm (i)}] First, we use Lemma \ref{RoughLocalStructure} to get
\[
\hat{\OO}_Y(\tilde{Y}) = \OO_Y(Y)\hat{\otimes}_{k\langle T_1^{\pm 1},\ldots,T_m^{\pm 1}\rangle} K\langle T_1^{\pm 1/p^\infty},\ldots,T_m^{\pm 1/p^\infty}\rangle
\]
and
\[
\hat{\OO}_X(\tilde{X}_i) = \OO_X(X_i)\hat{\otimes}_{k\langle T_1^{\pm 1},\ldots,T_n^{\pm 1}\rangle} K\langle T_1^{\pm 1/p^\infty},\ldots,T_n^{\pm 1/p^\infty}\rangle\ .
\]
We may rewrite the latter as
\[\begin{aligned}
(\OO_X(X_i)&\hat{\otimes}_{k\langle T_1^{\pm 1},\ldots,T_m^{\pm 1}\rangle} K\langle T_1^{\pm 1/p^\infty},\ldots,T_m^{\pm 1/p^\infty}\rangle)\hat{\otimes}_{K\langle T_{m+1}^{\pm 1},\ldots,T_n^{\pm 1}\rangle} K\langle T_{m+1}^{\pm 1/p^\infty},\ldots,T_n^{\pm 1/p^\infty} \rangle\\
&= (\OO_X(X_i)\hat{\otimes}_{\OO_Y(Y)} \hat{\OO}_Y(\tilde{Y}))\hat{\otimes}_{K\langle T_{m+1}^{\pm 1},\ldots,T_n^{\pm 1}\rangle} K\langle T_{m+1}^{\pm 1/p^\infty},\ldots,T_n^{\pm 1/p^\infty} \rangle\ .
\end{aligned}\]
Next, using that $U\rightarrow \tilde{Y}$ is pro-finite \'{e}tale and that $\tilde{Y}$ and $\tilde{X}_i$ are perfectoid, we get
\[
\hat{\OO}_X^{+a}(\tilde{W}_i) = \hat{\OO}_Y^{+a}(U)\hat{\otimes}_{\hat{\OO}_Y^{+a}(\tilde{Y})} \hat{\OO}_X^{+a}(\tilde{X}_i)\ ;
\]
in particular,
\[
\hat{\OO}_X(\tilde{W}_i) = \hat{\OO}_Y(U)\hat{\otimes}_{\hat{\OO}_Y(\tilde{Y})} \hat{\OO}_X(\tilde{X}_i)\ .
\]
Using our description of $\hat{\OO}_X(\tilde{X}_i)$, this may be rewritten as
\[
\hat{\OO}_X(\tilde{W}_i) = (\hat{\OO}_Y(U)\hat{\otimes}_{\OO_Y(Y)} \OO_X(X_i))\hat{\otimes}_{K\langle T_{m+1}^{\pm 1},\ldots,T_n^{\pm 1}\rangle} K\langle T_{m+1}^{\pm 1/p^\infty},\ldots,T_n^{\pm 1/p^\infty} \rangle\ .
\]
Now one gets
\[
\hat{\OO}_X(W_i) = \hat{\OO}_Y(U)\hat{\otimes}_{\OO_Y(Y)} \OO_X(X_i)
\]
by taking invariants under $\Z_p^{n-m}$, e.g. by using the computation in Lemma \ref{BasicLocalComputation}.

\item[{\rm (ii)}] Note that $\hat{\mathcal{O}}_Y^+(U)$ is almost flat over $\hat{\mathcal{O}}_Y^+(\tilde{Y})$ by almost purity. But up to a bounded $p$-power, $\hat{\mathcal{O}}_Y^+(\tilde{Y})$ is equal to
\[
\OO_Y^+(Y)\hat{\otimes}_{\OO_k\langle T_1^{\pm 1},\ldots,T_m^{\pm 1}\rangle} \OO_K\langle T_1^{\pm 1/p^{\infty}},\ldots,T_m^{\pm 1/p^\infty}\rangle
\]
by Lemma \ref{RoughLocalStructure}, which is topologically free over $\OO_Y^+(Y)\hat{\otimes}_{\OO_k} \OO_K$, and hence flat over $\OO_Y^+(Y)\hat{\otimes}_{\OO_k} \OO_K$, which in turn is flat over $\OO_Y^+(Y)$.
\item[{\rm (iii)}] This follows from the finiteness of cohomology, proved by Kiehl, \cite{KiehlFiniteness}.
\item[{\rm (iv)}] The desired cohomology groups can be calculated using the Cech cohomology groups of the cover $\tilde{W}_i\rightarrow W_i$. The computation is exactly the same as in the proof of Proposition \ref{CohomOOBdR} (i), starting from the results of part (i).
\end{altenumerate}
\end{proof}

By part (iv), one can compute
\[
(Rf_{\proet\ast} (\mathcal{A}\otimes_{\OO_X} \gr^0\OO\B_{\dR,X}))(U)
\]
by using the complex given by the simplicial covering $U\times_Y X_\bullet = W_\bullet$ of $U\times_Y X$. Moreover, parts (i) and (iv) say that it is given by $(C\hat{\otimes}_{\OO_Y(Y)} \hat{\OO}_Y(U))[V_1,\ldots,V_m]$. But part (ii) says that the operation $\hat{\otimes}_{\OO_Y(Y)} \hat{\OO}_Y(U)$ is exact on strictly exact sequences of Banach-$\OO_Y(Y)$-modules, with part (iii) confirming that $C$ has the required properties implying that it commutes with taking cohomology, so that
\[
(R^if_{\proet\ast} (\mathcal{A}\otimes_{\OO_X} \gr^0\OO\B_{\dR,X}))(U) = ((R^if_{\proet\ast} \mathcal{A})(Y)\hat{\otimes}_{\OO_Y(Y)} \hat{\OO}_Y(U))[V_1,\ldots,V_m]\ .
\]
As $(R^if_{\proet\ast} \mathcal{A})(Y)$ is a coherent $\OO_Y(Y)$-module, one can replace $\hat{\otimes}_{\OO_Y(Y)} \hat{\OO}_Y(U)$ by $\otimes_{\OO_Y(Y)} \hat{\OO}_Y(U)$. Also, by Proposition \ref{CohomOOBdR} (i),
\[
\hat{\OO}_Y(U)[V_1,\ldots,V_m] = \gr^0 \OO\B_{\dR,Y}(U)\ .
\]
Finally, Corollary \ref{CompProetVSEt} (ii) and Proposition \ref{CoherentEtale} (ii) imply that
\[
(R^if_{\proet\ast} \mathcal{A})(U) = (R^if_{\proet\ast} \mathcal{A})(Y)\otimes_{\OO_Y(Y)} \OO_Y(U)\ ,
\]
so we get
\[
(R^if_{\proet\ast} (\mathcal{A}\otimes_{\OO_X} \gr^0\OO\B_{\dR,X}))(U) = (R^if_{\proet\ast} \mathcal{A})(U)\otimes_{\OO_Y(U)} \gr^0 \OO\B_{\dR,Y}(U)\ ,
\]
as desired.
\end{proof}

\begin{thm}\label{CompRelative} Let $f: X\rightarrow Y$ be a proper smooth morphism of smooth adic spaces over $\Spa(k,\OO_k)$. Let $\mathbb{L}$ be a lisse $\hat{\Z}_p$-sheaf on $X_\proet$, and let
\[
\mathbb{M} = \mathbb{L}\otimes_{\hat{\Z}_p} \B_{\dR,X}^+
\]
be the associated $\B_{\dR,X}^+$-local system. Assume that $R^if_{\proet\ast} \mathbb{L}$ is a lisse $\hat{\Z}_p$-sheaf on $Y_\proet$. Then:
\begin{altenumerate}
\item[{\rm (i)}] There is a canonical isomorphism
\[
R^if_{\proet\ast} \mathbb{M}\cong R^if_{\proet\ast} \mathbb{L}\otimes_{\hat{\Z}_p} \B_{\dR,Y}^+\ .
\]
In particular, $R^if_{\proet\ast} \mathbb{M}$ is a $\B_{\dR,Y}^+$-local system on $Y$, which is associated to $R^if_{\proet\ast}\mathbb{L}$.
\item[{\rm (ii)}] Assume that $\mathbb{L}$ is de Rham, and let $(\mathcal{E},\nabla,\Fil^\bullet)$ be the associated filtered $\OO_X$-module with integrable connection. Then the relative Hodge cohomology $R^{i-j,j}f_{\mathrm{Hodge}\ast}(\mathcal{E})$ is a locally free $\OO_Y$-module of finite rank for all $i,j$, the relative Hodge-de Rham spectral sequence
\[
R^{i-j,j}f_{\mathrm{Hodge}\ast}(\mathcal{E})\Rightarrow R^if_{\dR\ast}(\mathcal{E})
\]
degenerates, and $R^if_{\proet\ast} \mathbb{L}$ is de Rham, with associated filtered $\OO_Y$-module with integrable connection given by $R^if_{\dR\ast}(\mathcal{E})$.
\end{altenumerate}
\end{thm}

\begin{rem} By Theorem \ref{EtaleCohomAlgAnal}, the assumption is satisfied whenever $f: X\rightarrow Y$ and $\mathbb{L}$ come as the analytification of corresponding algebraic objects.
\end{rem}

\begin{proof}\begin{altenumerate}
\item[{\rm (i)}] Let $K$ be the completed algebraic closure of $k$. It suffices to check that one gets a canonical isomorphism on $Y_\proet / Y_K$. We start with the isomorphism
\[
(R^if_{\et\ast} \mathbb{L}^\prime)\otimes \OO_Y^{+a}/p\cong R^if_{\et\ast}(\mathbb{L}^\prime\otimes \OO_X^{+a}/p)
\]
from Corollary \ref{RelPrimComp}, for any $\mathbb{F}_p$-local system $\mathbb{L}^\prime$ on $X_\et$. By Corollary \ref{CompProetVSEt} (ii), we may replace $f_\et$ by $f_\proet$. Also, choose $\pi\in \OO_{K^\flat}$ with $\pi^\sharp = p$. Then we get
\[
(R^if_{\et\ast} \mathbb{L}^\prime)\otimes \OO_{Y^\flat}^{+a}/\pi\cong R^if_{\proet\ast}(\mathbb{L}^\prime\otimes \OO_{X^\flat}^{+a}/\pi)\ ,
\]
and by induction on $m$, also
\[
(R^if_{\et\ast} \mathbb{L}^\prime)\otimes \OO_{Y^\flat}^{+a}/\pi^m\cong R^if_{\proet\ast}(\mathbb{L}^\prime\otimes \OO_{X^\flat}^{+a}/\pi^m)\ .
\]
It is easy to see that this implies that for all $m$ and $n$, we have
\[
(R^if_{\et\ast} \mathbb{L}_n)\otimes_{\hat{\Z}_p} \A_{\inf,Y}^a/[\pi]^m\cong R^if_{\proet\ast}(\mathbb{L}_n\otimes_{\hat{\Z}_p} \A_{\inf,X}^a/[\pi]^m)\ .
\]
By assumption, all $R^if_{\et\ast} \mathbb{L}_n$ are locally on $Y_\et$ isomorphic to constant sheaves associated to finitely generated $\Z/p^n$-modules. This implies using Lemma \ref{InverseLimitExact} that one can take the inverse limit over $m$ to get
\[
(R^if_{\et\ast} \mathbb{L}_n)\otimes_{\hat{\Z}_p} \A_{\inf,Y}^a\cong R^if_{\proet\ast}(\mathbb{L}_n\otimes_{\hat{\Z}_p} \A_{\inf,X}^a)\ .
\]
Similarly, one may now take the inverse limit over $n$ to get
\[
R^if_{\proet\ast} \mathbb{L}\otimes_{\hat{\Z}_p} \A_{\inf,Y}^a\cong R^if_{\proet\ast}(\mathbb{L}\otimes_{\hat{\Z}_p} \A_{\inf,X}^a)\ .
\]
Then all further steps are the same as in the proof of Theorem \ref{CompAbsolute}.

\item[{\rm (ii)}] We follow the proof of Theorem \ref{DeRhamComparison}. Let us denote by $DR(\mathcal{E})$ the relative de Rham complex
\[
DR(\mathcal{E}) = ( 0 \rightarrow \mathcal{E} \buildrel\nabla\over\rightarrow \mathcal{E}\otimes \Omega_{X/Y}^1\buildrel\nabla\over\rightarrow \ldots )\ ,
\]
with its natural filtration. We claim that the map
\[
Rf_{\proet\ast}(DR(\mathcal{E}))\otimes_{\OO_Y} \OO\B_{\dR,Y}\rightarrow Rf_{\proet\ast}(DR(\mathcal{E})\otimes_{\OO_X} \OO\B_{\dR,X})
\]
induces a quasiisomorphism in the filtered derived category of abelian sheaves on $Y_\proet$. As in the proof of Theorem \ref{DeRhamComparison}, this reduces to Lemma \ref{RelCompCoherent}. Moreover, by Proposition \ref{RelPoincareLemma}, the right-hand side is the same as
\[
Rf_{\proet\ast}(\mathbb{M}\otimes_{f_\proet^\ast \B_{\dR,Y}^+} f_\proet^\ast \OO\B_{\dR,Y})\ .
\]
Using that $Rf_{\proet\ast} \mathbb{M}$ is a $\B_{\dR,Y}^+$-local system, this may in turn be rewritten as
\[
(Rf_{\proet\ast} \mathbb{M})\otimes_{\B_{\dR,Y}^+} \OO\B_{\dR,Y}\ :
\]
To check that this gives a quasiisomorphism in the filtered derived category, it suffices to check on associated gradeds, where one gets the identity
\[
Rf_{\proet\ast}(\gr^0\mathbb{M}\otimes_{f_\proet^\ast \hat{\OO}_Y} f_\proet^\ast \gr^i\OO\B_{\dR,Y})\cong (Rf_{\proet\ast} \gr^0\mathbb{M})\otimes_{\hat{\OO}_Y} \gr^i\OO\B_{\dR,Y}\ ,
\]
which follows from the fact that locally on $Y_\proet$, $\gr^i\OO\B_{\dR,Y}$ is isomorphic to $\gr^0\OO\B_{\dR,Y}$, which in turn is a polynomial ring over $\hat{\OO}_Y$, i.e. abstractly an infinite direct sum of copies of $\hat{\OO}_Y$.

Combining these results, we find that
\[
Rf_{\proet\ast}(DR(\mathcal{E}))\otimes_{\OO_Y} \OO\B_{\dR,Y}\cong (Rf_{\proet\ast} \mathbb{M})\otimes_{\B_{\dR,Y}^+} \OO\B_{\dR,Y}
\]
in the filtered derived category. In particular, in degree $0$, the left-hand side gives
\[
\bigoplus_j R^if_{\proet\ast}(\gr^j(DR(\mathcal{E})))\otimes_{\OO_Y} \gr^{-j}\OO\B_{\dR,Y} = \bigoplus_j R^{i-j,j}f_{\mathrm{Hodge}\ast}(\mathcal{E})\otimes_{\OO_Y} \gr^0\OO\B_{\dR,Y}(-j)\ ,
\]
whereas the right-hand side evalutes to
\[
(\gr^0 R^if_{\proet\ast} \mathbb{M})\otimes_{\hat{\OO}_Y} \gr^0 \OO\B_{\dR,Y}\ .
\]
The latter is a sheaf of locally free $\gr^0\OO\B_{\dR,Y}$-modules. As locally on $Y_\proet$, $\gr^0 \OO\B_{\dR,Y}$ is faithfully flat over $\OO_Y$, it follows that $R^{i-j,j}f_{\mathrm{Hodge}\ast}(\mathcal{E})$ is locally free for all $i$, $j$.

Similarly, we find that
\[
R^if_{\dR\ast}(\mathcal{E})\otimes_{\OO_Y} \OO\B_{\dR,Y}\cong (R^if_{\proet\ast} \mathbb{M})\otimes_{\B_{\dR,Y}^+} \OO\B_{\dR,Y}\ ,
\]
compatibly with filtration and connection. Counting ranks of locally free modules, one gets the desired degeneration result. Using part (i), the last displayed formula now implies that $R^if_{\et\ast} \mathbb{L}$ is de Rham, with associated filtered $\OO_Y$-module with integrable connection given by $R^if_{\dR\ast}(\mathcal{E})$.
\end{altenumerate}
\end{proof}

\section{Miscellany}\label{Miscellany}

In this section, we recall some facts that are used in the paper. We start with the following situation. Let $K$ be some complete nonarchimedean field, let $A$ be a complete topologically finitely generated Tate algebra over $K$, and let $S_0=\Spec A$, with corresponding adic space $S=\Spa(A,A^\circ)$. Further, let $f_0: X_0 \rightarrow S_0$ be a proper morphism of schemes, and let $f: X \rightarrow S$ be the corresponding morphism of adic spaces.

\begin{thm}\label{RigidGAGA}
\begin{altenumerate}
\item[{\rm (i)}] The category of coherent $\mathcal{O}_{X_0}$-modules is equivalent to that of coherent $\mathcal{O}_X$-modules.
\item[{\rm (ii)}] Let $\mathcal{F}_0$ be a coherent $\mathcal{O}_{X_0}$-module with analytification $\mathcal{F}$ on $X$. Then for all $i\geq 0$, $R^if_{\ast} \mathcal{F}$ is coherent and equal to the analytification of $R^if_{0\ast} \mathcal{F}_0$.
\end{altenumerate}
\end{thm}

\begin{proof} This is the main result of K\"opf's thesis, \cite{Koepf}. For a more modern reference, see Abbes' book on rigid geometry, \cite{AbbesEGR}.
\end{proof}

Further, we need to know that one can also use the \'{e}tale site to compute coherent cohomology. This is summarized in the following proposition.

\begin{prop}\label{CoherentEtale} Let $K$ be a complete nonarchimedean field. All adic spaces are assumed to be locally of finite type over $\Spa(K,\OO_K)$.
\begin{altenumerate}
\item[{\rm (i)}] Let $\mathcal{F}$ be a coherent module on an affinoid adic space $X$. Then the association mapping any affinoid \'{e}tale $U\rightarrow X$ to $\mathcal{O}_U(U)\otimes_{\mathcal{O}_X(X)} \mathcal{F}(X)$ is a sheaf $\mathcal{F}_\et$ on $X_\et$. For $i>0$, the higher cohomology group $H^i(X_\et,\mathcal{F}_\et)=0$ vanishes.
\item[{\rm (ii)}] Let $g: T\rightarrow S$ be an \'{e}tale morphism of affinoid adic spaces. Let $f: X\rightarrow S$ be proper. Let $\mathcal{F}_X$ be a coherent $\mathcal{O}_X$-module, let $Y=X\times_S T$, and let $\mathcal{F}_Y$ be the pullback of $\mathcal{F}_X$ to $Y$. Then for all $i$, we have an isomorphism
\[
H^i(X,\mathcal{F}_X)\otimes_{\OO_S(S)} \OO_T(T)\cong H^i(Y,\mathcal{F}_Y)\ .
\]
In particular, $(R^if_\ast \mathcal{F})_\et = R^if_{\et\ast} \mathcal{F}_\et$.
\end{altenumerate}
\end{prop}

\begin{proof} Part (i) follows from Proposition 3.2.5 of \cite{deJongvanderPut}. Using Proposition 3.2.2 of \cite{deJongvanderPut}, one easily reduces the assertion in part (ii) to the cases where $T\subset S$ is a rational subset, and where $T\rightarrow S$ is finite \'{e}tale. The first case is dealt with by Kiehl in \cite{KiehlFiniteness}, Satz 3.5. In the other case, choose some open affinoid cover of $X$ and compute $H^i(X,\mathcal{F}_X)$ via the associated Cech complex $C$. Then $H^i(Y,\mathcal{F}_Y)$ is computed by $C\otimes_{\OO_S(S)} \OO_T(T)$, as $\OO_T(T)$ is a finite $\OO_S(S)$-module. But $\OO_T(T)$ is flat over $\OO_S(S)$, so tensoring commutes with taking cohomology, which is what we wanted to prove.
\end{proof}

Let us recall a comparison between algebraic and analytic \'{e}tale cohomology. Let $K$ be some complete nonarchimedean field, let $f_0: X_0\rightarrow S_0$ be a proper smooth morphism of schemes over $K$, let $S$ be an adic space locally of finite type over $\Spa(K,\OO_K)$, and let $S\rightarrow S_0$ be some morphism of locally ringed topological spaces, which induces via base-change a proper smooth morphism of adic spaces $f: X\rightarrow S$.

\begin{thm}\label{EtaleCohomAlgAnal} Let $m$ be an integer which is invertible in $K$. Let $\mathbb{L}_0$ be a $\Z/m\Z$-sheaf on $X_{0\et}$ which is locally on $X_{0\et}$ the constant sheaf associated to a finitely generated $\Z/m\Z$-module. Recall that $R^if_{0\et\ast} \mathbb{L}$ is locally on $S_{0\et}$ the constant sheaf associated to a finitely generated $\Z/m\Z$-module. Let $\mathbb{L}$ be the associated sheaf on $X_{\et}$. Then $R^if_{\et\ast} \mathbb{L}$ is the analytification of $R^if_{0\et\ast} \mathbb{L}_0$.
\end{thm}

\begin{proof} This is Theorem  3.7.2 of \cite{Huber}.
\end{proof}

Finally, we recall some facts about affinoid algebras over algebraically closed nonarchimedean fields. So, let $K$ be an algebraically closed nonarchimedean field, and let $R$ be a topologically finitely generated Tate $K$-algebra, i.e. $R$ is a quotient of $K\langle T_1,\ldots,T_n\rangle$ for some $n$.

\begin{thm}\label{BGRStuff} Assume that $R$ is reduced. Then $R^\circ$ is a topologically finitely generated $\OO_K$-algebra, i.e. there is a surjection $\OO\langle T_1,\ldots,T_n\rangle\rightarrow R^\circ$ for some $n$. Moreover, if $S$ is a finite reduced $R$-algebra, then $S^\circ$ is a finite $R^\circ$-algebra.
\end{thm}

\begin{proof} This follows from \S 6.4.1, Corollary 5, of \cite{BoschGuentzerRemmert}.
\end{proof}

\bibliographystyle{abbrv}
\bibliography{pAdicHodgeTheory}

\def\cprime{$'$}
\begin{thebibliography}{10}

\bibitem{AbbesEGR}
A.~Abbes.
\newblock {\em \'{E}l\'ements de g\'eom\'etrie rigide. {V}olume {I}}, volume
  286 of {\em Progress in Mathematics}.
\newblock Birkh\"auser/Springer Basel AG, Basel, 2010.
\newblock Construction et {\'e}tude g{\'e}om{\'e}trique des espaces rigides.
  With a preface by Michel Raynaud.

\bibitem{AndreattaIovita}
F.~Andreatta and A.~Iovita.
\newblock {C}omparison {I}somorphisms for {F}ormal {S}chemes.
\newblock \\ http://www.mathstat.concordia.ca/faculty/iovita/paper$\_$14.pdf.

\bibitem{Beilinson}
A.~Beilinson.
\newblock $p$-adic periods and derived de {R}ham cohomology.
\newblock http://arxiv.org/abs/1102.1294.

\bibitem{BoschGuentzerRemmert}
S.~Bosch, U.~G{\"u}ntzer, and R.~Remmert.
\newblock {\em Non-{A}rchimedean analysis}, volume 261 of {\em Grundlehren der
  Mathematischen Wissenschaften [Fundamental Principles of Mathematical
  Sciences]}.
\newblock Springer-Verlag, Berlin, 1984.
\newblock A systematic approach to rigid analytic geometry.

\bibitem{BrinonRepresentations}
O.~Brinon.
\newblock Repr\'esentations {$p$}-adiques cristallines et de de {R}ham dans le
  cas relatif.
\newblock {\em M\'em. Soc. Math. Fr. (N.S.)}, (112):vi+159, 2008.

\bibitem{ColmezBanach}
P.~Colmez.
\newblock Espaces de {B}anach de dimension finie.
\newblock {\em J. Inst. Math. Jussieu}, 1(3):331--439, 2002.

\bibitem{deJongFundamentalGroup}
A.~J. de~Jong.
\newblock \'{E}tale fundamental groups of non-{A}rchimedean analytic spaces.
\newblock {\em Compositio Math.}, 97(1-2):89--118, 1995.
\newblock Special issue in honour of Frans Oort.

\bibitem{deJongvanderPut}
J.~de~Jong and M.~van~der Put.
\newblock \'{E}tale cohomology of rigid analytic spaces.
\newblock {\em Doc. Math.}, 1:No. 01, 1--56 (electronic), 1996.

\bibitem{FaltingsAlmostEtale}
G.~Faltings.
\newblock Almost \'etale extensions.
\newblock {\em Ast\'erisque}, (279):185--270, 2002.
\newblock Cohomologies $p$-adiques et applications arithm{\'e}tiques, II.

\bibitem{GabberRamero2}
O.~Gabber and L.~Ramero.
\newblock {F}oundations of almost ring theory.
\newblock \\ http://math.univ-lille1.fr/$\sim$ramero/hodge.pdf.

\bibitem{GabberRamero}
O.~Gabber and L.~Ramero.
\newblock {\em Almost ring theory}, volume 1800 of {\em Lecture Notes in
  Mathematics}.
\newblock Springer-Verlag, Berlin, 2003.

\bibitem{Huber}
R.~Huber.
\newblock {\em \'{E}tale cohomology of rigid analytic varieties and adic
  spaces}.
\newblock Aspects of Mathematics, E30. Friedr. Vieweg \& Sohn, Braunschweig,
  1996.

\bibitem{KiehlFiniteness}
R.~Kiehl.
\newblock Der {E}ndlichkeitssatz f\"ur eigentliche {A}bbildungen in der
  nichtarchimedischen {F}unktionentheorie.
\newblock {\em Invent. Math.}, 2:191--214, 1967.

\bibitem{Koepf}
U.~K{\"o}pf.
\newblock \"{U}ber eigentliche {F}amilien algebraischer {V}ariet\"aten \"uber
  affinoiden {R}\"aumen.
\newblock {\em Schr. Math. Inst. Univ. M\"unster (2)}, (Heft 7):iv+72, 1974.

\bibitem{LuetkebohmertAbeloid}
W.~L{\"u}tkebohmert.
\newblock From {T}ate's elliptic curve to abeloid varieties.
\newblock {\em Pure Appl. Math. Q.}, 5(4, Special Issue: In honor of John Tate.
  Part 1):1385--1427, 2009.

\bibitem{RapoportZinkPeriodDomains}
M.~Rapoport and T.~Zink.
\newblock {\em Period spaces for {$p$}-divisible groups}, volume 141 of {\em
  Annals of Mathematics Studies}.
\newblock Princeton University Press, Princeton, NJ, 1996.

\bibitem{SchneiderLocalSystems}
P.~Schneider.
\newblock The cohomology of local systems on {$p$}-adically uniformized
  varieties.
\newblock {\em Math. Ann.}, 293(4):623--650, 1992.

\bibitem{ScholzePerfectoidSpaces1}
P.~Scholze.
\newblock {P}erfectoid {S}paces.
\newblock 2011.
\newblock arXiv:1111.4914.

\bibitem{TatePDivGroups}
J.~T. Tate.
\newblock {$p-divisible$} {$groups.$}.
\newblock In {\em Proc. {C}onf. {L}ocal {F}ields ({D}riebergen, 1966)}, pages
  158--183. Springer, Berlin, 1967.

\bibitem{Temkin}
M.~Temkin.
\newblock On local properties of non-{A}rchimedean analytic spaces.
\newblock {\em Math. Ann.}, 318(3):585--607, 2000.

\end{thebibliography}

\end{document}